\documentclass[10pt,article]{amsart}
\usepackage[hmargin=30mm, vmargin=25mm, includefoot, twoside]{geometry}
\usepackage{ae} 
\usepackage[T1]{fontenc}
\usepackage[cp1250]{inputenc}
\usepackage{amsmath}
\usepackage{amssymb, amsfonts,amscd,verbatim}
\usepackage{mathtools}
\usepackage{MnSymbol}
\usepackage{tikz}
\usetikzlibrary{positioning}
\usepackage{graphicx}
\usepackage{tikz-cd}
\usepackage{comment}
\usepackage[all]{xy}
\usepackage{tensor}
\usepackage{enumitem}
\usepackage{color}
\usepackage{cite}

\usepackage[normalem]{ulem}
\usepackage{indentfirst}
\usepackage{latexsym}
\input xy
\xyoption{all}

\usepackage[bookmarks=true,
bookmarksnumbered=true, breaklinks=true,
pdfstartview=FitH, hyperfigures=false,
plainpages=false, naturalnames=true,
colorlinks=true,pagebackref=true,
pdfpagelabels]{hyperref}

\usepackage{amsmath}    
\newcommand{\ignore}[1]{ }

\newtheorem{Theorem}{Theorem}[section]
\newtheorem{Corollary}[Theorem]{Corollary}
\newtheorem{Lemma}[Theorem]{Lemma}
\newtheorem{Proposition}[Theorem]{Proposition}

\newtheorem{subcor}{Corollary}[Theorem]

\numberwithin{equation}{section}

\theoremstyle{definition}
\newtheorem{Definition}[Theorem]{Definition}

\theoremstyle{remark}
\newtheorem{Remark}[Theorem]{Remark}
\newtheorem{Example}[Theorem]{Example}

\def\CC{{\mathbb C}}

\def\ZZ{{\mathbb Z}}
\def\NN{{\mathbb N}}

\def\B{\mathfrak{B}}
\def\id{\mathrm{id}}
\def\K{\mathfrak{K}}

\def\G{\mathcal{G}}
\def\Gz{\mathcal{G}^{(0)}}
\def\d{\mathrm{d}}

\def\supp{\mathrm{supp}}

\numberwithin{equation}{section}

\author{Kyle ~ Austin}
\thanks{The first author was supported by the following Israeli Science Foundation grants: ISF-Moked grant 2095/15 and grant No. 522/14.}
\address{Weizmann Institute of Science and Technology}
\email{ksaustin88@gmail.com}

\author{Jiawen  ~ Zhang}
\thanks{The second author is supported by the Sino-British Trust Fellowship by Royal Society, International Exchanges 2017 Cost Share (China) grant EC$\backslash$NSFC$\backslash$170341, and NSFC11871342.}
\address{University of Southampton}
\email{jiawen.zhang@soton.ac.uk }

\title[Limit operator theory for groupoids]%
{Limit operator theory for groupoids}

\date{\today}



\includecomment{comment}
\begin{document}
	
	\maketitle
\begin{abstract}
	We extend the symbol calculus and study the limit operator theory for $\sigma$-compact, \'{e}tale and amenable groupoids, in the Hilbert space case. This approach not only unifies various existing results which include the cases of exact groups and discrete metric spaces with Property A, but also establish new limit operator theories for group/groupoid actions and uniform Roe algebras of groupoids. In the process, we extend a monumental result by Exel, Nistor and Prudhon, showing that the invertibility of an element in the groupoid $C^*$-algebra of a $\sigma$-compact amenable groupoid with a Haar system is equivalent to the invertibility of its images under regular representations.
\end{abstract}

\section{Introduction}

Recall that Atkinson's Theorem says that a bounded operator $T$ on $\ell^p(\ZZ^n)$ is \emph{Fredholm} if and only if it is invertible modulo the compact operators. Thus, no finite portion of its matrix coefficients, thinking of $T$ as a $\ZZ^n$-by-$\ZZ^n$ matrix $(T_{x,y})_{x,y \in \ZZ^n}$, has any impact on the Fredholmness of $T$. This suggests that the Fredholmness of an operator only depends on the asymptotic behaviour of its matrix coefficients. To carry out this idea, we mainly focus on band-dominated operators. Recall that an operator $T$ on $\ell^p(\ZZ^n)$ is a \emph{band operator} if all non-zero entries in its matrix sit within a fixed distance from the diagonal, and that $T$ is a \emph{band-dominated operator} if it is a norm-limit of band operators. To study the asymptotic behaviour of a band-dominated operator $T$, we associate a bounded operator $T_{\omega}$ on $\ell^p(\ZZ^n)$ to each limit point $\omega$ in the Stone-\v{C}ech boundary $\beta \ZZ^n\setminus \ZZ^n$, called the \emph{limit operator of $T$ at $\omega$}. More precisely, given a net $\{g_\alpha\}$ in $\ZZ^n$ converging to $\omega$, the operator $T_{\omega}$ has $(x,y)$-coefficient equal to $\lim_{\alpha \to \omega} T_{(g_\alpha+x,g_\alpha+y)}$, i.e., $T_{\omega}$ is the WOT-limit of the operators $U_{g_\alpha}TU_{g_\alpha^{-1}}$ ``in the $\omega$-direction'', where $U_{g_\alpha}$ denotes the canonical shift operator on $\ell^p(\ZZ^n)$ by $g_\alpha$.

The above discussion leads to the central problem in limit operator theory: to study the Fredholmness of a band-dominated operator on the discrete domain $\ZZ^n$ in terms of its limit operators. 
So far there are many significant contributions in this field by Lange, Rabinovich, Roch, Roe and Silbermann \cite{lange1985noether, rabinovich1998fredholm, rabinovich2012limit, rabinovich2004fredholm} with several recent results by Chandler-Wilde, Lindner, Seidel and others \cite{lindner2006infinite, chandler2011limit, Lindner--Seidel, seidel2014fredholm}. The following fundamental theorem was proved by Rabinovich, Roch, Silbermann in the case of $p\in (1,\infty)$ and later by Lindner filling the gaps for $p\in \{0,1,\infty\}$.

\begin{Theorem}[\cite{rabinovich2012limit, lindner2006infinite}]\label{oldmaintheorem}
	Let $T$ be a band-dominated operator on $\ell^p(\ZZ^n)$, where $p \in \{0\} \cup [1,\infty]$. Then $T$ is Fredholm \emph{if and only if} all the associated limit operators $T_\omega$ for $\omega \in \beta \ZZ^n\setminus \ZZ^n$ are invertible and their inverses are uniformly bounded in norm.
\end{Theorem}

It had been a longstanding question whether the uniform boundedness condition in the above theorem can be removed and, recently, Lindner and Seidel \cite{Lindner--Seidel} provided an affirmative answer. Also notice that in some literature (for example \cite{rabinovich2012limit, chandler2011limit}), Theorem \ref{oldmaintheorem} is stated in a Banach space valued version, making it possible to handle continuous underlying spaces as well.

Roe realised that it is essentially the coarse geometry of $\ZZ^n$ that plays a key role in the limit operator theory. Note that for coarse geometers, the algebra of band-dominated operators is called the \emph{uniform Roe algebra}. In \cite{roe-band-dominated}, he extended the symbol calculus and generalised limit operators in the Hilbert space case to all discrete groups, and proved Theorem \ref{oldmaintheorem} for exact discrete groups (also see \cite{willett2009band}). Following Roe's philosophy, Georgescu \cite{georgescu2011structure} generalised limit operator theory to metric spaces with Property A in the Hilbert space case (also see \cite{georgescu2002crossed, georgescu2006localizations}). Recall that the notion of Property A was introduced by Yu \cite{yu2000coarse}, and is equivalent to exactness in the case of groups. Later on, \v{S}pakula and Willett \cite{Spakula-Willett--limitoperators} studied the general $\ell^p$-case for discrete metric spaces with Property A, and established the limit operator theory for $p \in (1,\infty)$. The remaining extreme cases of $p \in \{0,1,\infty\}$ in the discrete metric setting was recently filled up by the second author \cite{ZhangExtreme}. Notice that in \v{S}pakula, Willett and Zhang's work \cite{Spakula-Willett--limitoperators, ZhangExtreme}, they also showed that the condition of uniform boundedness can be removed, extending Lindler and Seidel's result \cite{Lindner--Seidel} mentioned above.

In this paper, we use the formalism of groupoids to unify the above results in the Hilbert space case. Meanwhile, we are able to derive new limit operator theories including those for the cases of group actions and uniform Roe algebras of groupoids.  In the process, we extend the symbol calculus  using the unit space of the underlying groupoid as a ``spectrum.'' To achieve that, it is natural to consider limit operators collectively, which suggests the language of bundles. Inspired by the work of Roe \cite{roe-band-dominated}, we construct suitable operator fiber spaces whose underlying structures reveal the intrinsic nature of coarse geometry in the theory of limit operators. The following is our main result:

\begin{Theorem}[Theorem \ref{main theorem}]\label{thm: main thm}
	Let $\G$ be a locally compact, $\sigma$-compact and \'{e}tale groupoid with compact unit space $\Gz$, $X$ be an invariant open dense subset in $\Gz$ and $\partial X=\Gz \setminus X$. Suppose the reduction groupoid $\G(\partial X)$ is topologically amenable. Then for any element $T$ in the reduced groupoid $C^*$-algebra $C^*_r(\G)$, the following conditions are equivalent:
	\begin{enumerate}
		\item $T$ is invertible modulo $C^*_r(\G(X))$.
		\item The image of $T$ under the symbol morphism is invertible.
		\item All limit operators $\lambda_\omega(T)$ for $\omega \in \partial X$ are invertible and their inverses are uniformly bounded in norm.
		\item All limit operators $\lambda_\omega(T)$ for $\omega \in \partial X$ are invertible.
	\end{enumerate}
\end{Theorem}

At this point, some readers might wonder why we regard the above as the limit operator theory for groupoids. Let us explain here briefly in the group case, and more details will be provided in Section \ref{application section}. Let $G$ be a discrete group. It is well known that the uniform Roe algebra is isomorphic to the reduced groupoid $C^*$-algebra of the transformation groupoid $\beta G \rtimes G$, whose unit space has a natural decomposition: $\beta G = G \sqcup \partial G$. And it is a classic result that $G$ is exact if and only if $\beta G \rtimes G$ is amenable. Furthermore from definition, the reduced groupoid $C^*$-algebra of $G \rtimes G$ is exactly the algebra of compact operators on $\ell^2(G)$. Hence condition (1) in Theorem \ref{thm: main thm} is nothing but $T$ being Fredholm, and it is not hard to show that limit operators coincide with the ones defined by Roe \cite{roe-band-dominated}. Therefore, Theorem \ref{thm: main thm} recovers the classic limit operator theory for exact groups. Similar explanation for the case of discrete metric spaces with Property A was also provided in \cite[Appendix C]{Spakula-Willett--limitoperators}.

We would also like to mention several related recent works on Fredholmness and groupoids, but with different focus from ours. In \cite{carvalho2017fredholm, CVQ18},  Carvalho, Nistor and Qiao introduced Fredholm Lie groupoids to study the Fredholmness of pseudo-differential operators on open manifolds. And in \cite{Mantoui--Nistor--NumericalRanges, Mantoui--Nistor--Decompositions},  M\u{a}ntoiu and Nistor studied the (essential) spectral theory by general twisted groupoid methods, which is closely related to \cite{beckus--lenz--lindler} and \cite{beckusNittis18} in the case of tail equivalence groupoids of dynamical systems over finitely generated groups. We apologise for the many missing references, but guide the interested readers to the listed literature and the references therein.

There are several novelties about our approach that do not exist in previous literatures. For one, we construct the receptacle for the generalised symbol morphism using the language of operator fiber spaces established in Section \ref{op fibre sp}. Recall that in the case of groups, all limit operators live in the same space (i.e., the uniform Roe algebra of the group), and can be combined into the symbol morphim. Unfortunately, in the groupoid setting of Theorem \ref{thm: main thm}, limit operators may not live in the same space any more, which already occurred in the case of discrete metric spaces \cite{Spakula-Willett--limitoperators}. To overcome, we introduce  operator fibre spaces whose fibres are the uniform Roe algebras coming from uniformly chosen coarse structures (Definition \ref{limit morphism defn}), and construct the symbol as an element in the algebra of sections of that bundle. In a future work, we aim to look at the index formulas for K-theory and KK-theory as investigated in \cite{willett2009band}, especially in the case of a groupoid acting on its Higson fiberwise compacticification (see Example \ref{Higsonfiberwiseexample}).

Another novelty is the omission of the uniform boundedness condition in Theorem \ref{thm: main thm}, generalising Lindler, Seidel's result for $\ZZ^n$, and \v{S}pakula, Willett and the second author's result for discrete metric spaces in the Hilbert space case. Their approaches are more coarse geometric, depending heavily on Property A of the underlying spaces, while ours is quite different and more analytic, depending on an improvement of results of Exel \cite{exel-invertibiles} and  Nistor-Prudhon \cite{Nistor--Prudhon}. Briefly speaking, they   proved that for a second countable and amenable groupoid with a Haar system of measures, an element in the groupoid $C^*$-algebra is invertible if and only if its images under regular representations are invertible. Applying their results directly, we may omit the uniform boundedness condition in the case of second countability. Unfortunately, as we already pointed out,  even the groupoid $\beta \ZZ^n \rtimes \ZZ^n$ associated to $\ZZ^n$, is \emph{not} second countable. Notice, however, that these groupoids are usually $\sigma$-compact (they always are when the group or the metric space is countable), so we strengthen the approximation results (see Theorem \ref{sexyapproximationtheorem}) by the first author and Georgescu \cite{austin-georgescu} and prove the following:

\begin{Theorem}[Theorem \ref{thm:extention of Exel}]\label{thm:Exel's generalisation}
	Let $\G$ be a locally compact, $\sigma$-compact and amenable groupoid with a Haar system, then the element $1+a \in C_r^*(\G)^+$ is invertible if and only if $1+\lambda_x(a)$ is invertible for every $x\in \Gz$, where $\lambda_x$ is the regular representation at $x$.
\end{Theorem}

The above theorem is more than enough for Theorem \ref{thm: main thm}, and we choose to prove in its full generality since this result might be of independent interest. In fact, the property stated in Theorem \ref{thm:Exel's generalisation} was first studied by Roch \cite{roch2003algebras} where he introduced the notion of invertibility sufficient family and strictly norming family. These are also closely related to the notion of exhausting family introduced in \cite{Nistor--Prudhon} and the generalized Effros-Hahn property \cite{ionescu2009generalized}. Interested readers can refer to their literatures and the references therein.

We would also like to highlight the new limit operator theory for uniform Roe algebras of groupoids which we obtain as a consequence of Theorem \ref{thm: main thm} (see Section \ref{groupoid case ex}). Recall that the notion of uniform Roe algebras of groupoids was introduced by Anantharaman-Delaroche \cite{Ananth-Delaroche--ExactGroupoids} in her study of the extactness for groupoids. Also notice that there is another closed related but quite different notion called the Roe $C^*$-algebras of groupoids recently introduced by Tang, Willett and Yao \cite{tang2018roe} in their study of analytic indices of elliptic differential operators on Lie groupoids. Here we only focus on the former notion, and analogous to the classic case, we show in Corollary \ref{cor: groupoid case} that the Fredholmness of an operator in the sense of Hilbert module is equivalent to the invertibilities of all the associated limit operators. The technical part here is to provide a detailed and practical description for the limit operators (Lemma \ref{lem: groupoid limit op}), rather than just regarding them as the images under regular representations. Although unintuitive phenomena might occur (\cite[Example 2.3]{anantharaman2014fibrewise}), we bypass these technical difficulties via the universal property of the Stone-\v{C}ech fibrewise compactification together with  Lemma \ref{lemma: first countable implies fibrewise compact}. 

Finally, let us mention some other potential generalisations which readers might wish to continue. Firstly, in this paper, we only consider the case of $p=2$ while limit operator theories are usually establish for general $p \in \{0\} \cup [1,\infty]$, as well as for other classes of Banach space coefficients. The reason we discard the general case is that groupoid $C^*$-algebras have been developed mostly for the Hilbert space case. However, there should be alternative ways available in the general $p$-case, especially for the uniform Roe algebras of groupoids (our work already suggests what the limit operator theory might be).
Another facet of generalisation is that one could look at other operator algebras which are naturally filtered by groupoids, rather than just the groupoid $C^*$-algebras. These algebras might include $C^*$-dynamics, twisted groupoid $C^*$-algebras, or even algebras of sections of a Fell bundle over a groupoid or a semigroupoid. The motivation comes from the fact that not every $C^*$-algebra is a groupoid $C^*$-algebra \cite{Buss--Sims--Opposite}, however, every Elliott classifiable $C^*$-algebra is the twisted $C^*$-algebra of an amenable groupoid \cite{Li}. In other words, groupoid $C^*$-algebras do encompass a nice class of $C^*$-algebras, but there is still much room for generalisation.

\textbf{Outline.} The paper is organised as follows: In Section \ref{sec: preliminaries}, we recall basic notions related to groupoids that we will use throughout the paper. In Section \ref{op fibre sp}, we introduce the notion of operator fibre spaces which will be used as an engine for the symbol map, and study their topologies and other related properties. Then in Section \ref{main theorem section}, we construct the symbol morphism (Definition \ref{limit morphism defn}), state our limit operator theory for groupoids (Theorem \ref{thm: main thm}), and prove the relatively easier part. Section \ref{sec: extension of Exel} is mainly devoted to the rest of the proof that uniform boundedness condition is unnecessary. We do so by generalising the results of Exel and Nistor-Prudhon into full generality (Theorem \ref{thm:Exel's generalisation}). Finally in Section \ref{application section}, we apply Theorem \ref{thm: main thm} to recover the limit operator theories in the Hilbert space case for exact groups and discrete metric spaces with Property A (Section \ref{group case}-\ref{metric space ex}), and also establish new theories for group actions (Section \ref{group action ex}) and uniform Roe algebras of groupoids (Section \ref{groupoid case ex}).

\emph{Acknowledgments.} We would like to thank Jan \v{S}pakula for suggesting this topic and a lot of illuminating discussions and comments on an early draft of the paper. The first author would like to thank Uri Bader for all of his encouragement and hospitality for the duration of this project. The second author would also like to thank Baojie Jiang for several interesting discussions, and Graham Niblo and Nick Wright for their continuous supports.

\section{Preliminaries}\label{sec: preliminaries}

\subsection{Basic notions}
Let us start by recalling some basic notions and terminology on groupoids. The readers are supposed to be familiar with basic definitions about groupoids, and details can be found in \cite{renault}, or \cite{sims--notesongroupoids} in the \'etale case.

Recall that a \emph{groupoid} is a small category, in which every morphism is invertible. Roughly speaking, a groupoid consists of a set $\G$, a subset $\Gz$ called the \emph{unit space}, two maps $s,r: \G \to \Gz$ called the \emph{source} and \emph{range} maps respectively, a \emph{composition law}
$$ \G^{(2)}:=\{(\gamma_1,\gamma_2) \in \G \times \G: s(\gamma_1)=r(\gamma_2)\}\ni(\gamma_1,\gamma_2) \mapsto \gamma_1\gamma_2 \in \G,$$	and an \emph{inverse} map $\gamma \mapsto \gamma^{-1}$. These operations satisfy a couple of axioms, including associativity law and the fact that elements in $\Gz$ act as units. For $x \in \Gz$, denote $\G^x:=r^{-1}(x)$ and $\G_x:=s^{-1}(x)$. A \emph{groupoid morphism} is a functor.

A \emph{locally compact} groupoid is a groupoid equipped with a locally compact topology such that the structure maps (composition and inverse) are continuous with respect to the induced topologies, and the range map is open (which is equivalent to the source map being open). Note that the latter is a necessary condition for the existence of a Haar system of measures (see Definition \ref{Haarysituation} below).


Let $Y$ be a subset in the unit space $\Gz$ of a groupoid $\G$, and we set $\G(Y):=r^{-1}(Y) \cap s^{-1}(Y)$. Note that $\G(Y)$ is a subgroupoid of $\G$ (in the sense that it is stable under product and inverse), called the \emph{reduction of $\G$ by $Y$}. $Y$ is said to be \emph{invariant} if $r^{-1}(Y)=s^{-1}(Y)$. When $Y$ is an invariant locally compact subset of $\Gz$, it is obvious that the reduction $\G(Y)$ is itself a locally compact groupoid.

From now on, \emph{we always assume locally compact spaces to be Hausdorff}. Given a locally compact space $Z$, denote $C_b(Z)$ the $C^*$-algebra of bounded continuous complex-valued functions on $Z$, $C_0(Z)$ the $C^*$-subalgebra of functions vanishing at infinity, and $C_c(Z)$ those with compact support. The support of a function $f$ is denoted by $\supp(f)$.

\begin{Definition}\label{Haarysituation}
	Let $\G$ be a locally compact groupoid. A \emph{(right) Haar system of measures on $\G$} is a collection $\{\mu_x:x\in \Gz\}$ of positive regular Radon\footnote{By this, we mean that the measure is locally finite and is inner and outer regular with respect to subsets of finite measure.} measures on $\G$ such that:
	\begin{enumerate}
		\item The support of $\mu_x$ is exactly $\G_x$, for all $x \in \Gz$.
		\item For all $f\in C_c(\G)$, the function $x\to \int_{\G}f(\gamma)\d\mu_x(\gamma)$ is continuous on $\Gz$.
		\item For all $\gamma\in \G$ and $f\in C_c(\G)$, the following equality holds:
		$$\int_{\G_{r(\gamma)}}f(\gamma'\gamma)\d\mu_{r(\gamma)}(\gamma') = \int_{\G_{s(\gamma)}}f(\gamma')\d\mu_{s(\gamma)}(\gamma').$$
	\end{enumerate}
\end{Definition}

\begin{Definition}
	Let $\G, \mathcal{H}$ be locally compact groupoids with Haar systems of measures $\{\mu_x: x \in \Gz \}$ and $\{\nu_y: y \in \mathcal{H}^{(0)}\}$, respectively. A groupoid morphism $q: \G \to \mathcal{H}$ is said to be \emph{Haar system preserving} if $q$ is proper, continuous, and satisfies $q_{\ast}\mu_x=\nu_z$ for any $z \in \mathcal{H}^{(0)}$ and $x\in q^{-1}(z)$. Recall that $q_{\ast}\mu_x$ is the \emph{pushforward} of $\mu_x$, i.e., $(q_{\ast}\mu_x)(A):=\mu_x(q^{-1}(A))$ for any measurable set $A$.
\end{Definition}

A locally compact groupoid is called \emph{\'{e}tale} (also called \emph{$r$-discrete}) if the range (hence the source) map is a local homeomorphism. Clearly in this case, each fibre $\G^x$ (and $\G_x$) is discrete with the induced topology, and $\Gz$ is clopen in $\G$. The notion of \'{e}taleness for a groupoid can be regarded as an analogue of discreteness in the group case. Note that for a locally compact \'{e}tale groupoid, we may always choose the counting measure on each fibre to form a canonical Haar system of measures.

\subsection{Fibre spaces and fibrewise compactifications}\label{sec:fibre spaces}

Later on, we also need to consider more general notions than groupoids, called fibre spaces. Here we provide a brief introduction.

\begin{Definition}\label{fibre space defn}
	Let $X$ be a locally compact space. A \emph{fibre space over $X$} is a pair $(Y,p)$, where $Y$ is a locally compact space and $p$ is a continuous surjective map from $Y$ onto $X$. For each point $x\in X$, denote $Y^x$ the \emph{fibre} $p^{-1}(x)$. We say $(Y,p)$ is \emph{fibrewise compact} if $p$ is proper, i.e., $p^{-1}(K)$ is compact for any compact $K\subseteq X$.
\end{Definition}

\begin{Definition}
	A \emph{morphism} between two fibre spaces $(Y_1,p_1)$ and $(Y_2,p_2)$ over $X$ is a continuous map $\varphi: Y_1 \to Y_2$ such that $p_1=p_2\circ \varphi$.
\end{Definition}

\begin{Definition}
	A \emph{fibrewise compactification} of a fibre space $(Y,p)$ over $X$ is a fibrewise compact fibre space $(Z,q)$ over $X$ together with a morphism $\iota: Y \rightarrow Z$ such that the image of $Y$ is an open dense subset in $Z$ and $\iota$ is a homeomorphism onto its image. Usually we regard $Y$ as a subset of $Z$ in this case, and the morphism $\iota$ is just the inclusion.
\end{Definition}

To understand fibrewise compactifications in a more precise way, we need to refer to Gelfand spectra. Let $(Y,p)$ be a fibre space over $X$. Denote
$$p^*C_0(X):=\{f\circ p: f\in C_0(X)\},$$
and $C_0(Y,p)$ to the closure of
\begin{equation}\label{CcYp}
C_c(Y,p):= \{g \in C_b(Y): \exists \mbox{~compact~}K \subseteq X \mbox{~such~that~}g(x)=0 \mbox{~for~}x\notin p^{-1}(K)\}.
\end{equation}
Equivalently, a function $g \in C_b(Y)$ belongs to $C_0(Y,p)$ \emph{if and only if} for any $\varepsilon>0$, there exists a compact $K \subseteq X$ such that $|g(x)|<\varepsilon$ for any $x \notin p^{-1}(K)$. And we have the following characterisation:

\begin{Proposition}[Proposition 1.2, \cite{anantharaman2014fibrewise}]\label{fw cmpt}
	For a fibre space $(Y,p)$ over $X$, its fibrewise compactifications are one-to-one correspondence with Gelfand spectra of $C^*$-subalgebras in $C_0(Y,p)$ containing $p^*C_0(X)+C_0(Y)$.
\end{Proposition}

From the above proposition, there are two extreme cases $p^*C_0(X)+C_0(Y)$ and $C_0(Y,p)$. They corresponds to the \emph{fibrewise Alexandroff compactification}, denoted by $(Y^+,p^+)$, and the \emph{fibrewise Stone-\v{C}ech compactification}, denoted by $(\beta_p Y,p_\beta)$. As the classic Stone-\v{C}ech compactifications, we have the following universal property:

\begin{Proposition}[Proposition 1.4, \cite{anantharaman2014fibrewise}]\label{prop: universal property}
	Let $(Y,p)$ and $(Y_1,p_1)$ be two fibre spaces over the same space $X$, and $(Y_1,p_1)$ is fibrewise compact. Then for any morphism $\varphi:Y \to Y_1$, there exists a unique morphism $\Phi: (\beta_p Y,p_\beta) \to (Y_1,p_1)$ extending $\varphi$.
\end{Proposition}

Please be aware of the unintuitive phenomenon noticed by Anantharaman-Delaroche \cite[Example 2.3]{anantharaman2014fibrewise}. Let $(Y,p)$ be a fibre space over $X$, and $(\beta_p Y,p_\beta)$ be its fibrewise Stone-\v{C}ech compactification. Given an $x \in X$, it might occur that the fibre $Y^x$ is \emph{not} dense in the fibre $p_\beta^{-1}(x)$ of $\beta_p Y$.


\subsection{Groupoid actions}\label{groupoid action preliminary subsection}
Now we discuss groupoid actions, and several related notions. We start with the case of groups. Let $G$ be a discrete group acting on some locally compact space $X$ by homeomorphisms. The \emph{transformation groupoid} $X \rtimes G$ is set-theoretically $X \times G$. The groupoid structure is given by $s(x,\gamma)=(\gamma^{-1}x,e)$, $r(x,\gamma)=(x,e)$ and $(x,\gamma)(\gamma^{-1}x,\gamma')=(x,\gamma\gamma')$. The topology on $X \rtimes G$ is nothing but the product topology, and it is clear that the groupoid $X \rtimes G$ is \'{e}tale since the group $G$ is discrete.

Now we introduce groupoid actions and discuss the associated semi-direct products.
\begin{Definition}
	For two fibre spaces $(Y_i,p_i)_{i=1,2}$ over $X$, their \emph{fibred product} $Y_1 \tensor*[_{p_1}]{\ast}{_{p_2}} Y_2$ is defined to be
	$$\{(y_1,y_2) \in Y_1 \times Y_2: p_1(y_1)=p_2(y_2)\}$$
	with the induce topology by the product topology.
\end{Definition}

\begin{Definition}
	Let $\G$ be a locally compact groupoid. A \emph{(left) $\G$-space} is a fibre space $(Y,p)$ over $X=\Gz$, together with a continuous map $(\gamma,y) \mapsto \gamma y$ from $\G \tensor*[_s]{\ast}{_p} Y$ to $Y$, satisfying the following conditions:
	\begin{enumerate}
		\item $p(\gamma y)=r(\gamma)$ for $(\gamma, y) \in \G \tensor*[_s]{\ast}{_p} Y$, and $p(y)y=y$ for $y \in Y$;
		\item $(\gamma_2\gamma_1) y=\gamma_2(\gamma_1 y)$ for $(\gamma_1,y) \in \G \tensor*[_s]{\ast}{_p} Y$ and $s(\gamma_2)=r(\gamma_1)$.
	\end{enumerate}
\end{Definition}

For a $\G$-space $(Y,p)$, now we define the \emph{semi-direct product groupoid} $Y \rtimes \G$. As a topological space, it is $Y \tensor*[_p]{\ast}{_r} \G$. The range of $(y,\gamma)$ is $\mathbf{r}(y,\gamma)=(y,r(\gamma))=(y,p(y))$, and its source is $\mathbf{s}(y,\gamma)=(\gamma^{-1}y, s(\gamma))$. The product is given by
$$(y, \gamma) (\gamma^{-1}y, \gamma_1)=(y, \gamma\gamma_1)$$
and the inverse is given by
$$(y,\gamma)^{-1}=(\gamma^{-1}y, \gamma^{-1}).$$
Clearly, the unit space $(Y \rtimes \G)^{(0)}$ can be identified with $Y$ via the homeomorphism sending $(y,p(y))$ to $y$. Hence from now on, we may regard $Y$ as the unit space of the groupoid $Y \rtimes \G$.

Note that when $\G$ is a group, the semi-direct product groupoid $Y \rtimes \G$ is nothing but the transformation groupoid mentioned above. As shown in \cite[Proposition 1.4, 1.5]{Ananth-Delaroche--ExactGroupoids}, the range map $\mathbf{r}: Y \rtimes \G \to Y$ is always open, and $Y \rtimes \G$ is \'{e}tale when the groupoid $\G$ itself is.

As a trivial example, for a groupoid $\G$, $(\Gz,\id)$ is a fibre space over $\Gz$ itself. And it is obvious that the associated semi-direct product $\Gz \rtimes \G$ is isomorphic to the original groupoid $\G$. Hence there is no difference to study semi-direct product groupoids and general groupoids.

Now we discuss the notion of equivariant fibrewise compactifications associated to a groupoid action. Briefly speaking, it is a special class of fibrewise compactifications which are compatible with groupoid actions. To be more precise, we start with the following notion.

\begin{Definition}
	A morphism $\varphi: (Y,p) \to (Z,q)$ between two $\G$-spaces is said to be \emph{$\G$-equivariant} if $\varphi(\gamma y) = \gamma \varphi(y)$ for any $(\gamma, y) \in \G \tensor*[_s]{\ast}{_p} Y$.
\end{Definition}

\begin{Definition}
	A \emph{$\G$-equivariant fibrewise compactification} of a $\G$-space $(Y,p)$ is a fibrewise compactification $(Z,q)$ of $(Y,p)$ such that the associated morphism $\iota: Y \hookrightarrow Z$ is $\G$-equivariant.
\end{Definition}

As before, we would also like to provide a characterisation in terms of Gelfand spectra. Recall from Proposition \ref{fw cmpt} that fibrewise compactifications of a fibre space $(Y,p)$ can be characterised in terms of certain $C^*$-subalgebras in $C_b(Y)$. In the equivariant case, they can be characterised by certain invariant subalgebras in the following sense.

To simplify the discussion, we only focus on the \'{e}tale case. Let $(Y,p)$ be a $\G$-space where $\G$ is a locally compact and \'{e}tale groupoid. Given any $f \in C_c(\G)$ and $g \in C_b(Y)$, define the \emph{convolution product} $f \ast g$ by:
$$f\ast g(x) = \sum_{\gamma\in r^{-1}(p(x))} f(\gamma)g(\gamma^{-1}x).$$
As shown in \cite{anantharaman2014fibrewise}, $f \ast g \in C_c(Y,p)$ defined in (\ref{CcYp}). A $C^*$-subalgebra $\mathcal{A}$ in $C_b(Y)$ is said to be \emph{stable under convolution of $C_c(\G)$}, if $f \ast g \in \mathcal{A}$ for any $f\in C_c(\G)$ and $g \in \mathcal{A}$. Now we have the following:

\begin{Proposition}[Corollary 2.6, \cite{anantharaman2014fibrewise}]\label{prop:equiv fibrewise cmptf}
	Let $(Y,p)$ be a $\G$-space, where $\G$ is a locally compact and \'{e}tale groupoid. Then there is a one-to-one correspondence between $\G$-equivariant fiberwise compactifications of $Y$ and those subalgebras of $C_0(Y,p)$ that contain $p^*C_0(X)+C_0(Y)$ and are stable under convolution of $C_c(\G)$.
\end{Proposition}

By the above proposition, it is easy to check that for a $\G$-space $(Y,p)$ where $\G$ is a locally compact and \'{e}tale groupoid, the fibrewise Alexandroff compactification $(Y^+,p^+)$ and the fibrewise Stone-\v{C}ech compactification $(\beta_p Y,p_\beta)$ are $\G$-equivariant. As before we still have the following universal property:

\begin{Proposition}[Proposition 2.8, \cite{anantharaman2014fibrewise}]
	Let $\G$ be a locally compact and \'{e}tale groupoid, $(Y,p)$ and $(Y_1,p_1)$ be two $\G$-spaces. Suppose $(Y_1,p_1)$ is fibrewise compact, and $\varphi:Y \to Y_1$ is a $\G$-equivariant morphism. Then the unique morphism $\Phi: (\beta_p Y,p_\beta) \to (Y_1,p_1)$ extending $\varphi$ is also $\G$-equivariant.
\end{Proposition}

As in Proposition \ref{prop:equiv fibrewise cmptf}, the Stone-\v Cech and the Alexandroff fiberwise compactifications are the ``largest'' and the ``smallest'' equivariant fiberwise compactification of a fiber space being acted upon by an \'{e}tale groupoid. Here is an example of yet another natural equivariant fiberwise compactification.

\begin{Example}\label{Higsonfiberwiseexample}
	Let $\G$ be a locally compact and \'{e}tale groupoid. Let $C_h(\G)$ denote the set of bounded continuous functions $f$ on $\G$ such that for every $\epsilon >0$ and every compact subset $K\subset \G$, there exists a compact subset $\hat{K}\subset \G$ such that if $g,h\in \G\setminus \hat{K}$ with $r(g) = r(g)$ and $gh^{-1} \in K$, then we have $|f(g)-f(h)|\leq \epsilon$. Elements in $C_h(\G)$ are called the \emph{Higson functions on $\G$}. The readers can verify that $C_h(\G)$ is a closed subalgebra of $C_0(\G,r)$ that contains $r^*C_0(\Gz)+C_0(\G)$, and is stable under convolution. Hence it determines a $\G$-equivariant fiberwise compactification of $\G$, denote by $h\G$ and called the \emph{Higson fiberwise compactification} of $\G$. It is our intention of further investigating the index theory of a groupoid acting on its Higson fiberwise compactifications.
\end{Example}

\subsection{Groupoid $C^*$-algebras}\label{groupoid C algebra prlm}
Here we recall some $C^*$-algebras associated to a groupoid $\G$ with a (right) Haar system $\{\mu_x: x \in \Gz\}$. First note that the space $C_c(\G)$ is a $\ast$-involutive algebra with respect to the following operations: for $f,g \in C_c(\G)$,
\begin{eqnarray*}
	(f \ast g)(\gamma) &=& \int_{\alpha \in \G_{s(\gamma)}} f(\gamma \alpha^{-1}) g(\alpha) \d \mu_{s(\gamma)}(\alpha), \\
	f^*(\gamma) &=& \overline{f(\gamma^{-1})}.
\end{eqnarray*}
In the special case that $\G$ is \'{e}tale with the Haar system consisting of counting measures on each fibre, the above convolution can be simplified as follows:
\begin{equation*}
(f \ast g)(\gamma) = \sum_{\alpha \in \G_{s(\gamma)}} f(\gamma \alpha^{-1}) g(\alpha).
\end{equation*}
Consider the following algebraic norm on $C_c(\G)$ defined by:
$$\|f\|_I:=\max\Big\{\sup_{x \in \Gz} \int |f| \d \mu_x, \sup_{x \in \Gz} \int |f^*| \d \mu_x\Big\}.$$
The completion of $C_c(\G)$ with respect to the norm $\|\cdot\|_I$ is denoted by $L^1(\G)$.

The \emph{maximal (full) groupoid $C^*$-algebra} $C^*_{\max}(\G)$ is defined to be the completion of $C_c(\G)$ with respect to the norm
$$\|f\|_{\max}:=\sup \|\pi(f)\|,$$
where the supremum is taken over all contractive $\ast$-representations $\pi$ of $L^1(\G)$.

In order to define the reduced counterpart, we recall that for each $x \in \Gz$ the \emph{regular representation at $x$}, denoted by $\lambda_x: C_c(\G) \to \B(L^2(\G_x; \mu_x))$, is defined as follows:
\begin{equation*}
\big(\lambda_x(f)\xi\big)(\gamma):=\int_{\alpha \in \G_x} f(\gamma \alpha^{-1})\xi(\alpha) \d \mu_x(\alpha), \quad \mbox{where}~ f \in C_c(\G)\mbox{~and~}\xi \in L^2(\G_x; \mu_x).
\end{equation*}
Again in the special case that $\G$ is \'{e}tale with the Haar system consisting of counting measures on each fibre, the regular representation $\lambda_x: C_c(\G) \to \B(\ell^2(\G_x))$ can be simplified as follows:
\begin{equation}\label{reduced algebra defn}
\big(\lambda_x(f)\xi\big)(\gamma):=\sum_{\alpha \in \G_x} f(\gamma \alpha^{-1})\xi(\alpha), \quad \mbox{where}~ f \in C_c(\G)\mbox{~and~}\xi \in \ell^2(\G_x).
\end{equation}
It is routine work to check that $\lambda_x$ is a well-defined $\ast$-homomorphism. The \emph{reduced norm} on $C_c(\G)$ is
$$\|f\|_r:=\sup_{x \in \Gz} \|\lambda_x(f)\|,$$
and the \emph{reduced groupoid $C^*$-algebra} $C^*_r(\G)$ is defined to be the completion of the $\ast$-algebra $C_c(\G)$ with respect to this norm. Clearly, each regular representation $\lambda_x$ can be extended to a homomorphism $\lambda_x: C^*_r(\G) \to \B(L^2(\G_x;\mu_x))$ automatically. It is also a routine work to check that there is a canonical surjective homomorphism from $C^*_{\max}(\G)$ to $C^*_r(\G)$.

Now we discuss an alternative way to define the reduced groupoid $C^*$-algebra $C^*_r(\G)$ in terms of the Hilbert module language as follows. (For those who are not familiar with Hilbert modules, we refer to \cite{lance1995hilbert}.) Let $L^2(\G)$ be the Hilbert module over $C_0(\Gz)$ obtained by taking completion of $C_c(\G)$ with respect to the $C_0(\Gz)$-valued inner product
$$\langle \xi,\eta \rangle(x):=\int_{\gamma\in \G_x} \overline{\xi(\gamma)} \eta(\gamma) \d\mu_x(\gamma),$$
and the right $C_0(\Gz)$-module structure is given by
$$(\xi f)(\gamma):=\xi(\gamma) f(s(\gamma)).$$
Denote $\B(L^2(\G))$ the $C^*$-algebra of all adjointable operators on the Hilbert module $L^2(\G)$.

Note that all the regular representations $\lambda_x$ defined in (\ref{reduced algebra defn}) can be put together to a single representation $\Lambda: C_c(\G) \to \B(L^2(\G))$ by the formula:
$$((\Lambda f)\xi)(\gamma):=\int_{\alpha \in \G_{s(\gamma)}} f(\gamma \alpha^{-1})\xi(\alpha) \d\mu_{s(\gamma)}(\alpha)=((\lambda_{s(\gamma)}f)\xi|_{\G_{s(\gamma)}})(\gamma).$$
And it is easy to check that for $f \in C_c(\G)$, we have $\|f\|_r=\|\Lambda(f)\|_{\B(L^2(\G))}$. Therefore, $\Lambda$ can be extended to a faithful representation $\Lambda: C^*_r(\G) \to \B(L^2(\G))$, which is called the \emph{regular representation}.

\subsection{Amenability}
Amenable groupoids comprise a large class of groupoids with relatively nice properties, and they are the central objects of our paper. Literally, they are the analogue of amenable groups in the world of groupoids. However unlike the case of groups, there are different versions of amenability (for example measurewise, topological and Borel amenabilities) which might \emph{not} be equivalent for general groupoids. Fortunately, they behave quite well under the restriction of the \'{e}taleness. Here we mainly focus on topological amenability. A standard reference is \cite{ananth-delaroche-renault} and another reference for just \'{e}tale groupoids is \cite[Chapter 5.6]{brown-ozawa}.

\begin{Definition}\label{amenable}
	Recall that a locally compact groupoid $\G$ is \emph{topologically amenable} if there exists a \emph{topological approximate invariant mean}, i.e. a sequence $\{m^{(n)}\}_{n \in \mathbb{N}}$ of families of positive and finite Radon measures $m^{(n)} = \{m^{(n)}_x: x\in \Gz\}$ satisfying
	\begin{enumerate}
		\item $\|m^{(n)}_x\|_1\leq 1$ and $m^{(n)}_x(\G\setminus s^{-1}(x)) = 0$, for all $x\in \Gz$ and $n \in \mathbb{N}$;
		\item for all $n \in \mathbb{N}$ and $f\in C_c(\G)$, the function $x \mapsto \int f \d m^{(n)}_x$ is continuous on $\Gz$;
		\item $\|m^{(n)}_x\|_1\to 1$ as $n \to \infty$, uniformly on any compact subset of $\Gz$;
		\item $\|m^{(n)}_{s(\gamma)} - \gamma m^{(n)}_{r(\gamma)} \|_1 \to 0$ as $n \to \infty$, uniformly on any compact subset of $\G$.
	\end{enumerate}
\end{Definition}

Note that for a locally compact groupoid equipped with a Haar system of measures, topological amenability can also be characterised in terms of topological approximate invariant densities. We omit the details, and guide the readers to \cite[Proposition 2.2.13]{ananth-delaroche-renault}.

\begin{Proposition}[Corollary 5.6.17, Theorem 5.6.18, \cite{brown-ozawa}]\label{etale amen}
	Let $\G$ be a locally compact and \'{e}tale groupoid. Then $\G$ is topologically amenable \emph{if and only if} the reduced groupoid $C^*$-algebra $C^*_r(\G)$ is nuclear. In this case, the natural quotient $C^*_{\max}(\G) \to C^*_r(\G)$ is an isomorphism.
\end{Proposition}

\begin{Example}\label{group ex prlm}
	A discrete group $G$ can be regarded as a groupoid with the unit space consisting of a single point. In this case, topological amenability is nothing but the classic notion of amenability for a group.
	
	Now suppose a discrete group $G$ acts on a compact space $X$ by homeomorphisms. As discussed at the beginning of Section \ref{groupoid action preliminary subsection}, we consider the transformation groupoid $X \rtimes G$. It is not hard to check directly by definition that $X \rtimes G$ is topologically amenable \emph{if and only if} the action is amenable (see for example \cite{brown-ozawa}). And it also follows directly by definition that a group $G$ is amenable \emph{if and only if} the action induced by the left multiplication on its Alexandroff one point compactification is amenable. Finally recall from \cite[Theorem 5.1.7]{brown-ozawa} that for a discrete group $G$, the following are equivalent:
	\begin{itemize}
		\item $G$ is exact;
		\item The action induced by the left multiplication of $G$ on its Stone-\v{C}ech compactification $\beta G$ is amenable;
		\item $G$ acts amenably on some compact Hausdorff topological space.
	\end{itemize}
\end{Example}

\begin{Example}
	Given a discrete metric space $X$ with bounded geometry, as introduced in \cite{Skandalis-Tu-Yu}, we may associate the coarse groupoid $G(X)$ (see Section \ref{metric space ex} for more details). It is shown in \cite[Theorem 10.29]{RoeLectures} that the reduced groupoid $C^*$-algebra $C^*(G(X))$ is isomorphic to the uniform Roe algebra $C^*_u(X)$. Therefore from Proposition \ref{etale amen} and \cite[Theorem 5.5.7]{brown-ozawa}, amenability of the coarse groupoid $G(X)$ is equivalent to the space $X$ having Property A.
\end{Example}

Finally, we discuss briefly on the amenability of groupoid actions and the notion of exactness for groupoids. A full detailed discussion is provided in \cite{Ananth-Delaroche--ExactGroupoids}.

Given a locally compact groupoid $\G$ and a $\G$-space $(Y,p)$. Following \cite[Definition 2.5]{Ananth-Delaroche--ExactGroupoids}, we say that the action is \emph{amenable}, or $Y$ is an \emph{amenable} $\G$-space, if the associated semi-direct product $Y \rtimes \G$ is topologically amenable. By \cite[Corollary 2.2.10]{ananth-delaroche-renault}, if $\G$ is an amenable locally compact groupoid, then for every $\G$-space the action is amenable. As in the group case, an \'{e}tale groupoid $\G$ is amenable \emph{if and only if} the action of $\G$ on its fibrewise Alexandroff compactification $\G_r^+$ is amenable (see \cite[Proposition 3.3]{Ananth-Delaroche--ExactGroupoids}).

In the group case as we see in Example \ref{group ex prlm}, amenable actions of groups on compact spaces have close relation with the exactness of the given group. Unfortunately, things become complicated for general groupoids. We start with the following definition.
\begin{Definition}[\cite{Ananth-Delaroche--ExactGroupoids}]\label{strongly amenable at infinity}
	We say that a locally compact groupoid $\G$ is \emph{amenable at infinity} if there is an amenable $\G$-space $(Y,p)$ such that $Y$ is fibrewise compact. If the space $(Y,p)$ can be chosen to be the fibrewise Stone-\v{C}ech compactification $(\beta_r \G, r_\beta)$ and the action is induced by the left multiplication, then we say the groupoid $\G$ is \emph{strongly amenable at infinity}.
\end{Definition}

\begin{Definition}[\cite{Ananth-Delaroche--ExactGroupoids}]\label{defn:C*-exact}
	A locally compact groupoid $\G$ with a Haar system of measures is called \emph{$C^*$-exact} if the reduced groupoid $C^*$-algebra $C^*_r(\G)$ is exact.
\end{Definition}

\begin{Proposition}[\cite{Ananth-Delaroche--ExactGroupoids}, Corollary 6.4]\label{exact prop 1}
	Let $\G$ be a locally compact, second countable and \'{e}tale groupoid. We consider the following:
	\begin{enumerate}
		\item $\G$ is strongly amenable at infinity;
		\item $\G$ is amenable at infinity;
		\item $\beta_r \G \rtimes \G$ is topologically amenable;
		\item $\G$ is $C^*$-exact.
	\end{enumerate}
	Then we have (1) $\Rightarrow$ (2) $\Rightarrow$ (4), and (1) $\Rightarrow$ (3) $\Rightarrow$ (4).
\end{Proposition}

The converse direction of the above proposition is also discussed in \cite{Ananth-Delaroche--ExactGroupoids}. In order to state their result, we need an extra notion as follows:
\begin{Definition}[\cite{Ananth-Delaroche--ExactGroupoids}, Definition 4.2]\label{defn:weakly inner amenability}
	A locally compact groupoid $\G$ is \emph{weakly inner amenable} if for any compact subset $K \subseteq \G$ and any $\varepsilon>0$, there exists a continuous bounded positive definite function $f$ on the product groupoid $\G \times \G$, properly supported, such that $|f(\gamma,\gamma)-1|<\varepsilon$ for all $\gamma \in K$.
\end{Definition}

\begin{Proposition}[\cite{Ananth-Delaroche--ExactGroupoids}, Theorem 7.6]\label{exact prop 2}
	Let $\G$ be a second countable weakly inner amenable \'{e}tale groupoid. Then conditions (1) to (4) in Proposition \ref{exact prop 1} are all equivalent.
\end{Proposition}

\section{Operator Fibre Spaces}\label{op fibre sp}

Recall that in Section \ref{groupoid C algebra prlm}, for an \'{e}tale groupoid $\G$ and any point $x\in \Gz$,  we define the regular representation $\lambda_x: C^*_r(\G) \to \B(\ell^2(\G_x))$. Fix an element $T\in C^*_r(\G)$ we obtain a map $x \mapsto \lambda_x(T)$, which we will show can be regarded as a section of a  bundle of operators on the unit space $\Gz$. In this section, we will formalise this observation and introduce the notion of operator fibre spaces, which serves as the receptacle for the symbol map which we will establish later. Again in this section we only consider the \'{e}tale case.

Fix a locally compact and \'{e}tale groupoid $\G$ with unit space $\Gz$. Consider the space
\begin{equation}\label{fibre space of operators defn}
E:= \bigsqcup_{x \in \Gz}\B(\ell^2(\G_x)).
\end{equation}
For each element $T$ in $\B(\ell^2(\G_x))\subseteq E$, we write $T_x$ to indicate the fibre it lives in. We define the projection map
$$p:E \longrightarrow \Gz \quad\mbox{by}\quad T_x \mapsto x.$$
Now we endow a topology on $E$ as follows: a net $\{T_{x_i}\}_{i \in I}$ converges to $T_x$ \emph{if and only if} $x_i \to x$, and for any $\gamma'_i \to \gamma'$, $\gamma''_i \to \gamma''$ with $s(\gamma'_i)=x_i=s(\gamma''_i)$ (which implies that $s(\gamma')=x=s(\gamma'')$), we have
$$\langle T_{x_i}\delta_{\gamma'_i}, \delta_{\gamma''_i} \rangle \to \langle T_x\delta_{\gamma'}, \delta_{\gamma''} \rangle.$$

\begin{Definition}
Given a locally compact and \'{e}tale groupoid $\G$, the space $E$ defined in (\ref{fibre space of operators defn}) equipped with the above topology is called the \emph{operator fibre space associated to $\G$}.
\end{Definition}

Now we would like to provide a description for a local basis of a given point $T_x \in E$. First let us fix some notations. For a topological space $X$ and a point $x\in X$, denote $\mathcal{N}_x$ the set of all the neighbourhoods of $x$. Since $\G$ is \'{e}tale, for any $\gamma', \gamma'' \in \G_x$, there exist neighbourhoods $V_{\gamma'}\in \mathcal{N}_x$ and $V_{\gamma''}\in \mathcal{N}_x$ such that the restrictions of $s$ on $V_{\gamma'}$ and $V_{\gamma''}$ are homeomorphisms with open images $s(V_{\gamma'})$ and $s(V_{\gamma''})$. Take $U_{\gamma', \gamma''}=s(V_{\gamma'})\cap s(V_{\gamma''})$ and
$$\zeta_{\gamma'}:=(s|_{V_{\gamma'} \cap s^{-1}(U_{\gamma',\gamma''})})^{-1}: U_{\gamma',\gamma''} \longrightarrow V_{\gamma'} \cap s^{-1}(U_{\gamma',\gamma''}),$$
$$\zeta_{\gamma''}:=(s|_{V_{\gamma''} \cap s^{-1}(U_{\gamma',\gamma''})})^{-1}: U_{\gamma',\gamma''} \longrightarrow V_{\gamma''} \cap s^{-1}(U_{\gamma',\gamma''}).$$
For any $\varepsilon>0$ and $U\in \mathcal{N}_x$ with $U \subseteq U_{\gamma',\gamma''}$, we define
$$W_{T_x}(\varepsilon; \gamma', \gamma'', U):=\{T_y \in E: y \in U \mbox{~and~} |\langle T_y \delta_{\zeta_{\gamma'}(y)}, \delta_{\zeta_{\gamma''}(y)} \rangle  -  \langle T_x \delta_{\gamma'}, \delta_{\gamma''} \rangle| < \varepsilon\}.$$

\begin{Lemma}\label{nbhd basis}
Let $\G$ be a locally compact and \'{e}tale groupoid and $E$ be the associated operator fibre space. Then
$$\{W_{T_x}(\varepsilon; \gamma', \gamma'', U): \varepsilon>0, \gamma',\gamma'' \in \G_x, \mbox{and~open~} U \subseteq U_{\gamma',\gamma''}\}$$
is a local basis of $T_x$.
\end{Lemma}

\begin{proof}
By definition, a net $\{T_{x_i}\}_{i\in I}$ converges to $T_x$ in $E$ if and only if for any $\varepsilon>0$, $\gamma', \gamma'' \in \G_x$ and $U \in \mathcal{N}_x$, there exists $i_0 \in I$, $V' \in \mathcal{N}_x$ and $V'' \in \mathcal{N}_x$, such that for any $i > i_0$, we have $x_i \in U$ and for any $\overline{\gamma}' \in V' \cap \G_{x_i}$ and $\overline{\gamma}'' \in V'' \cap \G_{x_i}$, we have
$$|\langle T_{x_i} \delta_{\overline{\gamma}'}, \delta_{\overline{\gamma}''} \rangle  -  \langle T_x \delta_{\gamma'}, \delta_{\gamma''} \rangle| < \varepsilon.$$
Since $\G$ is \'{e}tale, we can shrink $V', V''$ if necessary to ensure that $V' \subseteq V_{\gamma'}$ and $V'' \subseteq V_{\gamma''}$, which implies that
$$V' \cap \G_{x_i} = \{\zeta_{\gamma'}(x_i)\} \quad \mbox{~and~} \quad V'' \cap \G_{x_i} = \{\zeta_{\gamma''}(x_i)\}.$$
Hence, $T_{x_i} \to T_x$ if and only if for any $\varepsilon>0$, $\gamma', \gamma'' \in \G_x$ and $U \in \mathcal{N}_x$ with $U \subseteq U_{\gamma',\gamma''}$, there exists $i_0 \in I$ such that for any $i > i_0$, we have $x_i \in U$ and
$$|\langle T_{x_i} \delta_{\zeta_{\gamma'}(x_i)}, \delta_{\zeta_{\gamma''}(x_i)} \rangle  -  \langle T_x \delta_{\gamma'}, \delta_{\gamma''} \rangle| < \varepsilon.$$
So we finish the proof.
\end{proof}

\begin{Lemma}
The projection map $p: E \to \Gz$ is an open and continuous surjection.
\end{Lemma}

\begin{proof}
Clearly, $p$ is continuous and surjective. To see $p$ is open, note that $p(W_{T_x}(\varepsilon; \gamma', \gamma'', U)) = U$ for any $\varepsilon>0$, $\gamma', \gamma'' \in \G_x$ and $U \in \mathcal{N}_x$ with $U \subseteq U_{\gamma', \gamma''}$.
\end{proof}

\begin{Remark}\label{fibre rmk}
Although $E$ is called an operator fibre space, it might not be a genuine fibre space in the sense of Definition \ref{fibre space defn}. From the above, the only obstruction is that we don't know whether the topology is locally compact in general. However in some special case, we do know certain subspace of $E$ is locally compact or even fibrewise compact (see Lemma \ref{op fibre space fw cpt}).
\end{Remark}

Now denote $\Gamma(E)$ the set of all continuous sections of $E$. Note that for a general element $\varphi$ in $\Gamma(E)$, the norms $\{\|\varphi(x)\|: x \in \Gz\}$ may not have an upper bound. To overcome, we introduce the following:
\begin{Definition}
Let $\G$ be a locally compact and \'{e}tale groupoid. For each $k \in \mathbb{N}$, we define the \emph{$k$-bounded operator fibre space} to be
$$E_k:= \bigsqcup_{x \in \Gz}\B(\ell^2(\G_x))_k \quad(\subseteq E),$$
where $\B(\ell^2(\G_x))_k$ is the $k$-ball at the origin. Denote $\Gamma(E_k)$ the set of all continuous sections of $E_k$, and
$$\Gamma_b(E):=\bigcup_{k \in \mathbb{N}} \Gamma(E_k).$$
Endow a norm on $\Gamma_b(E)$ by $\|\varphi\|:=\sup_{x \in \Gz} \|\varphi(x)\|$ for $\varphi \in \Gamma_b(E)$. It is routine to check that equipped with the norm and the point-wise addition, multiplication and adjoint, $\Gamma_b(E)$ is a $C^*$-algebra, which is called the \emph{$C^*$-algebra of bounded sections of $E$}.
\end{Definition}

Now we go back to the reduced groupoid $C^*$-algebra $C^*_r(\G)$. Given an element $T \in C^*_r(\G)$ with norm $\|T\|\leq k$ for some $k \in \mathbb{N}$ and any $x \in \Gz$, we have an operator $\lambda_x(T) \in \B(\ell^2(\G_x))_k$ defined in (\ref{reduced algebra defn}). In other words, we obtain a section of $E_k$ by $x \mapsto \lambda_x(T)$.

\begin{Lemma}
The section $x \mapsto \lambda_x(T)$ defined above is continuous.
\end{Lemma}

\begin{proof}
It suffices to prove the lemma for any element $f \in C_c(\G)$. Suppose the net $\{x_i\}_{i \in I}$ converges to $x$ in $\Gz$, and we need to show that $\lambda_{x_i}(f) \to \lambda_x(f)$ in $E_k$. For any $\gamma'_i \to \gamma'$, $\gamma''_i \to \gamma''$ with $s(\gamma'_i)=x_i=s(\gamma''_i)$, we have
$$\langle \lambda_{x_i}(f)\delta_{\gamma'_i}, \delta_{\gamma''_i} \rangle = f(\gamma_i''\cdot \gamma_i'^{-1}) \to f(\gamma''\cdot \gamma'^{-1})  = \langle \lambda_x(f)\delta_{\gamma'}, \delta_{\gamma''} \rangle.$$
So the lemma holds.
\end{proof}

Therefore, we obtain a $C^*$-homomorphism $\lambda: C^*_r(\G) \to \Gamma_b(E)$ defined by $\lambda(T)(x):=\lambda_x(T)$. Now we would like to provide a more detailed picture for the image $\mathrm{Im} \lambda$.
\begin{Definition}\label{defn: equivariant section}
A section $\varphi$ in $\Gamma(E)$ is called \emph{$\G$-equivariant}, if for any $\gamma \in \G$ we have
$$\varphi(r(\gamma))=R_\gamma^* \varphi(s(\gamma)) R_\gamma,$$
where the operator $R_\gamma: \ell^2(\G_{r(\gamma)}) \to \ell^2(\G_{s(\gamma)})$ is defined by $\delta_\alpha \mapsto \delta_{\alpha \gamma}$ for any $\alpha \in \G_{r(\gamma)}$. Denote $\Gamma(E)^{\G}$, $\Gamma(E_k)^{\G}$ and $\Gamma_b(E)^{\G}$ the subset of $\Gamma(E)$, $\Gamma(E_k)$ and $\Gamma_b(E)$ consisting of $\G$-equivariant sections, respectively. Clearly, $\Gamma_b(E)^{\G}$ is a $C^*$-subalgebra of $\Gamma_b(E)$, called the \emph{$C^*$-algebra of $\G$-equivariant bounded sections of $E$}.
\end{Definition}

\begin{Lemma}
Notations as above. We have $\lambda(C^*_r(\G)) \subseteq \Gamma_b(E)^{\G}$.
\end{Lemma}

\begin{proof}
It suffices to show that for any $f \in C_c(\G)$ and $\gamma$, we have $\lambda_{r(\gamma)}(f)=R_\gamma^* \lambda_{s(\gamma)}(f) R_\gamma$ in $\B(\ell^2(\G_{r(\gamma)}))$. For any $\alpha', \alpha'' \in \G_{r(\gamma)}$, we have
\begin{eqnarray*}
  \langle R_\gamma^* \lambda_{s(\gamma)}(f) R_\gamma \delta_{\alpha'}, \delta_{\alpha''} \rangle &=& \langle \lambda_{s(\gamma)}(f) \delta_{\alpha'\gamma}, \delta_{\alpha''\gamma} \rangle \\
  &=& f(\alpha''\gamma(\alpha'\gamma)^{-1}) = f(\alpha'' \alpha'^{-1})\\
  &=& \langle \lambda_{r(\gamma)}(f) \delta_{\alpha'}, \delta_{\alpha''} \rangle.
\end{eqnarray*}
So the lemma holds.
\end{proof}

\begin{Definition}\label{defn: uniform roe algebra fibre}
For $x \in \Gz$, we define the following $\ast$-algebra
$$\mathbb{C}[\G_x]:=\{T=(T_{\gamma', \gamma''})_{\gamma', \gamma'' \in \G_x} \in \B(\ell^2(\G_x)): \exists \mbox{~compact~}K \subseteq \G, \mbox{~such~that~}T_{\gamma', \gamma''} \neq 0 \mbox{~implies~} \gamma'' \gamma'^{-1} \in K\}.$$
The \emph{uniform Roe algebra at $x$} is defined to be $C^*_u(\G_x):=\overline{\mathbb{C}[\G_x]}^{\|\cdot\|_{\B(\ell^2(\G_x))}}$.
\end{Definition}

It is obvious that for each $T \in C^*_r(\G)$ and $x \in \Gz$, the operator $\lambda_x(T)$ belongs to the uniform Roe algebra $C^*_u(\G_x)$. This suggests us to consider the following:
\begin{Definition}
Let $\G$ be an \'{e}tale groupoid. The \emph{uniform Roe fibre space} is defined to be
$$E_u:=\bigsqcup_{x \in \Gz} C^*_u(\G_x) \quad (\subseteq E).$$
Similarly, for each $k \in \mathbb{N}$, we define the \emph{$k$-bounded uniform Roe fibre space} $E_{u,k} \subseteq E_k$ to be
$$E_{u,k}:=\bigsqcup_{x \in \Gz} C^*_u(\G_x)_k \quad (\subseteq E).$$
Note that $E_u$ may not be closed in $E$, neither is $E_{u,k}$. Denote all continuous sections of $E_u, E_{u,k}$ by $\Gamma(E_u), \Gamma(E_{u,k})$, and set
$$\Gamma_b(E_u):=\bigcup_{k \in \mathbb{N}} \Gamma(E_{u,k}).$$
Note that $\Gamma_b(E_u)$ is a $C^*$-subalgebra in $\Gamma_b(E)$, called the \emph{$C^*$-algebra of bounded uniform Roe sections of $E$}. And the intersection
\begin{equation}\label{equv bdd roe sec}
\Gamma_b(E_u)^\G:=\Gamma_b(E_u) \cap \Gamma_b(E)^{\G}
\end{equation}
is also a $C^*$-subalgebra, called the \emph{$C^*$-algebra of $\G$-equivariant bounded uniform Roe sections of $E$}.
\end{Definition}

Consequently, we obtain the following:
\begin{Proposition}\label{op fibre space fw cpt}
Let $\G$ be an \'{e}tale groupoid. Then the regular representations $x \mapsto \lambda_x$ induce a $C^*$-monomorphism $\lambda: C^*_r(\G) \longrightarrow \Gamma_b(E_u)^\G$.
\end{Proposition}

Finally, we study the following nice property for bounded operator fibre spaces when the underlying unit space is first countable. This result plays an important role when we study the uniform Roe algebra of a groupoid later in Section \ref{groupoid case ex}. Note that restricted on each fibre, the topology is nothing but the weak operator topology, hence compact by Banach-Alaoglu Theorem.

\begin{Lemma}\label{lemma: first countable implies fibrewise compact}
Let $\G$ be a locally compact, $\sigma$-countable and \'{e}tale groupoid with first countable unit space $\Gz$. Then for each $k \in \mathbb{N}$, the $k$-bounded operator fibre space $E_k$ is fibrewise compact.
\end{Lemma}

\begin{proof}
By definition, we need show that for any compact $K \subseteq \Gz$, $p^{-1}(K)$ is compact in $E_k$. Since $\Gz$ is first countable and $\G$ is $\sigma$-compact, we know that $E_k$ is also first countable by Lemma \ref{nbhd basis}. So it suffices to prove that for any sequence $\{T_n\}_{n \in \mathbb{N}}$ in $p^{-1}(K)$, we can find a subsequence $\{T_{n_k}\}_{k \in \mathbb{N}}$ such that it converges in $p^{-1}(K)$.

Consider the sequence $\{x_n:=p(T_n)\}_{n \in \mathbb{N}}\subseteq K$. Since $K$ is compact and first countable, there exists a subsequence which converges to some point $x\in K$. After taking the subsequence, we may assume that $x_n \to x$. For any $\gamma', \gamma'' \in \G_x$, take $U_{\gamma', \gamma''}, \zeta_{\gamma'}$ and $\zeta_{\gamma''}$ as in the paragraph before Lemma \ref{nbhd basis}. Then there exists $n_{\gamma', \gamma''} \in \mathbb{N}$ such that
$$\{x_n: n \geq n_{\gamma', \gamma''}\} \subseteq U_{\gamma', \gamma''}.$$
Now consider the sequence:
$$\{\langle T_{x_n} \delta_{\zeta_{\gamma'}(x_n)}, \delta_{\zeta_{\gamma''}(x_n)} \rangle\}_{n \geq n_{\gamma', \gamma''}},$$
which is contained in the compact set $\{z \in \mathbb{C}: |z| \leq k\}$. Hence there exists a subsequence $I_{\gamma',\gamma''} \subseteq \mathbb{N}$ such that $\{\langle T_{x_n} \delta_{\zeta_{\gamma'}(x_n)}, \delta_{\zeta_{\gamma''}(x_n)} \rangle : n \in I_{\gamma',\gamma''}\}$ converges to a complex number $\kappa_{\gamma', \gamma''}$ with $|\kappa_{\gamma', \gamma''}| \leq k$. Since the set $\G_x \times \G_x$ is countable, using a Cantor diagonal argument, we may find a subsequence $\{n_k\}_{k \in \mathbb{N}}$ in $\mathbb{N}$ such that for any $(\gamma', \gamma'') \in \G_x \times \G_x$, we have
$$\lim_{k \to \infty} \langle T_{x_{n_k}} \delta_{\zeta_{\gamma'}(x_{n_k})}, \delta_{\zeta_{\gamma''}(x_{n_k})} \rangle = \kappa_{\gamma', \gamma''}.$$

We claim: there exists $T_x\in \B(\ell(\G_x))_k$ such that $\langle T_x \delta_{\gamma'}, \delta_{\gamma''} \rangle = \kappa_{\gamma', \gamma''}$. In fact, consider the following sesquilinear map $F: C_c(\G_x) \times C_c(\G_x) \longrightarrow \mathbb{C}$ defined by
\begin{eqnarray*}
F(\sum_{i=1}^{l} c'_i\delta_{\gamma'_i}, \sum_{j=1}^{l} c''_j\delta_{\gamma''_j}) &:=& \sum_{i,j=1}^l c'_i \overline{c''_j} \kappa_{\gamma', \gamma''} = \sum_{i,j=1}^l c'_i \overline{c''_j} \lim_{k \to \infty} \langle T_{x_{n_k}} \delta_{\zeta_{\gamma'_i}(x_{n_k})}, \delta_{\zeta_{\gamma''_j}(x_{n_k})} \rangle\\
&=& \lim_{k \to \infty} \langle T_{x_{n_k}} (\sum_{i=1}^{l} c'_i\delta_{\zeta_{\gamma'_i}(x_{n_k})}), \sum_{j=1}^{l} c''_j \delta_{\zeta_{\gamma''_j}(x_{n_k})} \rangle.
\end{eqnarray*}
Since all $T_{x_{n_k}}$'s have norm at most $k$, we obtain that
$$\sup \{|F(\xi',\xi'')|: \xi', \xi'' \in C_c(\G_x) \mbox{~with~}\|\xi'\|_2 \leq 1, \|\xi''\|_2 \leq 1\}$$
does not exceed $k$. Hence $F$ can be extended to a bounded sesquilinear map on $\ell^2(\G_x) \times \ell^2(\G_x)$ with norm at most $k$, which implies there exists $T_x\in \B(\ell(\G_x))_k$ such that
$$\langle T_x \xi',\xi'' \rangle = F(\xi',\xi'')$$
for any $\xi',\xi'' \in \ell(\G_x)$. Hence the claim holds, and we finish the proof.
\end{proof}

\section{Main Theorem}\label{main theorem section}

Having introduced sufficient background tools in previous sections, now we are in the position to establish the limit operator theory for groupoids. This follows the same philosophy as the existing theories in the Hilbert space case of groups and metric spaces. We start with the following setting.

Let $\G$ be a locally compact and \'{e}tale groupoid with compact unit space $\Gz$. Suppose that there is a dense invariant open subset $X$ in $\Gz$, then its complement $\partial X:=\Gz \setminus X$ is an invariant closed subset in $\Gz$. Concerning the associated groupoid reductions, we have the following decomposition:
\begin{equation}\label{dec for groupoid}
\G=\G(X) \sqcup \G(\partial X).
\end{equation}
Note that $\G(X)$ is open in $\G$, while $\G(\partial X)$ is closed. Clearly, we have the following short exact sequence induced by (\ref{dec for groupoid}):
\begin{equation}\label{exact seq of Cc}
0 \longrightarrow C_c(\G(X)) \longrightarrow C_c(\G) \longrightarrow  C_c(\G(\partial X)) \longrightarrow 0,
\end{equation}
where the map $C_c(\G(X)) \rightarrow C_c(\G)$ is the inclusion, and $C_c(\G) \rightarrow  C_c(\G(\partial X))$ is the restriction.

We may complete this sequence with respect to the reduced norms and obtain the following sequence:
\begin{equation}\label{short seq of red}
0 \longrightarrow C^*_r(\G(X)) \stackrel{i}{\longrightarrow} C^*_r(\G) \stackrel{q}{\longrightarrow} C^*_r(\G(\partial X)) \longrightarrow 0.
\end{equation}
By constructions, $i$ is injective and $q$ is surjective, so we may regard $C^*_r(\G(X))$ as a ideal of $C^*_r(\G)$. However, in general (\ref{short seq of red}) fails to be exact at the middle item, which is crucial in \cite{HLS2002} to study the counterexample to the Baum-Connes conjecture. We may also complete the sequence (\ref{exact seq of Cc}) with respect to the maximal norms and obtain the following sequence:
\begin{equation}\label{short seq of max}
0 \longrightarrow C^*_{\max}(\G(X)) \stackrel{i}{\longrightarrow} C^*_{\max}(\G) \stackrel{q}{\longrightarrow} C^*_{\max}(\G(\partial X)) \longrightarrow 0,
\end{equation}
which is easy to check by definition to be exact automatically (see for example \cite[Lemma 2.10]{cts-trace-gpoid-III}).

\begin{Definition}\label{defn: limit op}
	Let $\G,X,\partial X$ be as above and $T \in C^*_r(\G)$. For each $\omega \in \partial X$, we define the \emph{limit operator} of $T$ at $\omega$ to be $\lambda_\omega(T) \in \B(\ell^2(\G_\omega))$, where $\lambda_\omega: C^*_r(\G) \to \B(\ell^2(\G_\omega))$ is the regular representation at $\omega$.
\end{Definition}

Note that for $\omega \in \partial X$, we have $\G_\omega = \G(\partial X)_\omega$. Hence from the definition we have
$$\|q(T)\|=\sup_{\omega \in \partial X} \|\lambda_\omega(T)\|$$
for any $T\in C^*_r(\G)$, where $q: C^*_r(\G) \longrightarrow C^*_r(\G(\partial X))$ is the quotient homomorphism from (\ref{short seq of red}).

As Roe did in \cite{roe-band-dominated}, we would like to amalgamate all limit operators into a single homomorphism. Unfortunately, limit operators $\lambda_\omega$'s do \emph{not} live in the same space generally, hence we have to appeal to the language of operator fibre spaces established in Section \ref{op fibre sp}.

For the given locally compact \'{e}tale groupoid $\G$, let $\lambda: C^*_r(\G) \to \Gamma_b(E_u)^{\G}$ be the $C^*$-monomorphism established in Proposition \ref{op fibre space fw cpt}. Denote the operator fibre space associated to the reduction $\G(\partial X)$ by
$$E^\partial:= \bigsqcup_{x \in \partial X}\B(\ell^2(\G_x)).$$
Similarly, denote $E^\partial_k$ and $E^\partial_u$ the $k$-bounded version and the uniform Roe version respectively, and their intersection by $E^\partial_{u,k}$, as we did in Section \ref{op fibre sp}. Recall that the $C^*$-algebra of $\G(\partial X)$-equivariant bounded uniform Roe sections of $E^\partial$ as defined in (\ref{equv bdd roe sec}) is:
$$\Gamma_b(E_u^\partial)^{\G(\partial X)} = \bigcup_{k \in \mathbb{N}} \Gamma(E^\partial_{u,k})^{\G(\partial X)}.$$
Now the restriction of $E$ on $E^\partial$ induces a $C^*$-homomorphism between sections
\begin{equation}\label{res map}
\mathrm{Res}: \Gamma_b(E_u)^{\G} \longrightarrow \Gamma_b(E_u^\partial)^{\G(\partial X)}.
\end{equation}

\begin{Definition}\label{limit morphism defn}
	Let $\G,X,\partial X$ be as above and $\lambda, \mathrm{Res}$ defined in Proposition \ref{op fibre space fw cpt} and (\ref{res map}). We define the \emph{symbol morphism} to be the composition
	$$\varsigma=\mathrm{Res} \circ \lambda: C^*_r(\G) \longrightarrow \Gamma_b(E_u^\partial)^{\G(\partial X)}.$$
\end{Definition}

It is easy to check that for each $\omega \in \partial X$, the regular representation $\lambda_\omega: C^*_r(\G) \to \B(\ell^2(\G_\omega))$ factors through $C^*_r(\G(\partial X))$. More precisely, we have the following commutative diagram:
\begin{displaymath}
\xymatrix{
	C^*_{r}(\G) \ar[r]^-{\textstyle \lambda_\omega} \ar[d]_-{\textstyle q} & \B(\ell^2(\G_\omega)) \ar@{=}[d] \\
	C^*_{r}(\G(\partial X)) \ar[r]^-{\textstyle \lambda_\omega^\partial} & \B(\ell^2(\G(\partial X)_\omega))}
\end{displaymath}
where the bottom map $\lambda_\omega^\partial$ is the regular representation for the reduction $\G(\partial X)$ at $\omega$. Hence we have
$$\lambda_\omega(T) = \lambda_\omega^\partial(q(T))$$
for $T\in C^*_r(\G)$. Consequently, we obtain the following lemma.

\begin{Lemma}\label{commt diag}
	Notations as above. The following diagram commutes:
	\begin{displaymath}
	\xymatrix{
		C^*_{r}(\G) \ar[r]^-{\textstyle \varsigma} \ar[d]_-{\textstyle q} & \Gamma_b(E_u^\partial)^{\G(\partial X)} \ar@{=}[d] \\
		C^*_{r}(\G(\partial X)) \ar[r]^-{\textstyle \lambda^\partial} & \Gamma_b(E_u^\partial)^{\G(\partial X)}}
	\end{displaymath}
	where $\lambda^\partial$ is the monomorphism defined as in Proposition \ref{op fibre space fw cpt} for the reduction groupoid $\G(\partial X)$.
\end{Lemma}

Now we are in the position to state our main result:

\begin{Theorem}\label{main theorem}
	Let $\G$ be a locally compact, $\sigma$-compact and \'{e}tale groupoid with compact unit space $\Gz$, $X$ be an invariant open dense subset in $\Gz$ and $\partial X=\Gz \setminus X$. Suppose the reduction groupoid $\G(\partial X)$ is topologically amenable. Then for any element $T$ in the reduced groupoid $C^*$-algebra $C^*_r(\G)$, the following conditions are equivalent:
	\begin{enumerate}
		\item $T$ is invertible modulo $C^*_r(\G(X))$.
		\item The image of $T$ under the symbol morphism, $\varsigma(T)$, is invertible in $\Gamma_b(E_u^\partial)^{\G(\partial X)}$.
		\item For each $\omega \in \partial X$, the limit operator $\lambda_\omega(T)$ is invertible, and
		$$\sup_{\omega \in \partial X} \|\lambda_\omega(T)^{-1}\| < \infty.$$
		\item For each $\omega \in \partial X$, the limit operator $\lambda_\omega(T)$ is invertible.
	\end{enumerate}
\end{Theorem}

Here we only prove the equivalence among conditions (1), (2) and (3), while the equivalence between (3) and (4) is left to the next section after some other technical tools are developed.
\begin{proof}[Proof of Theorem \ref{main theorem}, ``(1) $\Leftrightarrow$ (2) $\Leftrightarrow$ (3)'':]
	First we show that the sequence (\ref{short seq of red}) is exact. Notice that we have the following commutative diagram:
	\begin{displaymath}
	\xymatrix{
		0 \ar[r] & C^*_{\max}(\G(X)) \ar[r] \ar[d] & C^*_{\max}(\G) \ar[r] \ar[d] & C^*_{\max}(\G(\partial X)) \ar[r] \ar[d] & 0\\
		0 \ar[r] & C^*_{r}(\G(X)) \ar[r] & C^*_{r}(\G) \ar[r]^-{\textstyle q} & C^*_{r}(\G(\partial X)) \ar[r] & 0}
	\end{displaymath}
	As explained before, the top row is exact and the bottom row is exact at the second and the fourth items. Since $\G(\partial X)$ is amenable, the third vertical map is an isomorphism by Proposition \ref{etale amen}. Via an elementary diagram chasing argument, the bottom row is exact at the third item as well.
	
Now consider the following commutative diagram coming from Lemma \ref{commt diag}:
\begin{displaymath}
	\xymatrix{
		0 \ar[r] & C^*_r(\G(X)) \ar[r] \ar@{=}[d] & C^*_{r}(\G) \ar[r]^-{\textstyle \varsigma} \ar@{=}[d] & \Gamma_b(E_u^\partial)^{\G(\partial X)}\\
		0 \ar[r] & C^*_r(\G(X)) \ar[r] & C^*_{r}(\G) \ar[r]^-{\textstyle q} & C^*_{r}(\G(\partial X)) \ar[r] \ar[u]^{\lambda^\partial} & 0.}
\end{displaymath}
	Note that the lower horizontal sequence is exact and $\lambda^\partial$ is injective, hence the upper horizontal sequence is also exact. Therefore for $T \in C^*_{r}(\G)$, we have that $T$ is invertible modulo $C^*_r(\G(X))$ if and only if $\varsigma(T)$ is invertible, which proves ``(1) $\Leftrightarrow$ (2)''.
	
	Now we move on to ``(2) $\Leftrightarrow$ (3)''. Consider the following $C^*$-homomorphism
	$$\iota: \Gamma_b(E_u^\partial)^{\G(\partial X)} \longrightarrow \prod_{\omega \in \partial X}\B(\ell^2(\G_\omega))$$
	defined by $\iota(\xi)=(\xi(x))_{x \in \partial X}$, which is injective. Hence for any $T \in C^*_r(\G)$, $\varsigma(T)$ is invertible if and only if $\iota\circ \varsigma (T)$ is invertible in $\prod_{\omega \in \partial X}\B(\ell^2(\G_\omega))$ since $C^*$-algebras are inverse-closed. Note that
	$$\iota\circ \varsigma (T)=(\lambda_\omega(T))_{\omega \in \partial X},$$
	hence the above is also equivalent to the condition that each $\lambda_\omega(T)$ is invertible and their inverses have uniform bounded norms.
\end{proof}

\begin{Remark}
	Note that the hypothesis of the unit space $\Gz$ being compact ensures that the reduced groupoid $C^*$-algebra $C^*_r(\G)$ is unital. Although the main theorem can be modified to hold in the general case as well, here we only focus on the compact case to simplify the arguments. This also due to the fact that all the examples we study in Section \ref{application section} have compact unit spaces.
\end{Remark}

\begin{Remark}\label{rk:main theorem}
	It is clear from Definition \ref{amenable} that if $\G$ is amenable, then both of the reductions $\G(X)$ and $\G(\partial X)$ are amenable. Conversely when $\G(X)$ is amenable, $\G$ is amenable if and only if $\G(\partial X)$ is amenable from the Five Lemma. In Section \ref{application section}, most of the examples satisfy the condition that $\G(X)$ is amenable, so there is no difference between the hypothesis that $\G$ is amenable and $\G(\partial X)$ is amenable.
\end{Remark}

\begin{Remark}\label{main theorem remark}
	From the above proof of ``(1) $\Leftrightarrow$ (2) $\Leftrightarrow$ (3)", we know that the assumption of topological amenability is only used to show that the short sequence (\ref{short seq of red}) is exact. Hence we may weaken the amenability hypothesis in Theorem \ref{main theorem} to the exactness of (\ref{short seq of red}), and the result that ``(1) $\Leftrightarrow$ (2) $\Leftrightarrow$ (3)" still holds. This observation will be used later in Section \ref{grp cptf ex}.
\end{Remark}

\section{An Extension of Work by  Exel and Nistor-Prudhon}\label{sec: extension of Exel}

This section is mainly devoted to the proof of ``(3) $\Leftrightarrow$ (4)'' in Theorem \ref{main theorem}, based on our generalisation of a result of Exel and Nistor-Prudhon. Let us explain the idea first. Recall Exel originally proved the following:

\begin{Proposition}[\cite{exel-invertibiles}]\label{Prop: Exel's result}\label{thm: Exel's original thm}
Let $\G$ be a locally compact, \emph{second countable, \'{e}tale} and amenable groupoid \emph{with compact object space}, then an element $a\in C_r^*(\G)$ is invertible if and only if $\lambda_x(a)$ is invertible for every $x\in \Gz$, where $\lambda_x$ is the regular representation at $x$.
\end{Proposition}

Consequently when the groupoid $\G$ in Theorem \ref{main theorem} is additionally assumed to be second countable, we may apply Exel's result to obtain the equivalence between (3) and (4) directly. Unfortunately as we will see in Section \ref{application section}, many groupoids coming from interesting examples, even including those from the classic limit operator theory for groups (Section \ref{group case}), are \emph{not} second countable. However all of them are $\sigma$-compact, which from results of Austin-Georgescu in \cite{austin-georgescu} are approximable by second countable groupoids,  so we would like to study Exel-type result for $\sigma$-compact groupiods.

Our generalisation is more than enough to prove Theorem \ref{main theorem} and, as far as we know, it is the most general among these results. We choose to prove the most general version since it might have its own interest. Recall that Nistor and Prudhon extended Exel's result to the following:
\begin{Proposition}[\cite{Nistor--Prudhon}]\label{thm:NP extension of Exel}
Let $\G$ be a locally compact, \emph{second countable} and amenable groupoid \emph{with a Haar system}, then an element $1+a \in C_r^*(\G)^+:=C^*_r(\G) \oplus \CC$ is invertible if and only if $1+\lambda_x(a)$ is invertible for every $x\in \Gz$, where $\lambda_x$ is the regular representation at $x$.
\end{Proposition}

And the following is our generalisation. 

\begin{Theorem}\label{thm:extention of Exel}
Let $\G$ be a locally compact, \emph{$\sigma$-compact} and amenable groupoid \emph{with a Haar system}, then an element $1+a \in C_r^*(\G)^+$ is invertible if and only if $1+\lambda_x(a)$ is invertible for every $x\in \Gz$.
\end{Theorem}

Our approach relies on the result from \cite{austin-georgescu} that the topology of a $\sigma$-compact groupoid may be approximated in a very controlled way by second countable topologies.  The reader can compare our techniques here with \cite[Proposition 3.8]{Ananth-Delaroche--ExactGroupoids}. Before we get into the proof of Theorem \ref{thm:extention of Exel}, let us show how to use it to finish the proof of Theorem \ref{main theorem}:

\begin{proof}[Proof of Theorem \ref{main theorem}, ``(3) $\Leftrightarrow$ (4)'':] 

It suffices to show that ``(4) $\Rightarrow$ (1)''.
Note that $\G(\partial X)$ satisfies all the assumptions of Theorem \ref{thm:extention of Exel} with the addition that the object space is compact, which implies that  $C^*_r(\G(\partial X))$ is unital. Given $T \in C^*_r(\G)$ and applying Theorem \ref{thm:extention of Exel} to the groupoid $\G(\partial X)$, condition (4) implies that $q(T)$ is invertible in $C^*_r(\G(\partial X))$. As shown in the proof of ``(1) $\Leftrightarrow$ (2)'' in Theorem \ref{main theorem}, we know that the short sequence (\ref{short seq of red}) is exact. So condition (1) holds, and we finish the proof.
\end{proof}

The rest of this section is devoted to the proof of Theorem \ref{thm:extention of Exel}, and divided into three parts. As we need to strengthen the approximations given in \cite{austin-georgescu} for our purpose, we go over the notions and basic techniques in the first two subsections, where the theory of uniform spaces is recalled first as the major tool to construct the approximations. Then we state and prove our more controlled approximation result and finish the proof of Theorem \ref{thm:extention of Exel}. 

\subsection{Review of uniform spaces}\label{subsec:uniform spaces}
Here we recall briefly the theory of uniform spaces, mainly on the fact that they can be presented as inverse limits of metrisable spaces. This is the fundamental building block for the approximation results in \cite{austin-georgescu}. We suggest \cite{isbell} as a standard reference on uniform spaces.

Let $X$ be a set and $\mathcal{U}, \mathcal{V}$ be covers of $X$. $\mathcal{U}$ is said to \emph{refine} $\mathcal{V}$ (equivalently, $\mathcal{V}$ \emph{coarsens} $\mathcal{U}$), written $\mathcal{U} \prec \mathcal{V}$, if each element in $\mathcal{U}$ is contained in some element of $\mathcal{V}$. For a subset $A\subset X$, we define the \emph{star of $A$ against $\mathcal{U}$}, denoted by $st(A,\mathcal{U})$, to be the set $\bigcup\{U\in \mathcal{U}:U\cap A\neq \emptyset\}$.
We say that $\mathcal{U}$ \emph{star refines} $\mathcal{V}$, written $\mathcal{U}\leq \mathcal{V}$, if $\{st(U,\mathcal{U}):U\in \mathcal{U}\}$ refines $\mathcal{V}$. Now we recall the notion of uniform spaces, and a prototypical example is a metric space with covers of positive Lebesgue number.

\begin{Definition}
A \emph{uniform space} is a set $X$ with a collection of covers $\mathcal{C}$ of $X$, called the \emph{uniform structure} on $X$, that is closed under coarsening and such that if $\mathcal{U},\mathcal{V}\in \mathcal{C}$ then there exists $\mathcal{W}\in \mathcal{C}$ such that $\mathcal{W}$ star refines both $\mathcal{U}$ and $\mathcal{V}$. Elements in $\mathcal{C}$ are called \emph{uniform covers}. A function $f: X \to Y$ between uniform spaces is \emph{uniformly continuous} if the pre-image of uniform covers are uniform covers.
\end{Definition}

If $X$ is a set, we call a sequence of uniform covers $\mathcal{U}_0\ge \mathcal{U}_1\ge \mathcal{U}_2\ge \ldots$ a \emph{normal sequence of covers}. Note that every normal sequence of covers defines a uniform structure on $X$ by taking all coarsenings of covers from the sequence and, furthermore, every ``minimal'' uniform structure is obtained in this fashion from a normal sequence of covers. One nice feature of normal sequences is that they correspond exactly to the pseudo-metrisable uniform structures on $X$. We give an outline for reader's convenience.

For elements $x,y \in X$, let $n(x,y)$ denote the maximum integer $k$ such that $x$ and $y$ are both contained in an element of $\mathcal{U}_k$, and $\infty$ if no such maximum exists. Let $\rho :X\times X \to [0,1)$ be defined by $\rho(x,y) = 2^{-n(x,y)}$, with the convention that $2^{-\infty} = 0$. We observe that $\rho$ itself is not necessarily a pseudo-metric (because it may not satisfy the triangle inequality), but can be mofidied to a pseudo-metric $d$ via $d(x,y) = \inf \sum_{i=1}^n \rho(x_i,x_{i+1})$, where the infimum is taken over all chains $x = x_1, x_2, \ldots, x_n =y$ in $X$. Write $(X,\langle \{\mathcal{U}_n\}\rangle)$ to denote the resulting uniform/pseudo-metric structure on $X$, and $X_{\{\mathcal{U}_n\}}$ the resulting metric quotient.

\begin{Definition}\label{cofinalrefinement}
Let $\{\mathcal{U}_n\}_n$ and $\{\mathcal{V}_n\}_n$ be two normal sequences of covers. We say that $\{\mathcal{V}_n\}_n$ \emph{cofinally refines} $\{\mathcal{U}_n\}_n$, if for every $m\ge 0$ there exists $k(m)$ such that $\mathcal{V}_{k(m)}\leq \mathcal{U}_m$.
\end{Definition}

Notice that $\{\mathcal{U}_n\}$ cofinally refines $\{\mathcal{V}_n\}$ \emph{if and only if} the identity map $\mathrm{Id}:(X,\langle \{\mathcal{U}_n\}\rangle)\to (X,\langle \{\mathcal{V}_n\}\rangle)$ is uniformly continuous; moreover this implies that the canonical map $X_{\{\mathcal{U}_n\}} \to X_{\{\mathcal{V}_n\}}$ is uniformly continuous. It turns out that the collection of normal sequences of uniform covers forms a directed set under cofinal refinement, and we have the following:

\begin{Proposition}[\cite{isbell}]\label{Prop: inv app for uniform spaces}
Any uniform space is the inverse limit of metrisable uniform spaces; the inverse system being indexed by the collection of normal sequences with cofinal refinement being the partial ordering.
\end{Proposition}

\begin{Remark}[Remark on associated topologies]
Every uniform structure on a set $X$ induces a topology by saying that a set $A\subset X$ is a \emph{neighborhood} of a point $x \in X$ if there exists a uniform cover $\mathcal{U}$ such that $st(x,\mathcal{U})\subset A$. The uniform spaces we are interested in come from topology, and it is well known that the topologies which are induced by uniform structures are exactly the completely regular topologies. What is less well known is that there is a more canonical class of topological spaces which are intimately linked with uniform structures, namely the paracompact spaces. Indeed, one can define a space to be \emph{paracompact} if the collection of open covers forms a base for a uniform structure (one must restrict to only finite open covers for completely regular spaces). In fact, by identifying paracompact spaces with this particular uniform structure, one can easily see that continuous maps of paracompact spaces correspond exactly to uniformly continuous maps. The paracompact spaces we are interested in here are locally compact and $\sigma$-compact spaces. The strategy in \cite{austin-georgescu} and this paper is actually to approximate the underlying uniform structure for a locally compact and $\sigma$-compact groupoid, which consequently approximates its induced topological structure.
\end{Remark}

The following lemma will be used later:
\begin{Lemma}[\cite{austin-georgescu}]\label{lemma: paracompact}
Let $X$ be a locally compact, $\sigma$-compact topological space. Then $X$ is paracompact, hence every open cover admits an open start refinement.
\end{Lemma}

\subsection{Review of Austin-Georgescu's approximation result}

Now we focus on groupoids, and recall the explicit approximations constructed in \cite{austin-georgescu}. We start with the following notion:
\begin{Definition}[\cite{austin-georgescu}]\label{approximate}
Let $\G$ be a locally compact topological groupoid. An \emph{inverse approximation of $\G$} is an inverse system $\{\G_\alpha, q^\alpha_\beta:\G_\alpha \to \G_\beta\}_{\alpha \in A}$ where each $\G_\alpha$ is a locally compact groupoid and the index set $A$ is directed, satisfying:
\begin{enumerate}
\item for each $\alpha \ge \beta \in A$, the map $q^\alpha_\beta:\G_\alpha \to \G_\beta$ is a proper continuous and surjective groupoid morphism, and moreover, $q^\beta_\gamma \circ q^\alpha_\beta  = q^\alpha_\gamma$ whenever $\alpha \ge \beta \ge \gamma$;
\item $q^\alpha_\alpha = id_{\G_\alpha}$ for all $\alpha \in A$; and
\item $\varprojlim_\alpha \G_\alpha = \G$ in the category of topological groupoids with proper continuous morphisms.
\end{enumerate}
\end{Definition}

\begin{Remark}
We denote the canonical projections from $G$ to the inverse system by $q_\alpha:G\to G_\alpha$.
\end{Remark}

As indicated in Theorem \ref{thm:extention of Exel}, we are mostly interested in $\sigma$-compact spaces. In order to approximate them, the following concept is required and useful:

\begin{Definition}[\cite{austin-georgescu}]\label{groupoidexhuastion}
Let $X$ be a locally compact and $\sigma$-compact space. We say that a collection $\{K_n:n\in \NN\}$ is an \emph{exhaustion of $X$ by compact subsets}, or simply an \emph{exhaustion}, if each $K_n$ is a compact neighbourhood in $X$, $K_n\subset int (K_{n+1})$ and such that $\bigcup_n int(K_n) = X$.

In the case that $\G$ is a locally compact and $\sigma$-compact groupoid, for any exhaustion $\{K_n\}$ of $\G$, the sequence $\{K_n':= K_n\cup r(K_n)\cup s(K_n)\}$ is also an exhaustion of $\G$, and $\{K_n'|_{\Gz}\}$ is an exhaustion of the unit space $\Gz$. We call an exhaustion obtained in this manner a \emph{groupoid exhaustion}.
\end{Definition}

Also recall from \cite{austin-georgescu} that \emph{an open cover of a groupoid $\G$} is a pair $(\mathcal{W}^1,\mathcal{W}^0)$, where $\mathcal{W}^1$ is an open cover of $\G$ and $\mathcal{W}^0$ is an open cover of $\Gz$. The reason for having to take covers of both is that a very important part of a groupoid is its structure as a fibration over its unit space.

Now assume $\G$ is a locally compact and $\sigma$-compact groupoid with a fixed groupoid exhaustion $\{K_n\}$. A crucial technical part in \cite{austin-georgescu} is to construct normal sequences of open covers of $\G$ satisfying certain natural but delicately-designed conditions, such that each of the resulting metrisable quotients (see Section \ref{subsec:uniform spaces}) possesses a compatible second countable and locally compact groupoid structure, and the quotient maps are proper, continuous and surjective groupoid morphisms (\cite[Proposition 6.5, Theorem 6.8]{austin-georgescu}). Furthermore, these kind of normal sequences are cofinal in the collection of all normal sequences, thus Proposition \ref{Prop: inv app for uniform spaces} shows that $\G$ admits an inverse approximation by second countable and locally compact groupoids (\cite[Theorem 6.10]{austin-georgescu}). Additionally, if the given groupoid $\G$ possesses a Haar system, then the normal sequences above to be modified to ensure that each quotient also possess a Haar system, and the quotients are Haar system preserving (\cite[Theorem 6.9]{austin-georgescu}).

We would like to study this inverse approximations further, so we have to refer to those specific normal sequences of covers mentioned above. However, we decide \emph{not} to present all the precise details since they are very technical and require more notions which are only used to ensure that the resulting metrisable quotient possess a groupoid structure. Therefore, we choose the following neutral way to put them into a ``black box", and guide the interested readers to the original paper \cite{austin-georgescu} for details.

\begin{Definition}[\cite{austin-georgescu}, Definition 6.6]\label{defn:groupoid normal sequence}
A normal sequence of open covers which satisfies Properties (1), (3)-(7) in \cite[Proposition 6.5]{austin-georgescu} is called a \emph{groupoid normal sequence} for $\G$.
\end{Definition}

Now the above analysis can be converted into the following result:

\begin{Proposition}[\cite{austin-georgescu}, Theorem 6.8]
Let $\G$ be a locally compact and $\sigma$-compact groupoid. For any normal sequence $\{\mathcal{W}_n\}$ of open covers of $\G$, there exists a groupoid normal sequence $\alpha=\{\mathcal{U}_n\}$ of open covers of $\G$ cofinally refining $\{\mathcal{W}_n\}$, and the induced metrisable quotient $\G_{\alpha}$ is a locally compact and second countable groupoid. Moreover, the quotient map $q_\alpha: \G \to \G_{\alpha}$ is proper and continuous, and the pre-image of the cover of $\G_{\alpha}$ by balls of radius $\frac{1}{2^n}$ refines $\mathcal{W}_n$.
\end{Proposition}

Via a diagonal argument, we obtain the following version dealing with a countable family of normal sequences simultaneously. This result will be used several times later.

\begin{Proposition}\label{Prop: AG approximation 6.8}
Let $\G$ be a locally compact and $\sigma$-compact groupoid. For each $l \in \mathbb{N}$, suppose $\{\mathcal{W}_n^l\}_{n}$ is a normal sequence of open covers of $\G$. Then there exists a groupoid normal sequence $\alpha=\{\mathcal{U}_n\}$ of open covers of $\G$ cofinally refining $\{\mathcal{W}_n^l\}$ for every $l$, and the induced metrisable quotient $\G_{\alpha}$ is a locally compact and second countable groupoid. Moreover, the quotient map $q_\alpha: \G \to \G_{\alpha}$ is proper and continuous, and the sequence of open covers consisting of pre-image of the cover of $\G_{\alpha}$ by balls of radius $\frac{1}{2^n}$ cofinally refines $\{\mathcal{W}_n^l\}_n$ for each $l$.
\end{Proposition}

\subsection{Strengthened Approximation Theorem}

Now we would like to study the amenabilities of the inverse approximations given in \cite{austin-georgescu}, and use them to prove Theorem \ref{thm:extention of Exel}. Unfortunately, the groupoids $\G_\alpha$'s in the inverse approximation being amenable does not automatically follow from the fact that they are \emph{quotients} (i.e. topological quotients such that the quotient map is a groupoid morphism) of the amenable groupoid $\G$. The following example provides a counterexample in the general case.

\begin{Example}
Let $\Gamma$ be a countable discrete and exact group that is not amenable (a free group on two generators will do). Notice that the Stone-\v{C}ech compactification $\beta\Gamma$ equivariantly quotients to the Alexandroff one-point compactification $\Gamma^+$, and hence the transformation groupoid $\beta\Gamma\rtimes \Gamma$ quotients to $\Gamma^+ \rtimes\Gamma$. Note that $\Gamma^+ \rtimes\Gamma$ is not amenable while $\beta\Gamma\rtimes \Gamma$ is.
\end{Example}

However in the setting of Proposition \ref{Prop: AG approximation 6.8}, we may modify the groupoid normal sequences of covers to ensure that each quotient is still amenable. The idea is inspired by \cite[Theorem 6.9]{austin-georgescu} where it is proved that a Haar system of measures on $\G$ can be pushed forward to a Haar system on the metric quotient induced by the groupoid normal sequences. To do so, we generalise the result to a sequence of systems of measures which were introduced by Renault in \cite{Renault-Representations}.

\begin{Definition}
Let $X,Y$ be locally compact spaces, and $\pi:X\to Y$ be a continuous open surjective map. A \emph{$\pi$-system of measures} is a collection of positive Radon measures $\{m_y:y\in Y\}$ on $X$ such that $m_y$ is supported on $\pi^{-1}(y)$, and is continuous in the sense that for every $f\in C_0(X)$, the function $y\to \int_X f \d m_y$
is continuous on $Y$.
\end{Definition}

\begin{Definition}
Let $X_i,Y_i$ be locally compact spaces, $\pi_i: X_i \to Y_i$ be continuous open surjective maps for $i=1,2$, and $\{m_y: y\in Y_1\}$, $\{\overline{m}_{\bar{y}}: \bar{y} \in Y_2\}$ be $\pi_i$-system of measures, respectively. Suppose $f: X_1 \to X_2$, $g: Y_1 \to Y_2$ are proper and continuous functions satisfying $\pi_2 \circ f = g \circ \pi_1$. Then $f$ is said to be \emph{measure-preserving} if for any $y\in Y_1$, we have $f_\ast(m_y)=\overline{m}_{f(y)}$.
\end{Definition}

The systems we are interested in come from the source map $s$ of groupoids. Notice that a Haar system is an example of an $s$-system of measures. When the groupoid has a topological approximate invariant mean $\{m^{(k)}\}_{k \in \mathbb{N}}$ (Definition \ref{amenable}), then each $m^{(k)}$ is also an $s$-system of measures. Now we state the following technical lemma, dealing with a sequence of systems simultaneously. The proof follows almost the same as that of \cite[Theorem 6.5]{austin-georgescu} together with a diagonal argument, so we only provide the sketch here.
\begin{Lemma}\label{coverings}
Let $\G$ be a locally compact and $\sigma$-compact groupoid with a fixed groupoid exhaustion $\{K_n\}$. For each $l \in \mathbb{N}$, let $\{\mathcal{W}_n^l\}_{n}$ be a normal sequence of open covers of $\G$. For each $k \in \mathbb{N} \cup \{0\}$, suppose $m^{(k)} =\{m^{(k)}_x:x\in \Gz\}$ is a $\pi$-system of measures for some open surjective continuous map $\pi: \G \to \Gz$. Then the groupoid normal sequence $\{\mathcal{U}_n = (\mathcal{U}_n^0,\mathcal{U}^1_n)\}$ in Proposition \ref{Prop: AG approximation 6.8} can be modified to satisfy the following additional condition:

$\bullet$ Fix $\{f^n_j:j \in J_n\}$ a finite partition of unity of $K_n$ whose supports refine $\mathcal{U}^1_n$. Let $(\lambda_j)_j \subset \CC$ be any sequence with $|\lambda_j| < n$. Then for each open set $U\in \mathcal{U}^1_{n+1}$ and any $x,y\in s(U)$ and for all $k \leq n$, we have
\begin{equation*} \label{eqn.approximatespartition2}
\left|\int_{\G} \left(\sum_{j} \lambda_j f^n_j\right)\, \d m^{(k)}_x - \int_{\G} \left(\sum_{j} \lambda_j f^n_j\right)\, \d m^{(k)}_y\right| < \frac{1}{n}.
\end{equation*}

\ignore{; and for any $u,v \in U$ and $j \in J_n$, we have
$$|f^n_j(u)-f^n_j(v)| < \frac{1}{12n^2|J_n|M_n}=: \epsilon_n,$$
where $M_n:=\sup_{x\in \Gz} m^{(0)}_x(K_{n+1})$. Moreover, for any $U \in \mathcal{U}^1_{n+1}$ with $U \cap  \supp(f^n_j) \neq \emptyset$ for some $j \in J_N$, then $U \subseteq int(K_{n+1})$.}
\end{Lemma}

\begin{proof}[Sketch of the proof:]
Given $n \in \mathbb{N}$, suppose $\mathcal{U}_1, \ldots, \mathcal{U}_n$ have been chosen to satisfy the requirements for a groupoid normal cover (Definition \ref{defn:groupoid normal sequence}) and the above condition such that for $k=1,\ldots, n$, each $\mathcal{U}_k$ refines $\{\mathcal{W}_k^l\}$ for $l \leq k$. Note that the maps $x \mapsto \int_{\G} f^n_j \d m^{(k)}_x$ is continuous for each $j$ and $k=1,\ldots, n$. Hence we can find a cover $\mathcal{V}_n^0$ of $\Gz$ such that for $x,y \in V \in \mathcal{V}_n^0$ and $|\lambda_j| < n$, we have
\begin{equation*}
    \left|\int_{\G} \left(\sum_{j} \lambda_j f^n_j\right)\, \d m^{(k)}_x - \int_{\G} \left(\sum_{j} \lambda_j f^n_j\right)\, \d m^{(k)}_y\right| < \frac{1}{n}
\end{equation*}
for all $k \leq n$. Define $\mathcal{V}_n^1:=s^{-1}(\mathcal{V}_n^0)$, and we obtain an open cover $\mathcal{V}_n=(\mathcal{V}_n^0, \mathcal{V}_n^1)$ of $\G$. Furthermore, after shrinking $\mathcal{V}_n$ if necessary, we may assume that $\mathcal{V}_n$ refines $\{\mathcal{W}_{n+1}^l\}$ for $l \leq n+1$. Now we can construct an open cover $\mathcal{U}_{n+1}$ of $\G$ such that it refines $\mathcal{V}_n$ and satisfies the condition for groupoid normal sequences, as done in \cite[Theorem 6.5]{austin-georgescu}. 
\ignore{Furthermore, since each $f^n_j \in C_c(\G)$ is uniformly continuous and $J_n$ is a finite index set, there exists an open cover $\mathcal{V}'_n$ of $\G$ such that for any $u,v \in V \in \mathcal{V}'_n$ and any $j \in J_n$, we have
$$|f^n_j(u)-f^n_j(v)| < \epsilon_n.$$
Also fix an open cover $\mathcal{V}_n''$ of $\G$ such that if $V'' \in \mathcal{V}_n''$ satisfying $V''\cap  \mathrm{supp}(f^n_j) \neq \emptyset$ for some $j\in J_n$, then $V'' \subseteq int(K_{n+1})$. Now define $\mathcal{V}_n^1$ to be a refinement of $s^{-1}(\mathcal{V}_n^0)$, $\mathcal{V}'_n$ and $\mathcal{V}''_n$, and we obtain an open cover $\mathcal{V}_n=(\mathcal{V}_n^0, \mathcal{V}_n^1)$ of $\G$. Furthermore, after shrinking $\mathcal{V}_n$ if necessary, we may assume that $\mathcal{V}_n$ refines $\{\mathcal{W}_{n+1}^l\}$ for $l \leq n+1$.} 
\end{proof}

Now we are ready to prove the following approximation result for systems of measures, which generalises \cite[Theorem 6.9]{austin-georgescu} where they studied the case of Haar system.
\begin{Proposition}\label{prop: amenability app}
Let $\G$ be a locally compact and $\sigma$-compact groupoid, and $\{\mathcal{W}_n^l\}_{n \in \mathbb{N}}$ be a normal sequence of open covers on $\G$ for each $l \in \mathbb{N}$. Suppose for each $k \in \mathbb{N}$, $m^{(k)} =\{m^{(k)}_x:x\in \Gz\}$ is an $s$-system of measures on $\G$. Then the metrisable quotient $\G_\alpha$ in Proposition \ref{Prop: AG approximation 6.8} can be arranged so that for each $k \in \mathbb{N}$, there is an $\bar{s}$-system of measures $\overline{m}^{(k)} =\{\overline{m}^{(k)}_{\bar{x}}:\bar{x}\in \G^{(0)}_\alpha\}$ (here $\bar{s}$ is the source map of $\G_\alpha$) and the quotient map $q_\alpha$ is measure preserving.

\ignore{for the map
$$(q_\alpha)^*: C_c(\G_\alpha) \longrightarrow C_c(\G), \quad f \mapsto f \circ q_\alpha$$
and each $x\in \Gz$, the restriction of $(q_\alpha)^*$ on the $x$-fibre extends to an isometry
$$(q_\alpha)^*_x: L^2((\G_\alpha)_{q_\alpha(x)}; \overline{m^{(0)}}_{q_\alpha(x)}) \longrightarrow L^2(\G_x; m^{(0)}_x).$$}
\end{Proposition}

\begin{proof}
Fix a groupoid exhaustion $\{K_n\}$ of $\G$, and assume $\{\mathcal{U}_n\}$ is the groupoid normal sequences obtained in Lemma \ref{coverings}. Given $f \in C_c(\G_\alpha)$ and following the argument in \cite[Theorem 6.9]{austin-georgescu}, we know that for any $k \in \mathbb{N}$ and $\varepsilon>0$, there exists an $n$ such that if $x,y \in U \in \mathcal{U}_n^0$, then
$$\Big|\int (q_\alpha)^*(f) \d m^{(k)}_x - \int (q_\alpha)^*(f) \d m^{(k)}_y \Big| < \varepsilon.$$
Hence $q_\alpha(x)=q_\alpha(y)$ implies $\int (q_\alpha)^*(f) \d m^{(k)}_x = \int (q_\alpha)^*(f) \d m^{(k)}_y$. For each $\bar{x} \in \G_\alpha$, take a pre-image $x\in \G$, then the positive linear functional on $C_c(\G_\alpha)$ by $f \mapsto \int_{\G} (q_\alpha)^*(f) \d m^{(k)}_x$ provides a unique positive Radon measure $\overline{m}^{(k)}_{\bar{x}}$ on $\G_\alpha$ supported on $(\G_\alpha)_{\bar{x}}$ for each $k \in \mathbb{N}$, by Riesz-Markov-Kakutani Theorem. From the same argument as in \cite{austin-georgescu}, it is routine to check that $\overline{m}^{(k)}$ satisfies the requirements.
\end{proof}

As a direct corollary, we obtain the following:
\begin{Corollary}\label{cor: amenability app}
Let $\G$ and $\{\mathcal{W}_n^l\}_{n \in \mathbb{N}}$ be as above. Suppose $\{m^{(k)}\}$ is a topological approximation invariant mean on $\G$, then $\{\overline{m}^{(k)}\}$ in the above proposition forms a topological approximation invariant mean on $\G_\alpha$. Consequently, when $\G$ is amenable, the approximation $\G_\alpha$ can be made amenable as well.
\end{Corollary}

Also note that we recover \cite[Theorem 6.9]{austin-georgescu} as a special case when taking the Haar system. Furthermore in the \'{e}tale case, we may arrange the quotients to satisfy the following stronger condition. Note that it was already proved in \cite{austin-georgescu} that the approximation $\G_\alpha$ can be made \'{e}tale in this case.

\begin{Proposition}\label{prop: etale app}
Let $\G$ be a locally compact, $\sigma$-compact and \'{e}tale groupoid, and $\{\mathcal{W}_n^l\}_{n \in \mathbb{N}}$ be a normal sequence of open covers on $\G$ for each $l \in \mathbb{N}$. Then the metrisable quotient $\G_\alpha$ in Proposition \ref{Prop: AG approximation 6.8} can be arranged so that $\G_\alpha$ is \'{e}tale, and the quotient map $q_\alpha$ induces a bijection between fibres $\G_x$ and $(\G_\alpha)_{q_\alpha(x)}$ for any $x\in \Gz$.
\end{Proposition}

\begin{proof}
Following \cite[Section 6.1.3]{austin-georgescu}, recall that a groupoid $\G$ is \'{e}tale if and only if $\Gz$ is open in $\G$ and $\G$ admits a Haar system. Since $\G$ is \'{e}tale, there exists an open cover $\mathcal{V}$ of $\G$ consisting of bisections (recall that a \emph{bisection} is a subset $V$ that the source and range maps are homeomorphisms restricted on $V$). Since $\Gz$ is clopen, we may also assume that any element $V$ in $\mathcal{V}$ has the property that either $V \subseteq \Gz$ or $V \subseteq (\G \setminus \Gz)$. Hence by Lemma \ref{lemma: paracompact}, we may shrink each $\mathcal{W}_n^l$ to ensure that $\mathcal{W}_0^l$ refines $\mathcal{V}$.

Now applying Proposition \ref{prop: amenability app} to a Haar system, we obtain a groupoid normal sequence $\mathcal{U}_n$ of open covers of $\G$ satisfying the conditions therein. Consequently, $\G_\alpha$ has an induced Haar system. Note that each $\mathcal{U}_n$ refines $\mathcal{V}$, so $\G_\alpha^{(0)}$ is also open in $\G_\alpha$. Hence we obtian that $\G_\alpha$ is \'{e}tale as well. Moreover, since different points on each fibre cannot be identified under the quotient map (this is due to the fact that $\mathcal{U}_n$ consists of bisections for all $n$), so the restriction $q_\alpha|_{\G_x}: \G_x \to (\G_\alpha)_{q_\alpha(x)}$ is injective for any $x\in \Gz$.

As for surjection, note that $q_\alpha$ is measure preserving, hence $q_\alpha(x)=q_\alpha(y)$ implies $\sum_{\gamma \in \G_x} ((q_\alpha)_\ast f)(\gamma) = \sum_{\gamma \in \G_y} ((q_\alpha)_\ast f)(\gamma)$ for any $f \in C_c(\G_\alpha)$. Consequently, we obtain that $q_\alpha(\G_x) = q_\alpha(\G_y)$ provided $q_\alpha(x)=q_\alpha(y)$. Note that
$$(\G_\alpha)_{q_\alpha(x)}=\bigcup_{y \in q_\alpha^{-1}(q_\alpha(x))} q_\alpha(\G_y),$$
so the result holds.
\end{proof}

It is not hard to see from the above proofs that Corollary \ref{cor: amenability app}, Proposition \ref{prop: etale app} and \cite[Theorem 6.9]{austin-georgescu} can be dealt with simultaneously. In fact, we may apply Proposition \ref{prop: amenability app} for the sequence $\{m^{(k)}\}$ consisting of the Haar system and the topological approximation invariant mean at the same time. This is the reason for how we design Proposition \ref{prop: amenability app}. Consequently, we obtain the following:

\begin{Proposition}\label{prop: amenability and etale app}
Let $\G$ be a locally compact, $\sigma$-compact and amenable groupoid with a Haar system. Let $\{\mathcal{W}_n^l\}_{n \in \mathbb{N}}$ be a normal sequence of open covers of $\G$ for each $l$. Then the metrisable quotient $\G_\alpha$ obtained in Proposition \ref{Prop: AG approximation 6.8} can be arranged so that $\G_\alpha$ is amenable, and the Haar system on $\G$ can be pushforwarded to a Haar system on $\G_\alpha$ such that the quotient map $q_\alpha$ is measure preserving. Additionally if $\G$ is \'{e}tale, then $\G_\alpha$ can be arranged to be \'{e}tale as well and $q_\alpha$ induces bijections between the corresponding fibres.
\end{Proposition}

Finally we focus on the associated groupoid $C^*$-algebras. Roughly speaking, as the given groupoid $\G$ is the inverse limit of the metrisable quotients $\G_\alpha$, the groupoid $C^*$-algebra $C^*(\G)$ is the inductive limit of $C^*(\G_\alpha)$. However, in order to prove Theorem \ref{thm:extention of Exel}, we need to show that $C^*(\G)$ is exactly the union of the images of $C^*(\G_\alpha)$, which might be anti-intuitive at first glance. The following proposition (which is a minor adaptation to \cite[Proposition 5.11]{austin-georgescu}) will give us the necessary equipment for this to happen.

\begin{Lemma}\label{strongcondition}
Let $X$ be a locally compact and $\sigma$-compact space, and $\{X_\alpha,q_\alpha^\beta,A\}$ be an inverse approximation of $X$ by metric spaces. For every $\alpha$, denote $q_\alpha$ the canonical projection map $X\to X_\alpha$, and for each $n$, let $\mathcal{V}^\alpha_n$ be the cover of $X$ given by the pre-images of the cover of $X_\alpha$ by $\frac{1}{2^n}$-balls. Suppose for every countable family of normal sequences $\{\{\mathcal{W}^l_n\}_{n \in \mathbb{N}}: l \in \mathbb{N} \}$ in $X$, there exists an $\alpha$ such that the normal sequence $\{\mathcal{V}^\alpha_n\}_{n\in \mathbb{N}}$ cofinally refines $\{\mathcal{W}^l_n\}_{n \in \mathbb{N}}$ for every $l \in \mathbb{N}$. Then for any countable family of continuous functions $f_l:X\to Y_l$ where $Y_l$ is a second countable metric space, there exists an $\alpha$ such that each $f_l$ is the pullback of a uniformly continuous function $f_{l,\alpha}: X_\alpha\to Y_l$, i.e. $f_l = f_{l,\alpha}\circ q_\alpha$.
\end{Lemma}

\begin{proof}
Notice that each $f_l$ induces a normal sequence $\{\mathcal{W}_n^l\}_{n\in \mathbb{N}}$ of open covers on $X$ by taking the pre-image of the normal sequence of open covers of $Y_l$ given by $\frac{1}{2^n}$-balls. Let $\alpha\in A$ be provided from the hypothesis with respect to the family $\{\mathcal{W}_n^l\}_{n \in \mathbb{N}}$, such that the normal sequence $\{\mathcal{V}^\alpha_n\}_{n\in \mathbb{N}}$ cofinally refines all $\{\mathcal{W}_n^l\}_{n \in \mathbb{N}}$. It follows that if $q_\alpha(x) = q_\alpha(y)$ for $x,y \in X$, then $f_l(x) = f_l(y)$ for all $l\in \mathbb{N}$. Thus we may define $f_{l,\alpha}:X_\alpha\to Y_l$ by $f_{l,\alpha}(q_\alpha(x)) = f_l(x)$. It is routine to check that $f_{l,\alpha}$ is continuous for each $l\in \mathbb{N}$ by the choice of $\alpha$.
\end{proof}

The above lemma suggests the following notion:
\begin{Definition}
An inverse approximation $\{X_\alpha,q_\alpha^\beta,A\}$ of a uniform space $X$ by metric spaces is called a \emph{strong inverse approximation}, if for every countable family of normal sequences $\{\{\mathcal{W}^l_n\}_{n \in \mathbb{N}}: l \in \mathbb{N} \}$ of uniform covers, there exists an $\alpha$ such that the normal sequence $\{\mathcal{V}^\alpha_n\}_{n\in \mathbb{N}}$ obtained as in Lemma \ref{strongcondition} cofinally refines $\{\mathcal{W}^l_n\}_{n \in \mathbb{N}}$ for every $l \in \mathbb{N}$.
\end{Definition}

Consequently, we are ready to prove the following strengthened version of the approximation result:

\begin{Theorem}\label{sexyapproximationtheorem}
Let $\G$ be a locally compact, $\sigma$-compact and amenable groupoid with a Haar system. Then there exists a strong inverse approximation $\{\G_\alpha,q_\alpha^\beta,A\}$ of $\G$ by locally compact, second countable, amenable and metrisable groupoids with Haar systems such that the quotient $q_\alpha: \G \to \G_\alpha$ is Haar measure preserving. Moreover, for the induced direct system of topological star algebras $\{C_c(\G_\alpha),(q_\alpha^\beta)^*,A\}$, we have that $C_c(\G) = \bigcup_{\alpha}(q_\alpha)^*C_c(\G_\alpha)$ and $C_{\max}^*(\G) = \bigcup_\alpha (q_\alpha)^*C_{\max}^*(\G_\alpha)$.
\end{Theorem}

\begin{proof}
We follow the proof of \cite[Theorem 6.10]{austin-georgescu}, and let $A$ denote all groupoid normal sequences that satisfy the condition in Lemma \ref{coverings}. By Proposition \ref{prop: amenability and etale app}, we only need to prove the last statement.

Let $a\in C_{\max}^*(\G)$ and $a_n\in C_c(\G)$ such that $a$ is the $\max$-norm limit of the $a_n$'s. Since the approximation is a strong groupoid approximation, the conclusion of Lemma \ref{strongcondition} holds. Thus there exists $\alpha \in A$ such that $a_n\in (q_\alpha)^*\big(C_c(\G_\alpha)\big)$. As the map $(q_\alpha)^*: C_{\max}^*(\G_\alpha) \to C_{\max}^*(\G)$ is a norm-preserving from \cite[Proposition 3.2]{austin-georgescu}, we may take the pre-image $\widetilde{a_n}$ of $a_n$ under $(q_\alpha)^*$, and $\{\widetilde{a_n}\}$ is a Cauchy sequence in $C_{\max}^*(\G_\alpha)$. Assume that $\{\widetilde{a_n}\}$ converges to some element $\tilde{a}$ in $C_{\max}^*(\G_\alpha)$. Then we have
$$a= \lim_{n \to \infty} a_n = \lim_{n \to \infty} (q_\alpha)^*(\widetilde{a_n}) = (q_\alpha)^*(\tilde{a}).$$
It follows that $C_{\max}^*(\G) = \bigcup_\alpha (q_\alpha)^*C_{\max}^*(\G_\alpha)$, so we finish the proof.
\end{proof}

\begin{proof}[Proof of Theorem \ref{thm:extention of Exel}]
Suppose $\{\mu_x:x \in \Gz\}$ is the given Haar system on $\G$. By Theorem \ref{sexyapproximationtheorem}, we may take a strong inverse approximation $\{\G_\alpha,q_\alpha^\beta,A\}$ of $\G$ by locally compact, second countable and amenable groupoids with Haar systems $\{\mu_{\bar{x}}:\bar{x} \in \G_\alpha^{(0)}\}$ such that the quotient map $q_\alpha: \G \to \G_\alpha$ is measure preserving and induces an isometric embedding
$$U_x: L^2((\G_\alpha)_{q_\alpha(x)};\mu_{q_\alpha(x)}) \longrightarrow L^2(\G_x; \mu_x).$$
Furthermore, for the given $a \in C^*_{\max}(\G)$, there exists $\alpha \in A$ such that $a=(q_\alpha)^*(\tilde{a})$ for some $\tilde{a} \in C^*_{\max}(\G_\alpha)$. Now fix $x\in \Gz$ and denote $\bar{x}:=q_\alpha(x) \in \G^{(0)}_\alpha$. By assumption, we know that $1+\lambda_x(a)$ is invertible in $\B(L^2(\G_x; \mu_x))$, and we would like to show that $1+\lambda_{\bar{x}}(\tilde{a})$ is invertible in $\B(L^2((\G_\alpha)_{\bar{x}};\mu_{\bar{x}}))$. In order to do so, we make the following claim:

$\underline{Claim}:$ The subspace $\mathrm{Im} U_x$ is reduced with respect to $\lambda_x(a)$, i.e., both $\mathrm{Im} U_x$ and its orthogonal complement are invariant under $\lambda_x(a)$. Furthermore, $\lambda_x(a)$ coincides with $U_x \lambda_{\bar{x}}(\tilde{a}) U_x^*$ on $\mathrm{Im} U_x$.

Before we prove the claim, first let us use it to finish the proof of the theorem. It is implied by the claim that $1+\lambda_x(a)$ has a block diagonal form with respect to the following decomposition: 
$$L^2(\G_x; \mu_x) = \mathrm{Im} U_x \oplus (\mathrm{Im} U_x)^{\perp}.$$
And the block corresponding to $\mathrm{Im} U_x$ is unitary equivalent to $1+\lambda_{\bar{x}}(\tilde{a})$ via the unitary $U_x$ onto its image. Hence $1+\lambda_{\bar{x}}(\tilde{a})$ is invertible in $\B(L^2((\G_\alpha)_{\bar{x}};\mu_{\bar{x}}))$, and we may apply Proposition \ref{thm:NP extension of Exel} to $\G_\alpha$ directly since it is a locally compact, second countable and amenable groupoid with a Haar system of measures. Therefore, we we obtain that $1+\tilde{a}$ is invertible in $C^*_{\max}(\G_\alpha)$. Since $(q_\alpha)^*: C^*_{\max}(\G_\alpha) \to C^*_{\max}(\G)$ is an isometric embedding, we obtain that $1+a=(q_\alpha)^*(1+\tilde{a})$ is invertible in $C^*_{\max}(\G)^+ = C^*_r(\G)^+$ as well. This finishes the proof, so long as the claim is correct.

To prove the claim: first consider $a=(q_\alpha)^*(\tilde{a})\in C_c(\G)$ where $\tilde{a} \in C_c(\G_\alpha)$. For any $\xi = U_x(\tilde{\xi})$ with $\tilde{\xi} \in L^2((\G_\alpha)_{\bar{x}};\mu_{\bar{x}})$ and any $\gamma \in \G_x$, we calculate:
\begin{eqnarray*}
\big(\lambda_x (a)\big)(\xi)(\gamma) & = & \int_{\G_x} a(\gamma \delta^{-1}) \xi(\delta)d\mu_x(\delta) \\
&= &\int_{\G_x} \tilde{a}(q_\alpha(\gamma)q_\alpha(\delta)^{-1})\tilde{\xi}(q_\alpha(\delta))d\mu_x(\delta) \\
&=& \int_{(\G_\alpha)_{\bar{x}}} \tilde{a}(q_\alpha(\gamma)\bar{\delta}^{-1})\tilde{\xi}(\bar{\delta})d\mu_{\bar{x}}(\bar{\delta})\\
&=& U_x \left( \lambda_{\bar{x}}(\tilde{a})(\tilde{\xi})\right)(\gamma),
\end{eqnarray*}
where we use the fact the $q_\alpha$ is measure preserving in the third inequality. This implies that 
$$\lambda_x (a) \circ U_x = U_x \circ \lambda_{\bar{x}}(\tilde{a}).$$
By extension, the above holds for $a=(q_\alpha)^*(\tilde{a}) \in q_{\alpha}^*(C^*_{\max}(\G_\alpha))$ where $\tilde{a} \in C^*_{\max}(\G)$. Note that the algebra $(q_{\alpha})^*(C^*_{\max}(\G_\alpha))$ is a $C^*$-algebra, hence the claim holds and we finish the proof.
\end{proof}

\begin{Remark}
	Notice that by Proposition \ref{prop: etale app}, the above proof simplifies considerably when the groupoid $\G$ is \'{e}tale since $U_x: L^2((\G_\alpha)_{q_\alpha(x)};\mu_{q_\alpha(x)}) \longrightarrow L^2(\G_x; \mu_x)$ is an isometric isomorphism and, furthermore, $U_x$ establishes a unitary equivalence between the representations $\lambda_x$ and $\lambda_{q_\alpha(x)}$ on $C^*_r(\G_\alpha)$.
\end{Remark}

\section{Applications}\label{application section}

In Section \ref{main theorem section}, we established the limit operator theory for groupoids (Theorem \ref{main theorem}). Now we provide various applications to different groupoids, recovering the classic limit operator theory in the Hilbert space case for exact groups \cite{roe-band-dominated, rabinovich2012limit}, for spaces with Property A \cite{Spakula-Willett--limitoperators}, and also establishing new limit operator theories for amenable group actions, uniform Roe algebras for certain groupoids, and amenable groupoid actions.

\subsection{Discrete group case}\label{group case}
Our first application is to recover the limit operator theory for exact groups studied by Roe \cite{roe-band-dominated}, which includes the classic limit operator theory of Rabinovich, Roch and Silbermann \cite{rabinovich2012limit} in the Hilbert space case.

Let $G$ be a finitely generated exact discrete group, and $\beta G$ be its Stone-\v{C}ech compactification. Clearly, the left multiplication of $G$ on itself extends naturally to an action of $G$ on $\beta G$ by homeomorphisms. We consider the associated transformation groupoid $\G:=\beta G \rtimes G$. Its unit space $\Gz=\beta G$ can be decomposed into
\begin{equation}\label{dec for beta G}
\beta G=G \sqcup \partial G,
\end{equation}
where $\partial G$ is the Stone-\v{C}ech boundary. Clearly, $G$ is an open dense invariant subset in $\beta G$ and we have
$$\G(G)=G \rtimes G, \quad \mathrm{and} \quad \G(\partial G) = \partial G \rtimes G.$$

Since $G$ is exact, the groupoid $\beta G \rtimes G$ is topologically amenable (see Example \ref{group ex prlm}). Therefore we may apply Theorem \ref{main theorem} to the groupoid $\G=\beta G \rtimes G$ together with the decomposition (\ref{dec for beta G}), and obtain the associated limit operator theory for $\G$. Now we would like to provide a more detailed picture.

First we focus on the groupoid $C^*$-algebra $C^*_r (\G)=C^*_r(\beta G \rtimes G)$. Equipping $G$ with a proper right-invariant metric, we write $|G|$ to denote the associated metric space. For any $\gamma \in G$, let $L_\gamma$ and $R_\gamma$ be the unitary operators on $\ell^2(G)$ induced by left and right multiplication by $\gamma$. Denote $\mathbb{C}[|G|]$ the $\ast$-subalgebra in $\B(\ell^2(G))$ generated by the unitaries $L_\gamma$ and the diagonal matrices $\ell^\infty(G)$, and $C^*_u(|G|)$ its norm closure in $\B(\ell^2(G))$. For coarse geometers, elements in $\mathbb{C}[|G|]$ are said to have \emph{finite propagation}, and $C^*_u(|G|)$ is called \emph{the uniform Roe algebra of $G$}.

It is not hard to see that the groupoid $C^*$-algebra $C^*_r (\G)=C^*_r(\beta G \rtimes G)$ is isomorphic to the uniform Roe algebra $C^*_u(|G|)$. More precisely, we have the following algebraic $\ast$-isomorphism $\theta: C_c(\beta G \rtimes G) \to \mathbb{C}[|G|]$ defined by
$$\theta(f)(\gamma,\gamma_1):=f(\gamma,\gamma\gamma_1^{-1}).$$
One can verify that $\theta$ is an isometry with respect to the reduced norm on $C_c(\beta G \rtimes G)$ and the operator norm on $\mathbb{C}[|G|]$ (see for example \cite{RoeLectures}). Hence $\theta$ can be extended to an isomorphism
\begin{equation}\label{iso group case}
\Theta: C^*_r(\beta G \rtimes G) \to C^*_u(G).
\end{equation}
Furthermore, it is not hard to see that $\Theta(C^*_r(G \rtimes G))= \K(\ell^2(G))$. Consequently for $T \in C^*_r(\beta G \rtimes G)$, $T$ is invertible modulo $C^*_r(G \rtimes G)$ \emph{if and only if} $\Theta(T)$ is Fredholm.

Now we move on to the discussion of limit operators. Let us first recall the classic definition of limit operators introduced by Roe. Following \cite{roe-band-dominated}, given $T \in C^*_u(|G|)$, the map
$$G \to C^*_u(|G|), \quad \gamma \mapsto R_\gamma^*TR_\gamma$$
has range in a WOT-compact set and hence extends uniquely to a WOT-continuous map
$$\sigma(T): \partial G \to C^*_u(|G|).$$
So we obtain a $C^*$-homomorphism $\sigma: C^*_u(G) \to C_{\mathrm{WOT}}(\partial G, C^*_u(|G|))$. Furthermore, as Willett pointed out in his thesis \cite{willett2009band}, the image of $\sigma$ sits in the set of $G$-equivariant functions $C_{\mathrm{WOT}}(\partial G, C^*_u(|G|))^G$ in the following sense: a function $F \in C_{\mathrm{WOT}}(\partial G, C^*_u(|G|))$ is called \emph{$G$-equivariant} if $F(\gamma \omega) = R_{\gamma}^* F(\omega) R_{\gamma}$ for any $\gamma\in G$ and $\omega \in \partial G$. Therefore, we obtain a $C^*$-homomorphism, also called the \emph{symbol morphism}\footnote{As we will see in Proposition \ref{prop: group case}, this coincides with the symbol morphism we defined in Definition \ref{limit morphism defn} up to isomorphisms. The same happens in the next four subsections and we will not explain anymore},
$$\sigma: C^*_u(G) \to C_{\mathrm{WOT}}(\partial G, C^*_u(|G|))^G.$$
For any $\omega \in \partial G$, the \emph{limit operator of $T$ at $\omega$} is defined to be $\sigma_\omega(T):=\sigma(T)(\omega)$\footnote{As we will see in Lemma \ref{lem3 grp case}, this coincides with the limit operator we defined in Definition \ref{defn: limit op} up to isomorphisms. The same happens in the next four subsections and we will not explain anymore}. It is obvious that the map $T \mapsto \sigma_\omega(T)$ is an endomorphism of $C^*_u(G)$.

On the other hand, we have the symbol morphism associated to $\G=\beta G \rtimes G$ from Definition \ref{limit morphism defn}:
$$\varsigma: C^*_r(\G) \longrightarrow \Gamma_b(E_u^\partial)^{\partial G \rtimes G}.$$
Now we would like to compare the above two morphisms $\sigma$ and $\varsigma$, and show that they coincide under the isomorphism $\Theta$. Fix $\omega \in \partial G \subset \Gz$, and note that
$$\G_\omega=\{(g\omega,g)\in \G: g \in G\}$$
is bijective to $G$ via the map $(g\omega,g) \mapsto g$. This induces a unitary $U_\omega: \ell^2(\G_\omega) \cong \ell^2(G)$, and hence an isomorphism $\mathrm{Ad}_{U_\omega}: \B(\ell^2(\G_\omega)) \cong \B(\ell^2(G))$. These isomorphisms can be put together and provide a map between the operator fibre spaces as follows:
$$\mathrm{Ad}_U=\bigsqcup_{\omega \in \partial G}\mathrm{Ad}_{U_\omega}: E^\partial=\bigsqcup_{\omega \in \partial G}\B(\ell^2(\G_\omega)) \longrightarrow \partial G \times \B(\ell^2(G)),$$
where the final item is regarded as a trivial fibre space over $\partial G$, and the topology coincides with  the product topology. It is not hard to check that $\mathrm{Ad}_U$ is a homeomorphism, hence an isomorphism between operator fibre spaces over $\partial G$. Note that continuous sections of $\partial G \times \B(\ell^2(G))$ coincides with $C_{\mathrm{WOT}}(\partial G, \B(\ell^2(G)))$, hence $\mathrm{Ad}_U$ induces the following $C^*$-isomorphism:
$$(\mathrm{Ad}_U)_*: \Gamma(E^\partial) \longrightarrow \Gamma(\partial G \times \B(\ell^2(G))) = C_{\mathrm{WOT}}(\partial G, \B(\ell^2(G))).$$
Note that elements in $C_{\mathrm{WOT}}(\partial G, \B(\ell^2(G)))$ have uniformly bounded norms, hence $\Gamma(E^\partial) = \Gamma_b(E^\partial)$.

\begin{Lemma}\label{lem 1 grp case}
$(\mathrm{Ad}_U)_* \big(\Gamma(E^\partial)^{\partial G \rtimes G} \big)= C_{\mathrm{WOT}}(\partial G, \B(\ell^2(G)))^G$.
\end{Lemma}

\begin{proof}
Given $\varphi \in \Gamma(E^\partial)$, by definition $\varphi$ is $\partial G \rtimes G$-equivariant if and only if for any $(g\omega,g) \in \partial G \rtimes G$, we have
$$\varphi(g\omega) = R_{(g\omega,g)}^* \varphi(\omega) R_{(g\omega,g)}.$$
Note that here $R_{(g\omega,g)}: \ell^2(\G_{g\omega}) \to \ell^2(\G_{\omega})$ is defined by $\delta_{(hg\omega,h)} \mapsto \delta_{(hg\omega,hg)}$, which induces the following commutative diagram:
\begin{displaymath}
    \xymatrix@=4em{
        \ell^2(\G_{g\omega}) \ar[r]^-{\textstyle R_{(g\omega, g)}} \ar[d]_-{\textstyle U_{g\omega}}^-{\textstyle\cong} & \ell^2(\G_{\omega}) \ar[d]^-{\textstyle U_{\omega}}_-{\textstyle\cong} \\
        \ell^2(G) \ar[r]^-{\textstyle R_g} & \ell^2(G)}
\end{displaymath}
Therefore, we have $\big((\mathrm{Ad}_U)_*\varphi\big)(g\omega) = R_g^* \big((\mathrm{Ad}_U)_*\varphi\big)(\omega) R_g$. So the lemma holds.
\end{proof}

\begin{Lemma}\label{lem2 grp case}
For each $\omega \in \partial G$ and $g \in G$, we have $\mathrm{Ad}_{U_\omega}\big(C^*_u(\G_{\omega})\big) = C^*_u(|G|)$. Hence we have $(\mathrm{Ad}_U)_* \big(\Gamma(E^\partial_u)^{\partial G \rtimes G} \big)= C_{\mathrm{WOT}}(\partial G, C^*_u(|G|))^G$.
\end{Lemma}

\begin{proof}
Recall that by definition, $T=(T_{(g'\omega,g'), (g''\omega, g'')})_{g',g'' \in G} \in \B(\ell^2(\G_x))$ belongs to $\mathbb{C}[\G_{\omega}]$ if and only if there exists a compact set $K \subseteq \beta G \rtimes G$ such that $T_{(g'\omega,g'), (g''\omega, g'')} \neq 0$ implies that $(g''\omega, g'')(g'\omega, g')^{-1} \in K$. Since $\beta G$ is compact and $(g''\omega, g'')(g'\omega, g')^{-1}=(g''\omega, g''g'^{-1})$, the above is equivalent to the condition that there exists a finite set $F \subseteq G$ such that $\mathrm{Ad}_{U_\omega}(T)_{g',g''} \neq 0$ implies that $g''g'^{-1} \in F$. In other words, we have $\mathrm{Ad}_{U_\omega}\big(\mathbb{C}[\G_{\omega}]\big) = \mathbb{C}[|G|]$. Taking completions on both sides, the lemma holds.
\end{proof}

\begin{Lemma}\label{lem3 grp case}
For $\omega \in \partial G$, the following diagram commutes.
\begin{displaymath}
    \xymatrix@=3em{
        C^*_r(\beta G \rtimes G) \ar[r]^-{\textstyle \lambda_\omega} \ar[d]_-{\textstyle\Theta}^-{\textstyle\cong} & C^*_u(\G_\omega) \ar[d]^-{\textstyle \mathrm{Ad}_{U_\omega}}_-{\textstyle\cong} \\
        C^*_u(|G|) \ar[r]^-{\textstyle \sigma_\omega} & C^*_u(|G|)}
\end{displaymath}
\end{Lemma}

\begin{proof}
It suffices to show that $f\in C_c(\beta G \rtimes G)$, we have $\mathrm{Ad}_{U_\omega} \circ \lambda_\omega (f)=\sigma_\omega\circ\Theta (f)$. For any $g,h\in G$, we have
$$\big(\lambda_\omega (f)\delta_{(h\omega,h)}\big)(g\omega,g)=\sum_{\gamma\in G}f((g\omega,g)(\gamma\omega,\gamma)^{-1})\delta_{(h\omega,h)}(\gamma\omega,\gamma)=f(g\omega,gh^{-1}).$$
Hence,
\begin{eqnarray*}
\langle \mathrm{Ad}_{U_\omega} \circ \lambda_\omega (f)\delta_h, \delta_g\rangle &=& \langle \lambda_\omega (f)\delta_{(h\omega,h)}, \delta_{(g\omega,g)}\rangle = f(g\omega,gh^{-1})=\lim_{\alpha \to \omega}f(g\alpha,gh^{-1}), \\
\langle \sigma_\omega\circ\Theta (f)\delta_h, \delta_g\rangle &=& \lim_{\alpha \to \omega} \langle R_\alpha^*\Theta (f)R_\alpha\delta_h, \delta_g\rangle = \lim_{\alpha \to \omega}\langle \Theta (f)\delta_{h\alpha}, \delta_{g\alpha}\rangle=\lim_{\alpha \to \omega}f(g\alpha,gh^{-1}).
\end{eqnarray*}
So we finish the proof.
\end{proof}

Combining Lemma \ref{lem 1 grp case}, \ref{lem2 grp case} and \ref{lem3 grp case} together, we obtain the following:
\begin{Proposition}\label{prop: group case}
Notations as above. The following digram commutes:
\begin{displaymath}
    \xymatrix{
        C^*_{r}(\beta G \rtimes G) \ar[r]^-{\textstyle \varsigma} \ar[d]_-{\textstyle \Theta}^-{\textstyle \cong} & \Gamma_b(E_u^\partial)^{\partial G \rtimes G} \ar[d]^-{\textstyle (\mathrm{Ad}_U)_* }_-{\textstyle \cong} \\
        C^*_{u}(G) \ar[r]^-{\textstyle \sigma} & C_{\mathrm{WOT}}(\partial G, C^*_u(|G|))^G}
\end{displaymath}
\end{Proposition}

Consequently, we recover the classic limit operator theory for exact groups \cite{roe-band-dominated} as a corollary of Theorem \ref{main theorem}:
\begin{Corollary}\label{group case cor}
Let $G$ be a finitely generated exact discrete group. For any operator $T \in C^*_u(|G|)$, the following are equivalent:
\begin{enumerate}
  \item $T$ is Fredholm.
  \item $\sigma(T)$ is invertible in $C_{\mathrm{WOT}}(\partial G, C^*_u(|G|))^G$.
  \item For each $\omega \in \partial G$, the limit operator $\sigma_\omega(T)$ is invertible, and
  $$\sup_{\omega \in \partial G} \|\sigma_\omega(T)^{-1}\| < \infty.$$
  \item For each $\omega \in \partial G$, the limit operator $\sigma_\omega(T)$ is invertible.
\end{enumerate}
\end{Corollary}

\subsection{General compactifications of group case}\label{grp cptf ex}
In the previous subsection, we recover the classic limit operator theory for exact groups which characterises Fredholmness for operators in the uniform Roe algebra. And as we observed, the uniform Roe algebra of a group corresponds to the action on its Stone-\v{C}ech compactification. Now we study the limit operator theory associated to general compactifications recovering the results from \cite[Chapter 3.2]{willett2009band}.

Let $G$ be a finitely generated exact discrete group, and $Y$ be an \emph{equivariant compactification} of $G$ in the sense that $Y$ is a compact space containing $G$ as an open and dense subset, and the action of the left multiplication of $G$ on itself extends to an action on $Y$. As before, we consider the transformation groupoid $\G=Y \rtimes G$, whose unit space is $Y$ with a natural decomposition:
$$Y = G \sqcup \partial_Y G.$$
And we have
$$\G(G)=G \rtimes G, \quad \mathrm{and} \quad \G(\partial_Y G) = \partial_Y G \rtimes G.$$
Note that although $G$ is exact, the action on $Y$ is in general \emph{not} amenable, i.e., the groupoid $\G=Y \rtimes G$ might \emph{not} be amenable. In fact, when $Y$ is the Alexandroff one-point compactification, the amenability of the action is equivalent to the amenability of the group itself, and there exist non-amenable but exact groups (e.g. free groups). However we still have the following:
\begin{Lemma}\label{exactness lemma}
Let $G$ be a finitely generated exact discrete group and $Y$ be an equivariant compactification of $G$. Then the following short sequence is exact:
$$0 \longrightarrow C^*_r(G \rtimes G) \longrightarrow C^*_r(Y \rtimes G) \longrightarrow C^*_r(\partial_Y G \rtimes G) \longrightarrow 0.$$
\end{Lemma}

\begin{proof}
We have the following commutative diagram:
\begin{displaymath}
    \xymatrix{
        0 \ar[r] & C^*_r(G \rtimes G) \ar[r] \ar[d] & C^*_r(Y \rtimes G) \ar[r] \ar[d] & C^*_r(\partial_Y G \rtimes G) \ar[r] \ar[d] & 0\\
        0 \ar[r] & C_0(G)\rtimes_r G \ar[r] & C(Y) \rtimes_r G \ar[r] & C(\partial_Y G) \rtimes_r G \ar[r] & 0,}
\end{displaymath}
where all vertical maps are $C^*$-isomorphisms (this is well-known, and the readers can refer to Section \ref{group action ex} for details where we deal with general group actions). By \cite[Theorem 5.1.10]{brown-ozawa}, we know the bottom line is exact, hence so is the top one.
\end{proof}

Therefore by Remark \ref{main theorem remark}, we know that part of our main result: ``(1) $\Leftrightarrow$ (2) $\Leftrightarrow$ (3)" in Theorem \ref{main theorem} holds directly for the transformation groupoid $\G=Y \rtimes G$, which establishes part of the limit operator theory in this case. Now we would like to provide a more detailed picture and compare it with the limit operator theory studied in Section \ref{group case}.

First we focus on the groupoid $C^*$-algebra $C^*_r (\G)=C^*_r(Y \rtimes G)$. Equipping $G$ with a proper right-invariant metric, and denote the uniform Roe algebra by $C^*_u(|G|)$ as in Section \ref{group case}. From the universal property of the Stone-\v{C}ech compactification, there is a $G$-equivariant surjection $\phi:\beta G \to Y$, which induces an embedding $C_c(Y \rtimes G) \hookrightarrow C_c(\beta G \rtimes G)$ and hence a $C^*$-monomorphism
$$\phi^*: C^*_r(Y \rtimes G) \longrightarrow C^*_r(\beta G\rtimes G).$$
Recall that in (\ref{iso group case}) we provide an isomorphism $\Theta: C^*_r(\beta G \rtimes G) \to C^*_u(|G|)$. Combining them together, we have an embedding
$$\Theta \circ \phi^*: C^*_r(Y \rtimes G) \longrightarrow C^*_u(|G|),$$
whose image is denoted by $\mathcal{A}_Y$. Consequently, we obtain the following commutative diagram:
\begin{displaymath}
    \xymatrix@=3em{
        C^*_r(G \rtimes G) \ar[r] \ar[d]^-{\textstyle\cong} & C^*_r(Y \rtimes G) \ar[r]^-{\textstyle \phi^*} \ar[d]^-{\textstyle\cong} & C^*_r(\beta G \rtimes G) \ar[d]^-{\textstyle\Theta}_-{\textstyle\cong} \\
        \mathfrak{K}(\ell^2(G)) \ar@{^{(}->}[r] & \mathcal{A}_Y \ar@{^{(}->}[r] & C^*_u(|G|).}
\end{displaymath}

Note that for $T \in \mathcal{A}_Y \subseteq C^*_u(|G|)$, we already established the limit operator theory in Corollary \ref{group case cor}. Here we provide another intrinsic viewpoint to gain a more precise description of the limit operator theory. Recall that we have the symbol morphism associated to $\G=Y \rtimes G$ from Definition \ref{limit morphism defn}:
$$\varsigma_Y: C^*_r(Y \rtimes G) \longrightarrow \Gamma(E^\partial_{Y,u})^{\partial_Y G \rtimes G},$$
where $E^\partial_Y$ is the operator fibre space associated to the reduction groupoid $\partial_Y G \rtimes G$. Now we would like to compare $\varsigma_Y$ with $\varsigma, \sigma$ in Section \ref{group case}.

Fix $\omega \in \partial_Y G$ and note that $(Y \rtimes G)_\omega=\{(g\omega,g)\in Y \rtimes G: g \in G\}$ is bijective to $G$ via the map $(g\omega,g) \mapsto g$. This induces a unitary $U_{Y,\omega}: \ell^2((Y \rtimes G)_\omega) \cong \ell^2(G)$, and hence an isomorphism $\mathrm{Ad}_{U_{Y,\omega}}: \B(\ell^2((Y \rtimes G)_\omega)) \cong \B(\ell^2(G))$. These isomorphisms can be put together and provide a map between operator fibre spaces:
$$\mathrm{Ad}_{U_Y}=\bigsqcup_{\omega \in \partial_Y G}\mathrm{Ad}_{U_{Y,\omega}}: E^\partial_Y=\bigsqcup_{\omega \in \partial_Y G}\B(\ell^2((Y \rtimes G)_\omega)) \longrightarrow \partial_Y G \times \B(\ell^2(G)).$$
It is easy to check that $\mathrm{Ad}_{U_Y}$ is a homeomorphism, hence an isomorphism between operator fibre spaces over $\partial_Y G$. So $\mathrm{Ad}_{U_Y}$ induces the following $C^*$-isomorphism:
$$(\mathrm{Ad}_{U_Y})_*: \Gamma(E^\partial_Y) \longrightarrow C_{\mathrm{WOT}}(\partial_Y G, \B(\ell^2(G))).$$
Note that elements in $C_{\mathrm{WOT}}(\partial_Y G, \B(\ell^2(G)))$ have uniform bounded norms, hence $\Gamma(E^\partial_Y) = \Gamma_b(E^\partial_Y)$. As we did in Section \ref{group case}, we have
$$(\mathrm{Ad}_{U_Y})_*: \Gamma(E^\partial_{Y,u})^{\partial_Y G \rtimes G} \stackrel{\textstyle \cong}{\longrightarrow} C_{\mathrm{WOT}}(\partial_Y G, C^*_u(|G|))^G.$$
On the other hand, from Lemma \ref{lem2 grp case} we have the following isomorphism for $\beta G \rtimes G$:
$$(\mathrm{Ad}_U)_*: \Gamma(E^\partial_u)^{\partial G \rtimes G} \stackrel{\textstyle \cong}{\longrightarrow} C_{\mathrm{WOT}}(\partial G, C^*_u(|G|))^G.$$
Furthermore, the map $\phi: \beta G \to Y$ induces an injective $C^*$-morphism:
$$\phi_{WOT}^*: C_{\mathrm{WOT}}(\partial_Y G, C^*_u(|G|))^G \to C_{\mathrm{WOT}}(\partial G, C^*_u(|G|))^G.$$
\begin{Lemma}\label{commut diagram}
The following diagram commutes:
\begin{displaymath}
    \xymatrix@=4em{
        C^*_r(\beta G \rtimes G) \ar[r]^-{\textstyle \varsigma}  & \Gamma(E^\partial_u)^{\partial G \rtimes G} \ar[r]^-{(\mathrm{Ad}_U)_*}_-{\textstyle\cong}& C_{\mathrm{WOT}}(\partial G, C^*_u(|G|))^G  \\
        C^*_r(Y \rtimes G) \ar[r]^-{\textstyle \varsigma_Y}\ar[u]^-{\textstyle \phi^*} & \Gamma(E^\partial_{Y,u})^{\partial_Y G \rtimes G} \ar[r]^-{(\mathrm{Ad}_{U_Y})_*}_-{\textstyle\cong} &  C_{\mathrm{WOT}}(\partial_Y G, C^*_u(|G|))^G  \ar[u]^{\textstyle\phi_{WOT}^*}.}
\end{displaymath}
\end{Lemma}

\begin{proof}
Note that for $\omega \in \partial_Y G$ and any of its inverse $\widetilde{\omega}$ under $\phi$, the following diagram commutes:
\begin{displaymath}
    \xymatrix@=4em{
        C^*_r(\beta G \rtimes G) \ar[r]^-{\textstyle \lambda_{\widetilde{\omega}}}  & C^*_u((\beta G \rtimes G)_{\widetilde{\omega}}) \ar[r]^-{\textstyle \mathrm{Ad}_{U_{\widetilde{\omega}}}}_-{\textstyle\cong}& C^*_u(|G|) \ar@{=}[d] \\
        C^*_r(Y \rtimes G) \ar[r]^-{\textstyle \lambda_\omega}\ar[u]^-{\textstyle \phi^*} & C^*_u((Y \rtimes G)_\omega) \ar[r]^-{\textstyle \mathrm{Ad}_{U_{Y,\omega}}}_-{\textstyle\cong} &  C^*_u(|G|)}
\end{displaymath}
So the lemma holds.
\end{proof}

\begin{Proposition}\label{comm cube diag}
There exists a unique morphism $\sigma_Y: \mathcal{A}_Y \to C_{\mathrm{WOT}}(\partial_Y G, C^*_u(|G|))^G$ (also called the \emph{symbol morphism}) such that the following diagram commutes:
\begin{displaymath}
\xymatrix{
    C^*_r(\beta G \rtimes G)\ar[rr]^-{\textstyle \varsigma}\ar[dr]^-{\textstyle \Theta}_-{\textstyle\cong} & & \Gamma(E^\partial_u)^{\partial G \rtimes G}\ar[dr]^-{\textstyle (\mathrm{Ad}_U)_*}_-{\textstyle\cong}\\
    & C^*_u(|G|) \ar[rr]^-{\textstyle \sigma} & & C_{\mathrm{WOT}}(\partial G, C^*_u(|G|))^G\\
     C^*_r(Y \rtimes G) \ar[dr]_-{\textstyle\cong} \ar[uu]^-{\textstyle \phi^*}\ar[rr]^</1.3cm/{\textstyle \varsigma_Y}|</1.65cm/{\hole} & & \Gamma(E^\partial_{Y,u})^{\partial_Y G \rtimes G}\ar[dr]^-{\textstyle (\mathrm{Ad}_{U_Y})_*}_-{\textstyle\cong}\\
    & \mathcal{A}_Y \ar@{-->}[rr]^-{\textstyle \sigma_Y}\ar@{_{(}->}[uu] & & C_{\mathrm{WOT}}(\partial_Y G, C^*_u(|G|))^G  \ar[uu]_{\textstyle\phi_{WOT}^*}.}
\end{displaymath}
\end{Proposition}

\begin{proof}
Define $\sigma_Y = (\mathrm{Ad}_{U_Y})_* \circ \varsigma_Y \circ (\Theta \circ \phi^*)^{-1}: \mathcal{A}_Y \to C_{\mathrm{WOT}}(\partial_Y G, C^*_u(|G|))^G$. For any $\omega \in \partial_Y G$, take one of its inverses $\widetilde{\omega}$ under $\phi$. Then from Lemma \ref{commut diagram}, we have the following commutative diagram:

\begin{displaymath}
\xymatrix{
    C^*_r(\beta G \rtimes G)\ar[rr]^-{\textstyle \lambda_{\widetilde{\omega}}}\ar[dr]^-{\textstyle \Theta}_-{\textstyle\cong} & & C^*_u((\beta G \rtimes G)_{\widetilde{\omega}}) \ar[dr]^-{\textstyle \mathrm{Ad}_{U_{\widetilde{\omega}}}}_-{\textstyle\cong}\\
    & C^*_u(|G|) \ar[rr]^-{\textstyle \sigma_{\widetilde{\omega}}} & & C^*_u(|G|)\ar@{=}[dd]\\
     C^*_r(Y \rtimes G) \ar[dr]_-{\textstyle\cong} \ar[uu]^-{\textstyle \phi^*}\ar[rr]^</1.3cm/{\textstyle \lambda_\omega}|-{\hole} & & C^*_u((Y \rtimes G)_\omega) \ar[dr]^-{\textstyle \mathrm{Ad}_{U_{Y,\omega}}}_-{\textstyle\cong}\\
    & \mathcal{A}_Y \ar[rr]^-{\textstyle\sigma_{Y,\omega}}\ar@{_{(}->}[uu] & & C^*_u(|G|),
}
\end{displaymath}
where the map $\sigma_{Y,\omega}:\mathcal{A}_Y \to C^*_u(|G|)$ is defined by $\sigma_{Y,\omega}(T)= \sigma_Y(T)(\omega)$. So the result holds.
\end{proof}

\begin{Remark}
Note that the symbol morphism $\sigma_Y$ defined above is the same as the morphism $\sigma$ in \cite[Proposition 3.2.1]{willett2009band}.
\end{Remark}

Consequently, we obtain the limit operator theory for general equivariant compactifications:
\begin{Corollary}
Let $G$ be a finitely generated exact discrete group, and $Y$ be a $G$-equivariant compactification. Let $\phi: \beta G \to Y$ be the induced surjection, and $Z$ be a subset in $\partial G$ such that $\phi(Z)=\partial_Y G$. Suppose $\mathcal{A}_Y$ is the associated $C^*$-subalgebra of $C^*_u(|G|)$ defined above. Then for any $T \in \mathcal{A}_Y$, the following are equivalent:
\begin{enumerate}
  \item $T$ is Fredholm.
  \item $\sigma_Y(T)$ is invertible in $C_{\mathrm{WOT}}(\partial_Y G, C^*_u(|G|))^G$.
  \item For each $\widetilde{\omega} \in Z$, the limit operator $\sigma_{\widetilde{\omega}}(T)$ is invertible, and
  $$\sup_{\widetilde{\omega} \in Z} \|\sigma_{\widetilde{\omega}}(T)^{-1}\| < \infty.$$
  \item For each $\widetilde{\omega} \in Z$, the limit operator $\sigma_{\widetilde{\omega}}(T)$ is invertible.
\end{enumerate}
\end{Corollary}

\begin{proof}
From the above analysis, we know ``(1) $\Leftrightarrow$ (2) $\Leftrightarrow$ (3)" holds, so it suffices to prove ``(4) $\Rightarrow$ (3)". Unfortunately since the action of $G$ on $Y$ is not amenable in general, we cannot refer to Theorem \ref{main theorem} directly. However, from the proof of Proposition \ref{comm cube diag}, we know that for any $\widetilde{w}_1, \widetilde{w}_2 \in \beta G$ with $\phi(\widetilde{w}_1) = \phi(\widetilde{w}_2)$, we have that $\sigma_{\widetilde{\omega}_1}(T) = \sigma_{\widetilde{\omega}_2}(T)$. So condition (4) implies that for any $\widetilde{\omega} \in \beta G$, the limit operator $\sigma_{\widetilde{\omega}}(T)$ is invertible. Hence from ``(4) $\Rightarrow$ (3)" in Corollary \ref{group case cor}, we know that $\|\sigma_{\widetilde{\omega}}(T)^{-1}\|$ is uniformly bounded for all $\widetilde{\omega} \in \beta G$. As a special case, condition (3) holds.
\end{proof}

\begin{Remark}
Note that the key difference between Corollary \ref{group case cor} and the above is that in order to obtain the Fredholmness for operators in the subalgebra $\mathcal{A}_Y$, we do not need to check the invertibility of limit operators for all limit points in $\partial G$, but only those related to the given compactification ($Z$, in the previous corollary).
\end{Remark}

\begin{Example}
Suppose $G=\mathbb{Z}$, and $Y=\mathbb{Z} \cup \{\pm \infty\}$ be the two-point compactification of $\mathbb{Z}$ (more explicitly, the topology on $\mathbb{Z}$ is discrete, $\lim_{n \to +\infty} n=+\infty$ and $\lim_{n \to -\infty} n=-\infty$). In this case, the above algebra $\mathcal{A}_Y$ can be determined as follows:
$$\mathcal{A}_Y=\{T\in C^*_u(|\mathbb{Z}|): \lim_{n \to +\infty} T_{n,m+n} \mbox{~and~}\lim_{n \to -\infty} T_{n,m+n}\mbox{~exist~,~for~all~}m\in \mathbb{Z}\}.$$
And the limit operator of $T$ over $+\infty$ is the operator whose entries on the $m$-th diagonal
$$\{(m+n,n): n\in \mathbb{Z}\}$$
are all equal to $\lim_{n \to +\infty} T_{n,m+n}$. Similar formula also holds for the limit operator over $-\infty$. Under the Fourier transformation $\ell^2(\mathbb{Z})\cong L^2(S^1)$, these two limit operators correspond to multiplication operators by two continuous functions, denoted by $f^{+\infty}$ and $f^{-\infty}$ respectively. Hence the limit operator theory in the case says that $T$ is Fredholm \emph{if and only if} $f^{+\infty}, f^{-\infty}$ are invertible (i.e., point-wise nonzero). We remark that the Gohberg-Krein index theorem says that in this case, when $T$ is Fredholm, then
$$\mathrm{Index}(T)=-(\mbox{winding~number})(f^{-\infty}) - (\mbox{winding~number})(f^{+\infty}).$$
\end{Example}

\subsection{Discrete metric space case}\label{metric space ex}
The next application recovers (in the Hilbert space case) the limit operator theory for discrete metric spaces introduced by \v{S}pakula and Willett in \cite{Spakula-Willett--limitoperators}. In fact, part of the proof (excluding the omission of uniform bounded condition) using the groupoid language was already mentioned in \cite[Appendix C]{Spakula-Willett--limitoperators}. For the convenience to the readers, we recall the details again.

Let $(X,d)$ be a uniformly discrete metric space with bounded geometry and propert A. The \emph{coarse groupoid} $G(X)$ on $X$ was introduced by Skandalis, Tu and Yu in \cite{Skandalis-Tu-Yu} (also see \cite[Chapter 10]{RoeLectures}). As a set,
$$G(X):=\bigcup_{r>0}{\overline{E_r}}^{\beta X \times \beta X}$$
where $E_r:=\{(x,y) \in X\times X: d(x,y) \leq r\}$. The groupoid structure is just the restriction of the pair groupoid $\beta X \times \beta X$. On each $\overline{E_r}$, the topology of $G(X)$ agrees with the subspace topology of $\beta X \times \beta X$; and generally, a subset $U$ in $G(X)$ is open \emph{if and only if} the intersection $U \cap \overline{E_r}$ is open in $\overline{E_r}$ for each $r>0$. Equipped with this topology, $G(X)$ is a locally compact, $\sigma$-compact, principal and \'{e}tale groupoid. Clearly, the unit space of $G(X)$ is $\beta X$, and we consider the following decomposition:
\begin{equation}\label{dec for coarse groupoid}
\beta X=X \sqcup \partial X,
\end{equation}
where $\partial X$ is the Stone-\v{C}ech boundary. Obviously $X$ is an open invariant dense subset in $\beta X$, hence the above decomposition induces the following:
$$G(X) = (X \times X) \sqcup G_\infty(X),$$
where $G_\infty(X)$ is the reduction of $G(X)$ by $\partial X$.

Recall from \cite{Skandalis-Tu-Yu} that the space $X$ has Property A \emph{if and only if} the coarse groupoid $G(X)$ is topoligically amenable. Hence we may apply Theorem \ref{main theorem} to $\G=G(X)$ together with the decomposition (\ref{dec for coarse groupoid}), and obtain the associated limit operator theory. Now we would like to follow \cite[Appendix C]{Spakula-Willett--limitoperators} to provide a more detailed picture.

As before, we start with the reduced groupoid $C^*$-algebra $C^*_r(G(X))$. For any $f \in C_c(G(X))$, by definition $f$ is a continuous function supported in $\overline{E_r}$ for some $r>0$; equivalently, we may interpret $f$ as a bounded continuous function on $E_r$. Hence we can define an operator $\theta(f)$ on $\ell^2(X)$ by setting its matrix coefficients to be $\theta(f)_{x,y}:=f(x,y)$, where $x,y\in X$. And we have the following lemma:

\begin{Lemma}[\cite{RoeLectures}, Proposition 10.29]
The map $\theta$ provides a $\ast$-isomorphism from $C_c(G(X))$ to $\mathbb{C}[X]$, and extends to a $C^*$-isomorphism $\Theta$ from $C^*_r(G(X))$ to the uniform Roe algebra $C^*_u(X)$, which maps the $C^*$-subalgebra $C^*_r(X \times X)$ onto the compact operators $\mathfrak{K}(\ell^2(X))$. Consequently $T \in C^*_r(G(X))$ is invertible modulo $C^*_r(G \rtimes G)$ \emph{if and only if} $\Theta(T)$ is Fredholm.
\end{Lemma}

Now we move on to the discussion of limit operators. Let us first recall the notion of limit operators for metric spaces introduced by \v{S}pakula and Willett \cite{Spakula-Willett--limitoperators}. A function $t: D \to R$ with $D,R \subseteq X$ is called a \emph{partial translation} if $t$ is a bijection from $D$ to $R$, and $\sup_{x\in X}d(x,t(x))$ is finite.

\begin{Definition}[\cite{Spakula-Willett--limitoperators}]
Fix an ultrafilter $\omega \in \beta X$. A partial translation $t:D \to R$ on $X$ is \emph{compatible with $\omega$} if $\omega(D)=1$. In this case, regarding $t$ as a function from $D$ to $\beta X$, we define
$$t(\omega):=\lim_\omega t \in \beta X.$$
An ultrafilter $\alpha \in \beta X$ is \emph{compatible with $\omega$} if there exists a partial translation $t$ compatible with $\omega$ and $t(\omega)=\alpha$. Denote $X(\omega)$ the collection of all ultrafilters on $X$ compatible with $\omega$. A \emph{compatible family for $\omega$} is a collection of partial translations $\{t_\alpha\}_{\alpha \in X(\omega)}$ such that each $t_\alpha$ is compatible with $\omega$ and $t_\alpha(\omega)=\alpha$.
\end{Definition}

Fix an ultrafilter $\omega$ on $X$, and a compatible family $\{t_\alpha\}_{\alpha \in X(\omega)}$. Define a function $d_{\omega}: X(\omega) \times X(\omega) \to [0,\infty)$ by
$$d_{\omega}(\alpha, \beta):=\lim_{x\to \omega}d(t_\alpha(x),t_\beta(x)).$$
It is shown in \cite[Proposition 3.7]{Spakula-Willett--limitoperators} that $d_\omega$ is a uniformly discrete metric of bounded geometry on $X(\omega)$ which does not depend on the choice of $\{t_{\alpha}\}$.

\begin{Definition}[\cite{Spakula-Willett--limitoperators}]\label{limit operator defn for metric space}
For each non-principal ultrafilter $\omega$ on $X$, the metric space $(X(\omega),d_{\omega})$ is called the \emph{limit space} of $X$ at $\omega$. Fix a compatible family $\{t_\alpha\}_{\alpha \in X(\omega)}$ for $\omega$, and suppose $A \in C^*_u(X)$. The \emph{limit operator of $A$ at $\omega$}, denoted by $\Phi_\omega(A)$, is an $X(\omega)$-by-$X(\omega)$ indexed matrix defined by
$$\Phi_\omega(A)_{\alpha\beta}:=\lim_{x \to \omega}A_{t_{\alpha}(x)t_{\beta}(x)}.$$
\end{Definition}

It was studied in \cite[Chapter 4]{Spakula-Willett--limitoperators} that the above definition does not depend on the choice of compatible family $\{t_\alpha\}_{\alpha \in X(\omega)}$ for $\omega$. Furthermore, the limit operator $\Phi_\omega(A)$ is indeed a bounded operator on $\ell^2(X(\omega))$, and belongs to the uniform Roe algebra $C^*_u(X(\omega))$.

On the other hand, we have the symbol morphism associated to $\G=G(X)$ from Definition \ref{limit morphism defn}:
$$\varsigma: C^*_r(G(X)) \longrightarrow \Gamma_b(E_u^\partial)^{G_\infty(X)}.$$
Now we would like to compare $\varsigma$ with limit operators $\Phi_\omega(A)$ in Definition \ref{limit operator defn for metric space}, and show that they coincide under the isomorphism $\Theta$. In the process, we will comprise all limit operators into a single homomorphism in terms of the language of operator fibre spaces established in Section \ref{op fibre sp}.

\begin{Lemma}[Lemma C.3, \cite{Spakula-Willett--limitoperators}]\label{metric limit op}
Given $ \omega \in \partial X$, we define
$$F : X(\omega) \rightarrow G(X)_\omega, \alpha \mapsto (\alpha, \omega).$$
Then $F$ is a bijection. Let $W_\omega: \ell^2(G(X)_\omega) \to \ell^2(X(\omega))$ be the unitary operator induced by $F$. Then we have the following commutative diagram:
\begin{displaymath}
    \xymatrix@=3em{
        C^*_r(G(X)) \ar[r]^-{\textstyle \lambda_\omega} \ar[d]_-{\textstyle\Theta}^-{\textstyle\cong} & \B(\ell^2(G(X)_\omega)) \ar[d]^-{\textstyle \mathrm{Ad}_{W_\omega}}_-{\textstyle\cong} \\
        C^*_u(X) \ar[r]^-{\textstyle \Phi_\omega} & \B(\ell^2(X(\omega))).}
\end{displaymath}
\end{Lemma}

The above lemma suggests us to consider the following operator fibre space over $\partial X$:
$$\widehat{E}^\partial:=\bigsqcup_{\omega \in \partial X}\B(\ell^2(X(\omega))),$$
with the topology defined as follows: a net $\{T_{\omega_i}\}_{i \in I}$ converges to $T_\omega$ \emph{if and only if} $\omega_i \to \omega$, and for any $\alpha'_i \to \alpha'$, $\alpha''_i \to \alpha''$ with $\alpha'_i,\alpha''_i \in X(\omega_i)$ and $\alpha', \alpha'' \in X(\omega)$, we have
$$\langle T_{\omega_i}\delta_{\alpha'_i}, \delta_{\alpha''_i} \rangle \to \langle T_\omega\delta_{\alpha'}, \delta_{\alpha''} \rangle.$$
Applying the unitary operators $W_\omega$ in Lemma \ref{metric limit op} together, we obtain an isomorphism between the operator fibre spaces as follows:
$$\mathrm{Ad}_{W}=\bigsqcup_{\omega \in \partial X}\mathrm{Ad}_{W_{\omega}}: E^{\partial}=\bigsqcup_{\omega \in \partial X}\B(\ell^2(G(X)_\omega)) \longrightarrow \widehat{E}^\partial =\bigsqcup_{\omega \in \partial X}\B(\ell^2(X(\omega))).$$
Denote the subspace
$$\widehat{E}^\partial_u:=\bigsqcup_{\omega \in \partial X}C^*_u(X(\omega)).$$
A continuous section $\varphi$ of $\widehat{E}^\partial$ is called \emph{coarsely constant} (c.c.) if for any $\omega \in \partial X$ and $\alpha \in X(\omega)$, we have $\varphi(\alpha)=\varphi(\omega)$. Denote the set of all the bounded coarsely constant sections of $\widehat{E}^\partial_u$ by $\Gamma_b(\widehat{E}^\partial_u)^{\mathrm{c.c.}}$.

\begin{Lemma}
The isomorphism $\mathrm{Ad}_{W}$ induces a $C^*$-isomorphism:
$$(\mathrm{Ad}_{W})_*: \Gamma_b(E^{\partial}_u)^{G_\infty(X)} \stackrel{\textstyle \cong}{\longrightarrow} \Gamma_b(\widehat{E}^\partial_u)^{\mathrm{c.c.}}.$$
\end{Lemma}

\begin{proof}
For a bounded section $\varphi$ of $E^\partial$, it is $G_\infty(X)$-equivariant if and only if for any $(\alpha,\omega) \in G_\infty(X)$, we have
$$\varphi(\alpha)=R_{(\alpha,\omega)}^* \varphi(\omega) R_{(\alpha,\omega)}.$$
This is equivalent to the condition that for any $\omega \in \partial X$ and $\alpha \in X(\omega)$, we have
\begin{equation}\label{loc const sec}
\mathrm{Ad}_{W_\alpha}(\varphi(\alpha)) = \mathrm{Ad}_{W_\alpha}\big( R_{(\alpha,\omega)}^* \varphi(\omega) R_{(\alpha,\omega)} \big).
\end{equation}
Note that $X(\alpha) = X(\omega)$ and the following diagram commutes:
\begin{displaymath}
    \xymatrix@=4em{
        \ell^2(G(X)_\alpha) \ar[r]^-{R_{(\alpha,\omega)}} \ar[d]_-{W_\alpha}^-{\textstyle\cong} & \ell^2(G(X)_\omega)) \ar[d]^-{W_\omega}_-{\textstyle\cong} \\
        \ell^2(X(\alpha)) \ar[r]^-{\mathrm{Id}} & \ell^2(X(\omega)).}
\end{displaymath}
Hence (\ref{loc const sec}) is equivalent to $\big((\mathrm{Ad}_{W})_*\varphi\big)(\alpha) = \big((\mathrm{Ad}_{W})_*\varphi\big)(\omega)$, i.e., $(\mathrm{Ad}_{W})_*\varphi$ is coarsely constant.

On the other hand, given a point $\omega \in \partial X$, we claim $\mathrm{Ad}_{W_\omega}(\mathbb{C}[G(X)_\omega]) = \mathbb{C}[X(\omega)]$. By definition, an operator $T=(T_{(\alpha',\omega),(\alpha'',\omega)}) \in \ell^2(G(X)_\omega)$ belongs to $\mathbb{C}[G(X)_\omega]$ if and only if there exists a compact set $K \subseteq G(X)$ such that $T_{(\alpha',\omega),(\alpha'',\omega)} \neq 0$ implies that $(\alpha',\alpha'') \in K$. By definition, there exists some $r>0$ such that $K \subseteq \overline{E_r}$. Therefore, the above is equivalent to the condition that there exists some $r>0$ such that $(\alpha',\alpha'') \in \overline{E_r}$, which is equivalent to $d_\omega(\alpha',\alpha'') \leq r$. In other words, this is equivalent to $\mathrm{Ad}_{W_\omega}(T) \in \mathbb{C}[X(\omega)]$. Taking completion, we have $\mathrm{Ad}_{W_\omega}(C^*_u(G(X)_\omega)) = C^*_u(X(\omega))$.
\end{proof}

Now we define the following morphism (also called the \emph{symbol morphism})
$$\Phi:=(\mathrm{Ad}_{W})_* \circ \varsigma \circ \Theta^{-1}: C^*_u(X) \longrightarrow \Gamma_b(\widehat{E}^\partial_u)^{\mathrm{c.c.}}.$$
In other words, the following diagram commutes:
\begin{displaymath}
     \xymatrix@=4em{
        C^*_r(G(X)) \ar[r]^-{\textstyle \varsigma} \ar[d]_-{\textstyle \Theta}^-{\textstyle \cong} & \Gamma_b(E^{\partial}_u)^{G(X)} \ar[d]^-{\textstyle (\mathrm{Ad}_W)_* }_-{\textstyle \cong} \\
        C^*_u(X) \ar[r]^-{\textstyle \Phi} & \Gamma_b(\widehat{E}^\partial_u)^{\mathrm{c.c.}}.}
\end{displaymath}
And Lemma \ref{metric limit op} implies that $\Phi(T)(\omega)=\Phi_\omega(T)$ for any $T \in C^*_u(X)$ and $\omega \in \partial X$. Consequently, combining the above analysis we recover the limit operator theory for metric spaces in the Hilbert space case with complex number values as a corollary of Theorem \ref{main theorem}:
\begin{Corollary}
Let $X$ be a uniformly discrete metric space with bounded geometry and Property A. For any operator $T \in C^*_u(X)$, the following are equivalent:
\begin{enumerate}
  \item $T$ is Fredholm.
  \item $\Phi(T)$ is invertible in $\Gamma_b(\widehat{E}^\partial_u)^{\mathrm{c.c.}}$.
  \item For each $\omega \in \partial X$, the limit operator $\Phi_\omega(T)$ is invertible, and
  $$\sup_{\omega \in \partial X} \|\Phi_\omega(T)^{-1}\| < \infty.$$
  \item For each $\omega \in \partial X$, the limit operator $\Phi_\omega(T)$ is invertible.
\end{enumerate}
\end{Corollary}

\begin{Example}
Let us recall an interesting example from \cite[Example 4.13]{Spakula-Willett--limitoperators}. Take $X=\mathbb{N}$ with the usual metric, and identify the Hilbert space $\ell^2(\mathbb{N})$ with the Hardy space $H^2$ of the disk. All limit spaces are isometric to $\mathbb{Z}$ with the usual metric. Let $T_f$ be the Toeplitz operator on $H^2$ with continuous symbol $f: S^1 \to \mathbb{C}$. Then all limit operators of $T_f$ correspond to the symbol $f$, acting as the multiplication operator on $L^2(S^1) \cong \ell^2(\mathbb{Z})$ via Fourier transformation. Hence we recover the classical theory in Toeplitz operators that $T_f$ is Fredholm \emph{if and only if} $f$ is an invertible function on $S^1$.
\end{Example}

\subsection{Group action case}\label{group action ex}
In Section \ref{grp cptf ex}, we studied the limit operator theory for groups acting on their equivariant compactifications. Now we move on to general group actions and establish the associated limit operator theory.

Let $G$ be a countable exact discrete group acting on a compact space $X$ by homeomorphisms. Let $Y$ be a $G$-invariant open dense subset in $X$, and denote $\partial Y:=X \setminus Y$ its boundary. Clearly, both $Y$ and $\partial Y$ are locally compact and $G$ acts on them by homeomorphisms. We consider the associated transformation groupoid $\G=X \rtimes G$. 

As we pointed out in Section \ref{grp cptf ex}, although $G$ is exact, the action on $X$ is in general \emph{not} amenable. However we still have the following, and proof is the same as Lemma \ref{exactness lemma} hence omitted.
\begin{Lemma}
Let $G$ be a countable exact discrete  group acting on a compact space $X$, and $Y$ be an invariant open dense subset in $X$. Then the following short sequence is exact:
$$0 \longrightarrow C^*_r(Y \rtimes G) \longrightarrow C^*_r(X \rtimes G) \longrightarrow C^*_r(\partial Y \rtimes G) \longrightarrow 0.$$
\end{Lemma}

Therefore by Remark \ref{main theorem remark}, part of our main result: ``(1) $\Leftrightarrow$ (2) $\Leftrightarrow$ (3)" in Theorem \ref{main theorem} holds directly for the transformation groupoid $\G=X \rtimes G$, and they are also equivalent to condition (4) if the action is additionally assumed to be amenable. This establishes the limit operator theory for group actions. Now we would like to provide a more detailed and practical picture.

First we focus on the groupoid $C^*$-algebra $C^*_r(X \rtimes G)$, which is isomorphic to the reduced cross product $C(X)\rtimes_r G$. To make it more precise, let us recall the definition of reduced cross product (see \cite[Chapter 4.1]{brown-ozawa}). Denote $C_c(G,C(X))$ the linear space of finitely supported functions on $G$ with values in $C(X)$, equipped with the twisted convolution product and the $\ast$-operation. The reduced cross product $C(X)\rtimes_r G$ is the norm closure with respect to a regular representation of $C_c(G,C(X))$ (see \cite[Definition 4.1.4]{brown-ozawa}). We have the following algebraic $\ast$-isomorphism:
$$\theta: C_c(X \rtimes G) \longrightarrow C_c(G,C(X))$$
given by
$$F \mapsto \sum_{\gamma \in G} f_\gamma \gamma, \quad \mbox{where~} f_\gamma \in C(X) \mbox{~is~defined~by~} f_\gamma(x)=F(x,\gamma).$$
It is known that $\theta$ can be extended to be a $C^*$-isomorphism:
\begin{equation}\label{iso reduced product}
\Theta: C^*_r(X \rtimes G) \longrightarrow C(X) \rtimes_r G,
\end{equation}
which maps $C^*_r(Y \rtimes G)$ to $C_0(Y) \rtimes_r G$.

Now we move on to the discussion of limit operators. Recall that we have the symbol morphism associated to $\G=X \rtimes G$ from Definition \ref{limit morphism defn}:
$$\varsigma: C^*_r(X \rtimes G) \longrightarrow \Gamma_b(E^\partial_{u})^{\partial Y \rtimes G},$$
where $E^\partial_u$ is the Roe fibre space associated to the reduction groupoid $\G(\partial Y)=\partial Y \rtimes G$. Now we would like to offer a more detailed description for the symbol morphism. For each $\omega \in \partial Y$, the fibre
$$(X \rtimes G)_\omega=\{(\gamma \omega, \gamma): \gamma \in G\}$$
is bijective to $G$ via the map $(\gamma \omega, \gamma) \mapsto \gamma$, and induces a unitary:
$$V_\omega: \ell^2((X \rtimes G)_\omega) \cong \ell^2(G).$$
These isomorphisms can be put together to provide a map between operator fibre spaces as follows:
$$\mathrm{Ad}_{V}=\bigsqcup_{\omega \in \partial Y}\mathrm{Ad}_{V_{\omega}}: E^{\partial}=\bigsqcup_{\omega \in \partial Y}\B(\ell^2((X \rtimes G)_\omega)) \longrightarrow \partial Y \times \B(\ell^2(G)),$$
where the final item can be regarded as a trivial fibre space over $\partial Y$, and the topology coincides with the product topology. It is easy to check that $\mathrm{Ad}_{V}$ is a homeomorphism, hence an isomorphism between operator fibre spaces over $\partial Y$. Note that continuous sections of $\partial Y\times \B(\ell^2(G))$ coincides with $C_{\mathrm{WOT}}(\partial Y, \B(\ell^2(G)))$, hence $\mathrm{Ad}_{V}$ induces the following $C^*$-isomorphism:
$$(\mathrm{Ad}_{V})_*: \Gamma(E^\partial) \longrightarrow C_{\mathrm{WOT}}(\partial Y, \B(\ell^2(G))).$$
Note that elements in $C_{\mathrm{WOT}}(\partial Y, \B(\ell^2(G)))$ have uniform bounded norms, hence $\Gamma(E^\partial) = \Gamma_b(E^\partial)$. As in Section \ref{grp cptf ex}, we obtain an isomorphism:
$$(\mathrm{Ad}_{V})_*: \Gamma_b(E^\partial_{u})^{\partial Y \rtimes G} \stackrel{\textstyle \cong}{\longrightarrow} C_{\mathrm{WOT}}(\partial Y, C^*_u(G))^G.$$
Therefore, we define the following morphism (also called the \emph{symbol morphism}):
$$\sigma:=(\mathrm{Ad}_{V})_* \circ \varsigma \circ \Theta^{-1}: C(X) \rtimes_r G \longrightarrow C_{\mathrm{WOT}}(\partial Y, C^*_u(G))^G.$$
In other words, the following diagram commutes:
\begin{displaymath}
    \xymatrix{
        C^*_{r}(X \rtimes G) \ar[r]^-{\textstyle \varsigma} \ar[d]_-{\textstyle \Theta}^-{\textstyle \cong} & \Gamma_b(E^\partial_{u})^{Y \rtimes G} \ar[d]^-{\textstyle (\mathrm{Ad}_V)_* }_-{\textstyle \cong} \\
        C(X) \rtimes_r G \ar[r]^-{\textstyle \sigma} & C_{\mathrm{WOT}}(\partial Y, C^*_u(G))^G.}
\end{displaymath}

Now we would like to provide a more concrete formula for the above symbol morphism $\sigma$. Fix a boundary point $\omega \in \partial Y$ and consider the following representations:
$$M_\omega: C(X) \longrightarrow \B(\ell^2(G)) \quad \mbox{by}\quad M_\omega(f) \delta_\gamma:=f(\gamma\omega)\delta_\gamma,$$
and
$$\rho: G \longrightarrow \B(\ell^2(G)) \quad \mbox{by}\quad \rho(\gamma') \delta_\gamma:=\delta_{\gamma'\gamma}.$$
It is routine to check that $\rho(\gamma)^*M_\omega(f)\rho(\gamma)=M_\omega(\gamma^{-1}.f)$. Hence they induce a $\ast$-representation $$M_\omega \times \rho: C_c(G, C(X)) \longrightarrow \B(\ell^2(G)) \quad \mbox{by}\quad \sum_{\gamma \in G}f_\gamma \gamma \mapsto \sum_{\gamma \in G}M_\omega(f_\gamma) \rho(\gamma).$$

\begin{Lemma}
$M_\omega \times \rho$ induces a $C^*$-representation: $C(X) \rtimes_r G \longrightarrow \B(\ell^2(G))$ whose image is contained in the uniform Roe algebra $C^*_u(|G|)$, still denoted by $M_\omega \times \rho$.
\end{Lemma}

\begin{proof}
First note that the $\ast$-representation $M_\omega \times \rho$ extends naturally to the maximal crossed product $C(X) \rtimes_{\mathrm{max}} G$. So, when the action is amenable, we have $C(X) \rtimes_{\mathrm{max}} G = C(X) \rtimes_r G$ and the lemma holds.

To deal with the general case, note that the restriction map $C(X) \to C(\overline{G\omega})$ is $G$-equivariant and contractively completely positive (c.c.p), hence induces a c.c.p map
$$\varrho_\omega: C(X) \rtimes_r G \to C(\overline{G\omega}) \rtimes_r G.$$
Consider the faithful representation $M'_\omega: C(\overline{G\omega}) \longrightarrow \B(\ell^2(G))$ defined by $M'_\omega(f) \delta_\gamma:=f(\gamma\omega)\delta_\gamma$, which induces another representation
$$\pi_\omega: C(\overline{G\omega}) \longrightarrow \B(\ell^2(G) \otimes \ell^2(G))$$
defined by
$$\big(\pi_\omega(f)\big)(\delta_\alpha \otimes \delta_\gamma):=M'_\omega(\gamma^{-1}.f)(\delta_\alpha) \otimes \delta_\gamma=f(\gamma\alpha\omega)\delta_\alpha \otimes \delta_\gamma.$$
Hence we obtain a faithful covariant representation
$$\pi_\omega \times (1 \otimes \rho): C_c(G,C(\overline{G\omega})) \longrightarrow \B(\ell^2(G) \otimes \ell^2(G)),$$
and the reduced cross product $C(\overline{G\omega}) \rtimes_r G$ is the completion of $C_c(G,C(\overline{G\omega}))$ with respect to the norm of its image.

Now consider the unitary $U: \ell^2(G) \otimes \ell^2(G) \to \ell^2(G) \otimes \ell^2(G)$ defined by $U(\delta_\alpha \otimes \delta_\gamma) = \delta_\alpha \otimes \delta_{\gamma\alpha}$. It is routine to check that $\mathrm{Ad}_U \circ \pi_{\omega}(f)=\mathrm{Id} \otimes M'_\omega(f)$ for any $f \in C(\overline{G\omega})$, and $\mathrm{Ad}_U \circ (\mathrm{Id}\otimes \rho)(\gamma)=(\mathrm{Id}\otimes \rho)(\gamma)$ for any $\gamma \in G$. Hence, we have
$$\mathrm{Ad}_U \circ \big(\pi_\omega \times (1 \otimes \rho)\big): C(\overline{G\omega}) \rtimes_r G \longrightarrow \mathbb{C}\cdot\mathrm{Id} \otimes C^*(\{M'_\omega(f), \rho_\gamma: f \in C(\overline{G\omega}), \gamma \in G\}) \subseteq \mathbb{C}\cdot\mathrm{Id} \otimes C^*_u(|G|) \cong C^*_u(|G|).$$
In conclusion, we obtain a $C^*$-representation:
$$\mathrm{Ad}_U \circ \big(\pi_\omega \times (1 \otimes \rho)\big) \circ \varrho_\omega: C(X) \rtimes_r G \longrightarrow C^*_u(|G|)$$
satisfying the requirements.
\end{proof}

\begin{Definition}
Notations as above. For each $\omega \in \partial Y$ and $T \in C(X) \rtimes_r G$, we define the \emph{limit operator of $T$ at $\omega$} to be $\sigma_\omega(T):=(M_\omega \times \rho)(T) \in C^*_u(|G|)$.
\end{Definition}

\begin{Lemma}
For $\omega \in \partial Y$, we have $\sigma(T)(\omega)=\sigma_\omega(T)$, i.e., the following diagram commutes:
\begin{displaymath}
    \xymatrix@=3em{
        C^*_r(X \rtimes G) \ar[r]^-{\textstyle \lambda_\omega} \ar[d]_-{\textstyle\Theta}^-{\textstyle\cong} & C^*_u((X \rtimes G)_\omega) \ar[d]^-{\textstyle \mathrm{Ad}_{V_\omega}}_-{\textstyle\cong} \\
        C(X) \rtimes_r G \ar[r]^-{\textstyle \sigma_\omega} & C^*_u(|G|).}
\end{displaymath}
\end{Lemma}

\begin{proof}
It suffices to show that for $f\in C_c(X \rtimes G)$, we have $\mathrm{Ad}_{U_\omega} \circ \lambda_\omega (f)=\sigma_\omega\circ\Theta (f)$. By definition, we may assume that $\Theta(f)=\sum_{\gamma \in G}f_\gamma \gamma$. For any $g,h\in G$, we have
$$\big(\lambda_\omega (f)\delta_{(h\omega,h)}\big)(g\omega,g)=\sum_{\gamma\in G}f((g\omega,g)(\gamma\omega,\gamma)^{-1})\delta_{(h\omega,h)}(\gamma\omega,\gamma)=f(g\omega,gh^{-1})=f_{gh^{-1}}(g\omega).$$
Hence,
\begin{eqnarray*}
\langle \mathrm{Ad}_{V_\omega} \circ \lambda_\omega (f)\delta_h, \delta_g\rangle &=& \langle \lambda_\omega (f)\delta_{(h\omega,h)}, \delta_{(g\omega,g)}\rangle = f(g\omega,gh^{-1})=\lim_{\alpha \to \omega}f(g\alpha,gh^{-1})=f_{gh^{-1}}(g\omega),\\
\langle \sigma_\omega\circ\Theta (f)\delta_h, \delta_g\rangle &=& \langle \sigma_\omega(\sum_{\gamma \in G}f_\gamma \gamma)\delta_h, \delta_g\rangle = \sum_{\gamma \in G} \langle M_\omega(f_{\gamma})\rho(\gamma)\delta_h, \delta_g\rangle = \sum_{\gamma \in G} \langle f_\gamma(\gamma h \omega) \delta_{\gamma h}, \delta_g\rangle\\
&=& f_{gh^{-1}}(g\omega).
\end{eqnarray*}
So we finish the proof.
\end{proof}

Combining the above analysis, we establish the limit operator theory for group actions as a corollary of Theorem \ref{main theorem}:
\begin{Corollary}\label{cor: group action}
Let $G$ be a countable exact discrete  group acting on a compact space $X$, $Y$ be a $G$-invariant open dense subset in $X$ and $\partial Y=X \setminus Y$. For any $T \in C(X) \rtimes_r G$, the following are equivalent:
\begin{enumerate}
  \item $T$ is invertible modulo $C_0(Y) \rtimes_r G$.
  \item $\sigma(T)$ is invertible in $C_{\mathrm{WOT}}(\partial Y, C^*_u(G))^G$.
  \item For each $\omega \in \partial Y$, the limit operator $\sigma_\omega(T)$ is invertible, and
  $$\sup_{\omega \in \partial Y} \|\sigma_\omega(T)^{-1}\| < \infty.$$
\end{enumerate}
Furthermore if the action of $G$ on $X$ is amenable, then the above are also equivalent to the following:
\begin{enumerate}
  \item[(4)] For each $\omega \in \partial Y$, the limit operator $\sigma_\omega(T)$ is invertible.
\end{enumerate}
\end{Corollary}


\subsection{Groupoid case}\label{groupoid case ex}
Recall that in Section \ref{group case}, we studied the classic limit operator theory for exact discrete groups, characterising the Fredholmness of operators in the uniform Roe algebra of the group in terms of invertibilities of their limit operators. Now we would like to study the uniform Roe algebra of a general groupoid, and establish the associated limit operator theory.

Let $\G$ be a locally compact, second countable and \'{e}tale groupoid with unit space $\Gz$. Suppose $\G$ is strongly amenable at infinity (see Definition \ref{strongly amenable at infinity}), or weakly inner amenable and $C^*$-exact (see Definition \ref{defn:C*-exact} and \ref{defn:weakly inner amenability}). Recall from Section \ref{groupoid action preliminary subsection}, the left action of $\G$ on itself can be extended to its fibrewise Stone-\v{C}ech compactification $\beta_r \G$. Hence we may form the semi-direct product groupoid $\beta_r \G \rtimes \G$, whose unit space $\beta_r \G$ has a natural decomposition:
\begin{equation}\label{dec for groupoid case}
\beta_r \G=\G \sqcup \partial_r \G,
\end{equation}
where $\G$ is open and dense in $\beta_r \G$, and $\partial_r \G:= \beta_r \G \setminus \G$ is the fibrewise Stone-\v{C}ech boundary of $\G$. 

Since $\G$ is strongly amenable at infinity, or weakly inner amenable and $C^*$-exact, the groupoid $\beta_r \G \rtimes \G$ is topologically amenable by Proposition \ref{exact prop 1} or \ref{exact prop 2}. Hence we may apply Theorem \ref{main theorem} to the groupoid $\beta_r \G \rtimes \G$ together with the decomposition (\ref{dec for groupoid case}), and obtain the associated limit operator theory for $\beta_r \G \rtimes \G$. Now we would like to provide a more detailed picture.

First we focus on the groupoid $C^*$-algebra $C^*_r(\beta_r \G \rtimes \G)$, which is isomorphic to the uniform Roe algebra $C^*_u(\G)$ of $\G$ introduced in \cite{Ananth-Delaroche--ExactGroupoids}. To recall the definition and the isomorphism, denote
$$\G \ast_s \G:=\{(\gamma,\gamma_1) \in \G \times \G: s(\gamma)=s(\gamma_1)\}.$$
A \emph{tube} is a subset of $\G \ast_s \G$ whose image by the map $(\gamma,\gamma_1) \mapsto \gamma\gamma^{-1}_1$ is relatively compact in $\G$. We denote by $C_t(\G \ast_s \G)$ the space of continuous bounded functions on $\G \ast_s \G$ with support in a tube. Recall that $L^2(\G)$ is the Hilbert $C_0(\Gz)$-module (see Section \ref{groupoid C algebra prlm}), and we consider the map $\Phi: C_t(\G \ast_s \G) \to \B(L^2(\G))$ defined by
$$((\Phi f)\xi)(\gamma):=\sum_{\alpha\in \G_{s(\gamma)}} f(\gamma, \alpha) \xi(\alpha), \quad \mbox{where~} \xi \in L^2(\G) \mbox{~and~}\gamma \in \G.$$
Clearly $\Phi$ is injective so we may pull back the $\ast$-algebraic structure of $\B(L^2(\G))$ to $C_t(\G \ast_s \G)$. Intuitively, this is nothing but the ``$(\G \times \G)$-matrix structure". Restricted on each fibre: for any $x\in \Gz$, we have the induced $\ast$-representation $\Phi_x:C_t(\G \ast_s \G) \to \B(\ell^2(\G_x))$ by
$$((\Phi_x f)\eta)(\gamma):=\sum_{\alpha\in \G_x} f(\gamma, \alpha) \eta(\alpha), \quad \mbox{where~} \eta \in \ell^2(\G_x) \mbox{~and~}\gamma \in \G_x.$$

\begin{Definition}[Definition 5.1, \cite{Ananth-Delaroche--ExactGroupoids}]
The \emph{uniform Roe algebra of a groupoid $\G$} is the $C^*$-completion of $C_t(\G \ast_s \G)$ with respect to the operator norm of $\Phi(C_t(\G \ast_s \G))$ in $\B(L^2(\G))$. Hence $\Phi$ extends to a faithful representation from $C^*_u(\G)$ to $\B(L^2(\G))$, still denoted by $\Phi$. For each $x\in \Gz$, $\Phi_x$ extends to a $\ast$-homomorphism from $C^*_u(\G)$ to the classic uniform Roe algebra $C^*_u(\G_x)$, still denoted by $\Phi_x: C^*_u(\G) \to C^*_u(\G_x)$.
\end{Definition}

Now we consider the following algebraic $\ast$-isomorphism $\theta: C_c(\beta_r \G \rtimes \G) \to C_t(\G \ast_s \G)$ defined by
$$\theta(f)(\gamma,\gamma_1):=f(\gamma,\gamma\gamma_1^{-1})$$
for $f \in C_c(\beta_r \G \rtimes \G)$.
It is proved in \cite[Theorem 5.3]{Ananth-Delaroche--ExactGroupoids} that $\theta$ extends to a $C^*$-isomorphism
$$\Theta: C^*_r(\beta_r \G \rtimes \G) \stackrel{\displaystyle \cong}{\longrightarrow} C^*_u(\G).$$
What does the image $\Theta(C^*_r(\G \rtimes \G))$ look like in the uniform Roe algebra $C^*_u(\G)$? Denote $\K(L^2(\G))$ the algebra of compact operators on $L^2(\G)$ (recall that an adjointable operator is said to be \emph{compact} if it can approximated by finite-rank operators), and we have the following lemma:

\begin{Lemma}
Notations as above. We have that $\Phi\circ \Theta(C^*_r(\G \rtimes \G))= \K(L^2(\G))$. Consequently, $T \in C^*_r(\beta_r\G \rtimes \G)$ is invertible modulo $C^*_r(\G \rtimes \G)$ \emph{if and only if} $\Phi\circ \Theta (T)$ is invertible modulo $\K(L^2(\G))$.
\end{Lemma}

\begin{proof}
The proof is almost the same as the group case, but we need to modify slightly since we are working in Hilbert modules. By definition, $\Theta(C^*_r(\G \rtimes \G))$ is the norm closure of $\theta(C_c(\G \rtimes \G))=C_c(\G \ast_s \G)$, where the norm of $f\in C_c(\G \ast_s \G)$ is the operator norm of $\Phi f$ acting on $L^2(\G)$ defined by
$$((\Phi f)\zeta)(\gamma):=\sum_{\alpha\in \G_{s(\gamma)}} f(\gamma, \alpha) \zeta(\alpha)$$

On the other hand, for any $\xi,\eta \in L^2(\G)$, denote by $T_{\xi,\eta}$ the rank one operator $\zeta \mapsto \xi \langle \eta,\zeta \rangle$. Hence every finite rank operator has the following form:
$$\big( \sum_{i=1}^N T_{\xi_i,\eta_i} \big) (\zeta) (\gamma) = \sum_{\alpha \in \G_{s(\gamma)}} \big( \sum_{i=1}^N \xi_i(\gamma) \overline{\eta_i(\alpha)} \big) \zeta(\alpha),$$
for some $\xi_i,\eta_i \in L^2(\G); i=1,\ldots,N$. Therefore, $\sum_{i=1}^N T_{\xi_i,\eta_i} = \Phi(g)$ for $g \in C_c(\G \ast_s \G)$ defined by $g(\gamma,\alpha):=\sum_{i=1}^N \xi_i(\gamma) \overline{\eta_i(\alpha)}$.

Now we only need to show that for a given $f \in C_c(\G \ast_s \G)$, the operator $\Phi(f)$ can be approximated by a sequence of finite rank operators. Since $f$ has compact support, there exists $M \in \mathbb{N}$ such that
$$\sup_{x\in \Gz} |\{\gamma\in \G_x: \exists \alpha \in \G_x \mbox{~such~that~}f(\gamma,\alpha) \neq 0\}| \leq M,$$
and
$$\sup_{x\in \Gz} |\{\alpha\in \G_x: \exists \gamma \in \G_x \mbox{~such~that~}f(\gamma,\alpha) \neq 0\}| \leq M.$$
Recall that for two locally compact Hausdorff topological spaces $X$ and $Y$, we have that $C_0(X \times Y) \cong C_0(X) \otimes C_0(Y)$. The same argument can be used to show that for any $\varepsilon>0$, there exist $\xi_1,\ldots, \xi_N$ and $\eta_1, \ldots, \eta_N$ in $C_c(\G)$ such that
$$\sup_{(\gamma,\alpha) \in \G \ast_s \G} |f(\gamma,\alpha) - \sum_{i=1}^N \xi_i(\gamma) \overline{\eta_i(\alpha)}| \leq \frac{\varepsilon}{M \|f\|_\infty}.$$
Consequently, we have that
$$\big\|\Phi f - \sum_{i=1}^N T_{\xi_i,\eta_i}\big\| \leq \varepsilon.$$
Therefore, $\Phi f \in \K(L^2(\G))$ and we finish the proof.
\end{proof}

Now we move on to the discussion of limit operators. Recall that we have the symbol morphism associated to $\beta_r \G \rtimes \G$ from Definition \ref{limit morphism defn}:
$$\varsigma: C^*_r(\beta_r \G \rtimes \G) \longrightarrow \Gamma_b(E^\partial_{u})^{\partial_r \G \rtimes \G},$$
where $E^\partial_u$ is the Roe fibre space associated to the reduction $\partial_r \G \rtimes \G$. We would like to offer a more detailed description for $\varsigma$. Given a boundary point $\omega \in \partial_r \G= \beta_r \G \setminus \G$, note that
$$(\beta_r \G \rtimes \G)_\omega=\{(\gamma \omega,\gamma): s(\gamma) = r_\beta(\omega)\}$$
is bijective to $\G_{r_\beta(\omega)}$ via the map $\phi_\omega:(\gamma \omega,\gamma) \mapsto \gamma$, and induces a unitary
$$U_\omega: \ell^2((\beta_r \G \rtimes \G)_\omega) \cong \ell^2(\G_{r_\beta(\omega)}).$$
Hence we obtain an isomorphism:
\begin{equation}\label{Eq: Ad groupoid case}
\mathrm{Ad}_{U_\omega}: \B(\ell^2((\beta_r \G \rtimes \G)_\omega)) \cong \B(\ell^2(\G_{r_\beta(\omega)})).
\end{equation}
This isomorphism suggests us to consider the following operator fibre space
$$E^{\partial}_\beta:= \bigsqcup_{\omega \in \partial_r \G} \B(\ell^2(\G_{r_\beta(\omega)})),$$
with the topology defined as follows: a net $\{T_{\omega_i}\}_{i \in I}$ converges to $T_\omega$ \emph{if and only if} $\omega_i \to \omega$, and for any $\gamma'_i \to \gamma'$, $\gamma''_i \to \gamma''$ in $\G$ with $s(\gamma'_i)=r_\beta(\omega_i)=s(\gamma''_i)$ (which implies that $s(\gamma')=r_\beta(\omega)=s(\gamma'')$), we have
$$\langle T_{\omega_i}\delta_{\gamma'_i}, \delta_{\gamma''_i} \rangle \to \langle T_\omega\delta_{\gamma'}, \delta_{\gamma''} \rangle.$$
Combining the unitary operators $\mathrm{Ad}_{U_\omega}$ in (\ref{Eq: Ad groupoid case}) together, we obtain an isomorphism between the operator fibre spaces as follows:
$$\mathrm{Ad}_{U}=\bigsqcup_{\omega \in \partial_r \G}\mathrm{Ad}_{U_{\omega}}: E^{\partial}=\bigsqcup_{\omega \in \partial_r \G}\B(\ell^2((\beta_r \G \rtimes \G)_\omega)) \longrightarrow E^{\partial}_\beta= \bigsqcup_{\omega \in \partial_r \G} \B(\ell^2(\G_{r_\beta(\omega)})).$$
Denote the sub-fibre space
$$E^\partial_{\beta,u}:=\bigsqcup_{\omega \in \partial_r \G} C^*_u(\G_{r_\beta(\omega)}),$$
where $C^*_u(\G_{r_\beta(\omega)})$ is the uniform Roe algebra defined in Definition \ref{defn: uniform roe algebra fibre}. A section $\widetilde{\varphi}$ of $E^\partial_\beta$ is called \emph{$\G$-equivariant} if for any $\omega \in \partial_r \G$ and $\gamma \in \G$ with $s(\gamma)=r_\beta(\omega)$, we have $\widetilde{\varphi}(\gamma\omega) = R_\gamma^* \widetilde{\varphi}(\omega) R_\gamma$ where $R_\gamma: \ell^2(\G_{r(\gamma)}) \to \ell^2(\G_{s(\gamma)})$ is the right multiplication as in Definition \ref{defn: equivariant section}. Denote the $C^*$-algebra of bounded $\G$-equivariant sections of $E^\partial_{\beta,u}$ by $\Gamma_b(E^\partial_{\beta,u})^{\G}$. Then we have the following result, whose proof is similar to that of Lemma \ref{lem2 grp case}, hence omitted.

\begin{Lemma}
The isomorphism $\mathrm{Ad}_{U}$ induces a $C^*$-isomorphism:
$$(\mathrm{Ad}_{U})_*: \Gamma_b(E^{\partial}_u)^{\beta_r \G \rtimes \G} \stackrel{\textstyle \cong}{\longrightarrow} \Gamma_b(E^\partial_{\beta,u})^{\G}.$$
\end{Lemma}

Now we define the following morphism (also called the \emph{symbol morphism})
$$\sigma:=(\mathrm{Ad}_{U})_* \circ \varsigma \circ \Theta^{-1}: C^*_u(\G) \longrightarrow \Gamma_b(E^\partial_{\beta,u})^{\G}.$$
In other words, the following diagram commutes:
\begin{displaymath}
    \xymatrix{
        C^*_{r}(\beta_r \G \rtimes \G) \ar[r]^-{\textstyle \varsigma} \ar[d]_-{\textstyle \Theta}^-{\textstyle \cong} & \Gamma_b(E^\partial_{u})^{\beta_r \G \rtimes \G} \ar[d]^-{\textstyle (\mathrm{Ad}_U)_* }_-{\textstyle \cong} \\
        C^*_u(\G) \ar[r]^-{\textstyle \sigma} & \Gamma_b(E^\partial_{\beta,u})^{\G}.}
\end{displaymath}
We are looking for a more practical formula of $\sigma$. Fix an element $T \in C^*_u(\G)$ with $\|T\| \leq k$ for some $k \in \mathbb{N}$. Recall from Section \ref{op fibre sp}, the $k$-bounded operator fibre space associated to $\G$ is
$$E_k=\bigsqcup_{x\in \Gz} \B(\ell^2(\G_x))_k.$$
Consider the following map $\tau_T$ between the fibre space $(\G,r)$ and $(E_k,p)$ over $\Gz$:
$$\tau_T: \G \longrightarrow E_k,\quad \gamma \mapsto R_\gamma^* \Phi_{s(\gamma)}(T) R_\gamma$$
where
$$R_\gamma: \ell^2(\G_{r(\gamma)}) \longrightarrow \ell^2(\G_{s(\gamma)}), \quad \delta_\alpha \mapsto \delta_{\alpha\gamma}.$$
It is obvious that $r=p\circ \tau_T$ and $\tau_T$ is continuous, hence $\tau_T$ is a morphism of fibre spaces over $\Gz$. Since $\G$ is second countable, we know that $(E_k,p)$ is fibrewise compact by Lemma \ref{lemma: first countable implies fibrewise compact}. Hence by the universal property of $\beta_r \G$ (Proposition \ref{prop: universal property}), $\tau_T$ can be extended to a morphism:
$$\tau_T: \beta_r \G \longrightarrow E_k.$$

\begin{Definition}
Let $\G$ be as above, and $\omega\in \partial_\beta G$. We define \emph{the limit operator of $T$ at $\omega$} to be the operator $\tau_T(\omega)$ in $\B(\ell^2(\G_{r_\beta(\omega)}))$. Also define $\tau_\omega: C^*_u(\G) \longrightarrow \B(\ell^2(\G_{r_\beta(\omega)}))$ by $\tau_\omega(T):=\tau_T(\omega)$.
\end{Definition}

\begin{Lemma}\label{lem: groupoid limit op}
For $T \in C^*_u(\G)$ and $\omega \in \partial_\beta \G$, we have $\sigma(T)(\omega)=\tau_\omega(T)$.
\end{Lemma}

\begin{proof}
It suffices to prove the following diagram commutes:
\begin{displaymath}
    \xymatrix@=3em{
        C^*_r(\beta_r \G \rtimes \G) \ar[r]^-{\textstyle \lambda_\omega} \ar[d]_-{\textstyle\Theta}^-{\textstyle\cong} & \B(\ell^2((\beta_r \G \rtimes \G)_\omega))\ar[d]^-{\textstyle \mathrm{Ad}_{U_\omega}}_-{\textstyle\cong} \\
        C^*_u(\G) \ar[r]^-{\textstyle \tau_\omega} & \B(\ell^2(\G_{r_\beta(\omega)})) }
\end{displaymath}
By density, we only need to prove the lemma for $T\in C_t(\G \ast_s \G)$. Choose a net $\{\gamma_i\}_{i \in I}$ in $\G$ converging to $\omega$. Let $x=r_\beta(\omega)$ and $x_i=r(\gamma_i)$. Recall that $\theta: C_c(\beta_r \G \rtimes \G) \to C_t(\G \ast_s \G)$ is an isomorphism, so we may assume that $T=\theta(f)$ for some $f \in C_c(\beta_r \G \rtimes \G)$. For any $\alpha', \alpha'' \in \G_x$, by \'{e}taleness, we may choose $\{\alpha_i'\}_{i \in I}$ and $\{\alpha_i''\}_{i\in I}$ converging to $\alpha'$ and $\alpha''$, respectively, with $s(\alpha_i')=s(\alpha_i'')=x_i$. We have:
\begin{eqnarray*}
\langle \sigma(T)(\omega) \delta_{\alpha'}, \delta_{\alpha''} \rangle  &=& \big\langle \mathrm{Ad}_{U_\omega} \circ \lambda_\omega \circ \Theta^{-1} (\theta(f)) \delta_{\alpha'}, \delta_{\alpha''} \big \rangle = \big\langle \mathrm{Ad}_{U_\omega} \circ \lambda_\omega (f) \delta_{\alpha'}, \delta_{\alpha''} \big \rangle\\
&=& f(\alpha'' \omega, \alpha''{\alpha'}^{-1}) = \lim_i f(\alpha_i'' \gamma_i, \alpha_i''{\alpha_i'}^{-1}).
\end{eqnarray*}
On the other hand, we have
\begin{eqnarray*}
\langle\tau_\omega(T) \delta_{\alpha'}, \delta_{\alpha''} \rangle &=& \langle\tau_T(\omega) \delta_{\alpha'}, \delta_{\alpha''} \rangle = \lim_i \langle\tau_T(\gamma_i) \delta_{\alpha_i'}, \delta_{\alpha_i''} \rangle\\
&=&\lim_{i}  \langle R_{\gamma_i}^* \Phi_{s(\gamma_i)}(T) R_{\gamma_i} \delta_{\alpha_i'}, \delta_{\alpha_i''}\rangle = \lim_{i}  \langle \Phi_{s(\gamma_i)}(T)  \delta_{\alpha_i'\gamma_i}, \delta_{\alpha_i''\gamma_i}\rangle\\
&=&\lim_{i} (\theta f)(\alpha_i''\gamma_i, \alpha_i'\gamma_i) = \lim_i f(\alpha_i'' \gamma_i, \alpha_i''{\alpha_i'}^{-1}).
\end{eqnarray*}
Hence we finish the proof.
\end{proof}

Consequently, we obtain the associated limit operator theory as a corollary of Theorem \ref{main theorem}:
\begin{Corollary}\label{cor: groupoid case}
Let $\G$ be a locally compact, second countable and \'{e}tale groupoid. Suppose $\G$ is either strongly amenable at infinity, or weakly inner amenable and $C^*$-exact. Then for any $T$ in the uniform Roe algebra $C^*_u(\G)$, the following are equivalent:
\begin{enumerate}
  \item $\Phi(T)$ is invertible modulo $\K(L^2(\G))$.
  \item $\sigma(T)$ is invertible in $\Gamma_b(E^\partial_{\beta,u})^{\G}$.
  \item For each $\omega \in \partial_r \G$, the limit operator $\tau_\omega(T)$ is invertible, and
  $$\sup_{\omega \in \partial_r \G} \|\tau_\omega(T)^{-1}\| < \infty.$$
  \item For each $\omega \in \partial_r \G$, the limit operator $\tau_\omega(T)$ is invertible.
\end{enumerate}
\end{Corollary}

The equivalence between the first three conditions also holds for groupoids with a weaker property: \emph{KW-exactness} (for abbreviation of Kirchberg-Wassermann exactness). This notion is introduced in \cite{Ananth-Delaroche--ExactGroupoids}, and roughly speaking it means that the cross product of the groupoid preserves short equivariant exact sequences (see \cite[Definition 6.6]{Ananth-Delaroche--ExactGroupoids}). It is also shown in \cite{Ananth-Delaroche--ExactGroupoids} that strong amenability at infinity implies KW-exactness, and KW-exactness implies $C^*$-exactness (Definition \ref{defn:C*-exact}) for second countable locally compact groupoids with Haar system. Furthermore, KW-exactness and $C^*$-exactness coincide when the groupoid is further assumed to be weakly inner amenable and \'{e}tale, hence are also equivalent to strong amenability at infinity by Proposition \ref{exact prop 2}. The proof for the equivalence of the first three conditions in the case of KW-exactness follows similarly as we did in Section \ref{grp cptf ex}, based on Remark \ref{main theorem remark}. Consequently, we have an alternative version of Corollary \ref{cor: groupoid case} as follows:

\begin{subcor}
Let $\G$ be a locally compact,  second countable, \'{e}tale and KW-exact groupoid. Then for any $T$ in the uniform Roe algebra $C^*_u(\G)$, the following are equivalent:
\begin{enumerate}
  \item $\Phi(T)$ is invertible modulo $\K(L^2(\G))$.
  \item $\sigma(T)$ is invertible in $\Gamma_b(E^\partial_{\beta,u})^{\G}$.
  \item For each $\omega \in \partial_r \G$, the limit operator $\tau_\omega(T)$ is invertible, and
  $$\sup_{\omega \in \partial_r \G} \|\tau_\omega(T)^{-1}\| < \infty.$$
\end{enumerate}
Furthermore, if in addition $\G$ is weakly inner amenable, then the above are also equivalent to
\begin{enumerate}
  \item[(4)] For each $\omega \in \partial_r \G$, the limit operator $\tau_\omega(T)$ is invertible.
\end{enumerate}
\end{subcor}

Finally, we mention briefly some other equivariant fibrewise compactifications of groupoids and study the associated limit operator theory, as we did in Section \ref{grp cptf ex}. We only provide the sketch and omit the details, since the proofs are similar.

Let $\G$ be a locally compact, second countable and \'{e}tale groupoid. Let $(Y,p)$ be a $\G$-equivariant fibrewise compactification of the $\G$-space $(\G,r)$, and $\partial_Y \G:=Y \setminus \G$. By the universal property of $\beta_r \G$ (Proposition \ref{prop: universal property}), there exists a $\G$-equivariant surjective morphism $\phi: \beta_r \G \to Y$. Hence, we obtain an embedding $C_c(Y \rtimes \G) \to C_c(\beta_r \G \rtimes \G)$, which induces an injective $\ast$-homomorphism:
$$\phi^*: C^*_r(Y \rtimes \G) \to C^*_r(\beta_r \G \rtimes \G).$$
Denote $\mathcal{A}_Y$ the subalgebra $\Theta \circ \phi^* (C^*_r(Y \rtimes \G))$ in $C^*_u(\G)$. We may also define the associated fibre space $E^{\partial}_Y:= \bigsqcup_{\omega \in \partial_Y \G} \B(\ell^2(\G_{p(\omega)}))$, and the symbol morphism $\sigma_Y: \mathcal{A}_Y \to \Gamma_b(E^\partial_{Y,u})^{\G}$. Then we have:

\begin{Corollary}\label{groupoid action on fiber space}
Let $\G$ be a locally compact, second countable and \'{e}tale groupoid which is either strongly amenable at infinity, or weakly inner amenable and $C^*$-exact. Let $(Y,p)$ be a $\G$-equivariant fibrewise compactification of $(\G,r)$, $\phi: \beta_r \G \to Y$ be the induced surjection, and $Z$ be a subset in $\partial_r \G$ such that $\phi(Z)=\partial_Y \G:=Y \setminus \G$. Suppose $\mathcal{A}_Y$ is the associated $C^*$-subalgebra of $C^*_u(\G)$ as defined above. Then for any $T \in \mathcal{A}_Y$, the following are equivalent:
\begin{enumerate}
  \item $\Phi(T)$ is invertible modulo $\K(L^2(\G))$.
  \item $\sigma_Y(T)$ is invertible in $\Gamma_b(E^\partial_{Y,u})^{\G}$.
  \item For each $\omega \in Z$, the limit operator $\tau_\omega(T)$ is invertible, and
  $$\sup_{\omega \in Z} \|\tau_\omega(T)^{-1}\| < \infty.$$
  \item For each $\omega \in Z$, the limit operator $\tau_\omega(T)$ is invertible.
\end{enumerate}
\end{Corollary}

As noted before, the equivalence between the first three conditions also holds when $\G$ is KW-exact. Consequently, we have the following alternative version:

\begin{subcor}
Let $\G$ be a locally compact, second countable, \'{e}tale and KW-exact groupoid. Let $(Y,p)$ be a $\G$-equivariant fibrewise compactification of $(\G,r)$, $\phi: \beta_r \G \to Y$ be the induced surjection, and $Z$ be a subset in $\partial_r \G$ such that $\phi(Z)=\partial_Y \G:=Y \setminus \G$. Suppose $\mathcal{A}_Y$ is the associated $C^*$-subalgebra of $C^*_u(\G)$ as defined above. Then for any $T \in \mathcal{A}_Y$, the following are equivalent:
\begin{enumerate}
  \item $\Phi(T)$ is invertible modulo $\K(L^2(\G))$.
  \item $\sigma_Y(T)$ is invertible in $\Gamma_b(E^\partial_{Y,u})^{\G}$.
  \item For each $\omega \in Z$, the limit operator $\tau_\omega(T)$ is invertible, and
  $$\sup_{\omega \in Z} \|\tau_\omega(T)^{-1}\| < \infty.$$
\end{enumerate}
Furthermore, if in addition $\G$ is weakly inner amenable, then the above are also equivalent to:
\begin{enumerate}
  \item[(4)] For each $\omega \in Z$, the limit operator $\tau_\omega(T)$ is invertible.
\end{enumerate}
\end{subcor}

\begin{Remark}
Finally we remark that one may also study the limit operator theory for groupoid actions, which is analogous to the one studied in Section \ref{group action ex}. However, here we choose not to mention it any further since more preliminary notions would be required, and this would definitely make the current paper more complicated.
\end{Remark}

\bibliographystyle{amsalpha}
\bibliography{eagbib2}

\end{document}